\documentclass[11pt]{article}

\usepackage{amsmath,amsthm,amsfonts,amssymb,csquotes,titlesec,float,caption,wrapfig,mathrsfs}
\usepackage{graphicx}
\usepackage[american]{babel}
\usepackage[table]{xcolor}
\usepackage[hidelinks]{hyperref}
\usepackage{tikz-cd}
\usepackage{tikz}
\usepackage{verbatim}
\usepackage{accents}
\usepackage{marvosym}
\usepackage{pdflscape}
\usepackage[citestyle=alphabetic,bibstyle=alphabetic,backend=bibtex,maxbibnames=10,minbibnames=5]{biblatex}
\usepackage{enumerate}
\usepackage{thmtools}
\usepackage{stmaryrd}
\usepackage[left=1in,top=1in,right=1in]{geometry}
\usepackage{subcaption}
\usepackage{zref-clever}
\zcsetup{cap=true}
\renewcommand\thmcontinues[1]{Continued}
\usetikzlibrary{calc,decorations.pathreplacing,decorations.pathmorphing,shapes.misc,shapes.geometric}

\newtheorem{thm}{Theorem}[section]
\zcRefTypeSetup{thm}{
    Name-sg = Theorem,
    name-sg = Theorem,
    Name-pl = Theorems,
    name-pl = Theorems,
}

\newtheorem{lem}[thm]{Lemma}
\newtheorem{lemma}[thm]{Lemma}
\newtheorem{prop}[thm]{Proposition}
\newtheorem{qs}[thm]{Question}
\newtheorem{cor}[thm]{Corollary}
\newtheorem{conj}[thm]{Conjecture}

\theoremstyle{remark} 
\newtheorem{rem}[thm]{Remark}
\zcRefTypeSetup{rem}{
    Name-sg = Remark,
    name-sg = Remark,
    Name-pl = Remarks,
    name-pl = Remarks,
}
\zcRefTypeSetup{lem}{
    Name-sg = Lemma,
    name-sg = Lemma,
    Name-pl = Lemmas,
    name-pl = Lemmas,
}
\zcRefTypeSetup{lemma}{
    Name-sg = Lemma,
    name-sg = Lemma,
    Name-pl = Lemmas,
    name-pl = Lemmas,
}
\zcRefTypeSetup{prop}{
    Name-sg = Proposition,
    name-sg = Proposition,
    Name-pl = Propositions,
    name-pl = Propositions,
}
\zcRefTypeSetup{cor}{
    Name-sg = Corollary,
    name-sg = Corollary,
    Name-pl = Corollaries,
    name-pl = Corollaries,
}
\zcRefTypeSetup{conj}{
    Name-sg = Conjecture,
    name-sg = Conjecture,
    Name-pl = Conjectures,
    name-pl = Conjectures,
}
\zcRefTypeSetup{qs}{
    Name-sg = Question,
    name-sg = Question,
    Name-pl = Questions,
    name-pl = Questions,
}
\newtheorem{example}[thm]{Example}

\zcRefTypeSetup{example}{
    Name-sg = Example,
    name-sg = Example,
    Name-pl = Examples,
    name-pl = Examples,
}
\zcRefTypeSetup{notation}{
    Name-sg = Notation,
    name-sg = Notation,
    Name-pl = Notations,
    name-pl = Notations,
}

\theoremstyle{definition} 
\newtheorem{df}[thm]{Definition} 

\zcRefTypeSetup{df}{
    Name-sg = Definition,
    name-sg = Definition,
    Name-pl = Definitions,
    name-pl = Definitions,
}
\zcRefTypeSetup{proposal}{
    Name-sg = Proposal,
    name-sg = Proposal,
    Name-pl = Proposals,
    name-pl = Proposals,
}

\titleformat*{\section}{\normalsize \bfseries \filcenter}
\titleformat*{\subsection}{\normalsize \bfseries }
\newtheorem{mainthm}{Theorem}
\zcRefTypeSetup{mainthm}{
    Name-sg = Theorem,
    name-sg = Theorem,
    Name-pl = Theorems,
    name-pl = Theorems,
}
\newtheorem{maincor}[mainthm]{Corollary}
\zcRefTypeSetup{maincor}{
    Name-sg = Corollary,
    name-sg = Corollary,
    Name-pl = Corollaries,
    name-pl = Corollaries,
}

\zcRefTypeSetup{equation}{
    Name-sg =,
    name-sg =,
    Name-pl =,
    name-pl =,
    Name-sg-ab =,
    name-sg-ab =,
    Name-pl-ab =,
    name-pl-ab =,
}

\AddToHook{env/lem/begin}{\zcsetup{countertype={thm=lem}}}
\AddToHook{env/lemma/begin}{\zcsetup{countertype={thm=lemma}}}
\AddToHook{env/prop/begin}{\zcsetup{countertype={thm=prop}}}
\AddToHook{env/qs/begin}{\zcsetup{countertype={thm=qs}}}
\AddToHook{env/cor/begin}{\zcsetup{countertype={thm=cor}}}
\AddToHook{env/conj/begin}{\zcsetup{countertype={thm=conj}}}
\AddToHook{env/rem/begin}{\zcsetup{countertype={thm=rem}}}
\AddToHook{env/example/begin}{\zcsetup{countertype={thm=example}}}
\AddToHook{env/notation/begin}{\zcsetup{countertype={thm=notation}}}
\AddToHook{env/df/begin}{\zcsetup{countertype={thm=df}}}
\AddToHook{env/proposal/begin}{\zcsetup{countertype={thm=proposal}}}
\AddToHook{env/maincor/begin}{\zcsetup{countertype={mainthm=maincor}}}

\NewCommandCopy{\cref}{\zcref}
\NewDocumentCommand{\Cref}{s O{} m}{\IfBooleanTF{#1}
        {\zcref*[S,#2]{#3}}
        {\zcref[S,#2]{#3}}}

\makeatletter
\def\namedlabel#1#2{\begingroup
   \def\@currentlabel{#2}\label{#1}\endgroup
}
\makeatother

\DeclareFontFamily{U}{mathx}{}
\DeclareFontShape{U}{mathx}{m}{n}{<-> mathx10}{}
\DeclareSymbolFont{mathx}{U}{mathx}{m}{n}
\DeclareMathAccent{\widehat}{0}{mathx}{"70}
\DeclareMathAccent{\widecheck}{0}{mathx}{"71}

\renewbibmacro{in:}{}

\tikzset{
  fuzz/.style={
    postaction={draw,decorate,decoration={border,amplitude=0.1cm,angle=90,segment length=.1cm}},
  },
}

\title{\normalsize \textbf{Monodromy action of mirror stops for toric Calabi--Yau surfaces}}
\author{Michela Barbieri, Andrew Hanlon, Jeff Hicks}
\date{}
\hypersetup{
  pdftitle={Monodromy action of mirror stops for toric Calabi-Yau surfaces},
  pdfauthor={Michela Barbieri, Andrew Hanlon, Jeff Hicks}
}

\newcommand{\eps}{\varepsilon}

\newcommand{\RR}{\mathbb R}
\newcommand{\ZZ}{\mathbb Z}
\newcommand{\CC}{\mathbb C}
\renewcommand{\AA}{\mathbb A}
\newcommand{\NN}{\mathbb N}
\newcommand{\PP}{\mathbb P}
\newcommand{\LL}{\mathbb L}

\newcommand{\into}{\hookrightarrow}

\renewcommand{\Im}{\text{Im}}

\newcommand{\stp}{\Lambda}

\newcommand{\X}{\mathcal X}

\newcommand{\cW}{\mathcal W}

\newcommand{\Os}{\mathcal{O}}
\newcommand{\Es}{\mathcal{E}}
\newcommand{\Fs}{\mathcal{F}}
\newcommand{\Ms}{\mathcal{M}}

\newcommand{\Leg}{\Lambda}

\DeclareMathOperator{\id}{id}

\DeclareMathOperator{\cone}{cone}

\DeclareMathOperator{\Hom}{Hom}

\DeclareMathOperator{\Ext}{Ext}
\DeclareMathOperator{\Coh}{Coh}

\DeclareMathOperator{\Pic}{Pic}
\DeclareMathOperator{\st}{\; |\; }
\DeclareMathOperator{\Fuk}{Fuk}

\DeclareMathOperator{\Ob}{Ob}

\DeclareMathOperator{\Int}{Int}

\DeclareMathOperator{\Area}{Area}
\DeclareMathOperator{\Aut}{Aut}
\DeclareMathOperator{\Cent}{Cent}
\DeclareMathOperator{\Vor}{Vor}
\DeclareMathOperator{\Conf}{Conf}
\DeclareMathOperator{\uConf}{uConf}

\DeclareMathOperator{\Lag}{Lag}
\DeclareMathOperator{\Corr}{Corr}
\DeclareMathOperator{\Ham}{Ham}
\DeclareMathOperator{\Tot}{Tot}

\newcommand{\XA}{X}
\newcommand{\XB}{\check X}
\newcommand{\JA}{{J_X}}
\newcommand{\JB}{{J_{\check X}}}
\newcommand{\omegaA}{\omega_X}
\newcommand{\omegaB}{\omega_{\check X}}
\newcommand{\Db}{D^b}
\newcommand{\n}{{n-1}}
\newcommand{\nn}{n}
\newcommand{\FIPS}{\mathbf{FIPS}}
\newcommand{\Ret}{\mathrm{Ret}}

\newcommand{\tPsi}{\tilde\Psi}

\newcommand{\Addresses}{{
  \bigskip
  \footnotesize

  \noindent M.~Barbieri, \textsc{Department of Mathematics, University College London}\par\nopagebreak
  \noindent \textit{E-mail address}: \texttt{michela.barbieri.21@ucl.ac.uk}

  \medskip

  \noindent A.~Hanlon, \textsc{Department of Mathematics, University of Oregon}\par\nopagebreak
  \noindent \textit{E-mail address}: \texttt{ahanlon@uoregon.edu}

  \medskip

  \noindent J.~Hicks, \textsc{School of Mathematics and Statistics,  University of St Andrews}\par\nopagebreak
  \noindent \textit{E-mail address}: \texttt{jeff.hicks@st-andrews.ac.uk}

  \medskip

}}

\begin{document}
\maketitle
\begin{abstract}
Mirror symmetry predicts an action by the fundamental group of a conjectural stringy K\"ahler moduli space on the derived category of an algebraic variety.
For a toric variety, a model for this space is understood, but constructing the action is still an open problem in general.
We propose that this action can be studied on the $A$-side via a moduli space of Legendrians isotopic to the FLTZ Legendrian.
For the $A_{\n}$ singularity, we construct an annular braid-group action on the corresponding partially wrapped Fukaya category by exact autoequivalences.
The standard braid subgroup recovers the Seidel--Thomas action on the derived category, while the additional annular generator corresponds to tensor product with $\mathcal O(-1)$.

We additionally extend the Floer-theoretic approach to homological mirror symmetry for toric varieties to the setting of semiprojective toric Deligne--Mumford stacks over an arbitrary field. 
\end{abstract}

\section{Introduction}
Homological mirror symmetry is a predicted duality between certain symplectic and complex manifolds called mirror pairs.
Broadly, a mirror pair consists of a symplectic manifold $(\XA, \omegaA)$, a complex manifold $(\XB, \JB)$, and a dictionary between certain symplectic invariants of $\XA$ and algebro-geometric invariants of $\XB$. 
For historical reasons arising from string theory, the symplectic geometric data is called the ``$A$-side'' while the algebro-geometric data is called the ``$B$-side''. 
In many cases, mirror symmetry is an involution in the sense that both sides of the correspondence carry mutually mirror symplectic and complex structures.
\[
  \begin{tikzpicture}
    \draw[rounded corners, fill=red!20]   (-7.5,0.5) rectangle (4.5,-0.5);
    \draw[rounded corners, fill=blue!20]  (-7.5,-1) rectangle (4.5,-2);
    \node at (-2,1) {$(\XA, \omegaA, \JA)$};
    \node at (2.5,1) {$(\XB, \omegaB, \JB)$};
    \node (v1) at (-2,0) {$\Fuk(\XA, \omegaA)$};
    \node (v4) at (2.5,0) {$\Fuk(\XB, \omegaB)$};
    \node (v3) at (-2,-1.5) {$\Db(\XA, \JA)$};
    \node (v2) at (2.5,-1.5) {$\Db(\XB, \JB)$};
    \draw[<->, thick]  (v1) edge (v2);
    \draw[<->, dashed]  (v3) edge (v4);
    \node at (-5.5,0) {Symplectic Invariants};
    \node at (-5.5,-1.5) {Algebraic Invariants};
    \draw[dotted] (-3.5,1.5) rectangle (-0.5,-2.5);
    \draw[dotted]  (1,1.5) rectangle (4,-2.5);
  \end{tikzpicture}
\]
In this paper, we will focus on the correspondence $\Fuk(\XA, \omegaA)\leftrightarrow \Db(\XB, \JB)$. Here, $\Fuk(\XA,\omegaA)$ denotes a Fukaya-type $A_\infty$ category of Lagrangians in $\XA$ (wrapped/partially wrapped as appropriate), and $\Db(\XB,\JB)$ denotes the bounded derived category of coherent sheaves on $\XB$.
A further expectation of the theory is that this duality is defined over families of symplectic and complex manifolds so that we have an identification of the moduli spaces of structures on $\XA$ and $\XB$:
\[
    \XA\in \mathcal M_\omega\leftrightarrow \mathcal M_{J}\ni \XB.
\] 
This already indicates that the symplectic story needs to be expanded, as the moduli space of complex structures itself naturally carries a complex structure, while the moduli space of symplectic structures on $\XA$ is a real manifold locally identified with $H^2(\XA, \RR)$.
To account for this discrepancy, the space $\mathcal M_\omega$ should be replaced with the conjectural ``stringy K\"ahler moduli space'' (SKMS), which complexifies the moduli space of symplectic structures on $\XA$. In examples, we can define proxies for the SKMS (see \cref{subsec:Bside} for an overview).

While $\Fuk(\XA, \omegaA)$ depends only on the symplectic structure on $\XA$, a choice of complex structure $\JA$ on $\XA$ enhances $\Fuk(\XA, \omegaA)$ with additional structure. At the most concrete level, the definition of $\Fuk(\XA, \omegaA)$ at the chain level requires a choice of almost complex structure. The conjectured existence of a stability condition on $\Fuk(\XA, \omegaA)$ arising from a $\JA$-holomorphic volume form on $\XA$ is an algebraic shadow of this additional data. On the $B$-side, we have a similar appearance of stability conditions on categories, arising from a choice of K\"ahler form $\omegaB$ on $\XB$.
This points to a subtle interplay between the complex (resp. K\"ahler) moduli space and the Fukaya category of $\XA$ (resp. derived category of $\XB$). 
This leads to the following question.

\begin{qs} [\cite{segal2011equivalences,donovan2019stringy} and others; see also the survey \cite{spenko2023survey}] \label{qs:localsystemoverstringy}
   Does there exist a local system of Fukaya categories (resp. derived categories of coherent sheaves) over $\mathcal M^{\XA}_{J}$ (resp. $\mathcal M^{\XB}_\omega$)?
\end{qs}
The difficulty of the question is compounded by a lack of rigorous definition of $\mathcal M^{\XB}_\omega$ and technical limitations in the definition of $\Fuk(\XA, \omegaA)$. Even defining a local system of derived categories is challenging \cite{pascaleff2025higherlocalsystemscategorified}.
In practice, we can use heuristics from mirror symmetry to guess $\mathcal M^{\XA}_J$, and then look for interesting actions of $\pi_1(\mathcal M^{\XA}_J, (\XA,\JA))$ on $\Fuk(\XA,\omegaA)$ and/or $\Db(\XB, \JB)$ via derived autoequivalences (up to natural isomorphism).
This provides an algebraic description of a local system of categories.
The seminal example of this action is given in
\cite{seidel2001braid} when $\XB$ is the $A_{\n}$-singularity.
\citeauthor{seidel2001braid} produced novel autoequivalences on the $B$-side---called spherical twists---by exhibiting a configuration of Lagrangian spheres on an $A$-side mirror and studying the symplectic Dehn twists around these spheres.
The braiding of symplectic Dehn twists manifests itself as a braid relation among these autoequivalences of $\Db(\XB)$.

When $\XB$ is a toric variety, there are explicit descriptions of a space of Laurent polynomials whose fundamental group should act on the derived category.
We will describe the space briefly here and refer to \cref{subsubsec:FIPS} and the survey \cite{spenko2023survey} for more details.
Namely, the $B$-side incarnation of \cref{qs:localsystemoverstringy} can be interpreted as the following conjecture.

\begin{conj}[Toric HMS monodromy conjecture] \label{conj:toricmon} There is an action 
\begin{equation} \label{eq:toricmonconj} \pi_1 \left( \frac{\AA^d \setminus V(E_A)}{T} \right) \to \mathrm{Auteq}\left(\Db(\XB) \right) 
\end{equation}
where $d$ is the number of rays in the fan of $\XB$, $T$ is an algebraic torus of rank equal to $\dim(\XB)$, and $V(E_A)$ is a hypersurface cut out by the principal $A$-determinant of a matrix $A$ determined by the toric GIT problem. 
\end{conj}
However, we note that \cite[Remark 2.8]{donovan2015mixed} points out that the space on the left side of \eqref{eq:toricmonconj}, which we will denote by $\FIPS$ as in \eqref{def:fips}, should merely map to, but not coincide with, the SKMS.
In this paper, we propose that another stand-in for the toric SKMS is the moduli space of piecewise Legendrians isotopic to the FLTZ stop mirror to $\XB$ (see \cref{def:FLTZ}).
Although this space of Legendrians is much larger, we note that the fundamental group on the left side of \eqref{eq:toricmonconj} is already difficult to compute in general, which is part of what makes \cref{conj:toricmon} challenging (see \cite{michelathesis}), and we hope that the Legendrian moduli space at least has combinatorially describable generators.
The main contribution of this paper is to verify that proposal for the $A_{\n}$ singularity by exhibiting a family of Fukaya categories over such a moduli space of Legendrian stops naturally inducing a braid group action.

\subsection{Results} 
Let $\XB_{\nn}$ be the $A_{\nn-1}$ singularity, understood as the toric stack $[\CC^2/ \ZZ_\nn]$, where the action of a primitive $\nn$-th root of unity $\zeta$ is given by $(z_1, z_2) \mapsto (\zeta z_1, \zeta^{\nn-1} z_2)$. 
We study a moduli of stops version of categorical monodromy for this example.
On the $A$-side, we work with $\XA=T^*T^2$ and the FLTZ stop $\stp_\nn\subset Y=\partial_\infty \XA$.
We use the Floer-theoretic approach to toric HMS in \cref{app:HMS} to identify the resulting partially wrapped Fukaya category with $\Db(\XB_{\nn})$ (see \cref{subsec:Aside} for a discussion of other approaches to toric HMS).  
Rather than varying a Landau--Ginzburg potential, we vary the stop itself inside the component of the stop moduli space consisting of Legendrians isotopic to $\stp_\nn$.
We model the relevant part of this moduli space on  $\mathcal N$,  the space of embedded graphs with unit-area faces in the cylinder, which comes with a lifting map $\mathcal N$ to $\mathcal L_{\stp_\nn}^{emb}(Y)$ (see \cref{def:lemb}).

The first main result is an informal version of \cref{thm:monodromyAction} proved in \cref{sec:monodromyAction}.
\begin{mainthm}[Informal version of \cref{thm:monodromyAction}]
    \label{thm:main}
A smooth path $\phi^t$ in $\mathcal N$ determines an exact Lagrangian correspondence between the stopped cotangent bundles $(\XA,\Leg_{\phi^0})$ and $(\XA,\Leg_{\phi^1})$.
These correspondences depend only on the endpoint-fixed homotopy class of the path and are compatible with concatenation.
Consequently, $\pi_1(\mathcal N,\phi_\nn)$ acts on $H^0\mathcal W(\XA,\stp_\nn)$.
\end{mainthm}
The precise theorem constructs the underlying correspondences and proves the stop-avoidance and composition properties needed for this action.

Our second main result identifies the annular braid group inside this graph space.
\begin{mainthm}[\cref{cor:braidLegendrian}]
    \label{thm:braidInclusion}
There is an inclusion
\[
\nu:\mathcal B_{\nn,1}\into \pi_1(\mathcal N,\phi_\nn)
\]
induced by explicit annular braid loops in the graph model, and composing with $\Leg_{(-)}$ gives loops of Legendrian stops in $\mathcal L_{\stp_\nn}^{emb}(Y)$.
\end{mainthm}

The proof combines explicit braid loops with a weighted Voronoi construction, which assigns to a configuration of points a canonical area-one planar graph, and then carries these graphs to Legendrian stops by the lift construction.

Combining Theorems A and B with the HMS equivalence \cref{thm:hms} gives the following consequence.

\begin{maincor}
     The action of $\mathcal B_{\nn,1}$ on $H^0\mathcal W(\XA,\stp_\nn)$ induces an action of $\mathcal B_{\nn,1}$ on $\Db(\XB_{\nn})$ by exact autoequivalences. Its restriction to the subgroup $\mathcal B_{\nn}\subset \mathcal B_{\nn,1}$ recovers the Seidel--Thomas braid action, and the extra generator $\rho$ corresponding to annular rotation acts by tensoring with $\mathcal O(-1)$.
\end{maincor}

The identification of the braid generators with Seidel--Thomas twists is proved in \cref{sec:applications}. The description of $\rho$ is a generalization of the toric line bundle monodromy action in \cite{hanlon2019monodromy, hanlon2022aspects} to toric Deligne--Mumford stacks.
This generalization of the line bundle monodromy action, in turn, relies on extending the Floer-theoretic proof of toric HMS to Deligne--Mumford stacks, as outlined in \cref{app:HMS}.

We also use the same stop moduli picture to read off changes in the Bondal--Thomsen collection (\cref{subsec:BTcollection}), to describe equivalences between VGIT chambers (\cref{subsec:VGITEquivalence}), and to study partial resolutions of the $A_{\n}$ singularity (\cref{subsec:partialResolutions}).

\subsection{Related results}
\label{subsec:previousResults}
This paper adds to the considerable body of mathematics analyzing the mirror monodromy action.
Even the study of the toric case (\cref{conj:toricmon}) has received considerable attention for which we do not attempt to give a comprehensive overview.
We again note the braid group action in \cref{cor:action} was first constructed in \cite{seidel2001braid}.
The largest class of toric varieties for which \cref{conj:toricmon} is proven is the class of quasi-symmetric GIT quotients \cite{halpern2020combinatorial}.

When restricting to the context of stringy K\"{a}hler moduli space actions compatible with homological mirror symmetry for toric varieties, the following works are perhaps the most closely related to this paper.
\begin{itemize}
    \item The simplest part of the conjectural action \eqref{eq:toricmonconj} comes from a real torus in $\FIPS$.
    The action of the fundamental group of this real torus is mirror to tensoring by line bundles as shown in \cite{hanlon2019monodromy} (see also \cite{treumann2019kasteleyn,hanlon2022aspects,shende2022toric,bose2025contact}).
    In many ways, our results can be seen as generalizations of the explicit construction of Hamiltonians exhibiting this action, which can be equivalently reformulated as isotopies of stops.
        \item \v{S}penko and Van den Bergh have proven the restriction of \cref{conj:toricmon} to toric boundary divisors in \cite{spenko2024hmssymmetriestoricboundary}. 
    They construct the action first on the $A$-side by considering families of fibers of Hori--Vafa potentials and their wrapped Fukaya categories and then apply homological mirror symmetry \cite{gammage2022mirror} to pass to the $B$-side.
        We note that they do not pass to the FLTZ stop and extending their approach to obtain \cref{conj:toricmon} itself would require a technical analysis of partially wrapped Fukaya categories associated to Hori--Vafa potentials of a rather different nature than that required here for families of Legendrian stops.

    \item In the quasi-symmetric setting, \cite{zhou2020variationgitvariationlagrangian,huang2022variation} give an $A$-side description of the action in \cite{halpern2020combinatorial} using the microlocal sheaf model for the Fukaya category of $(\XA, \stp_{\Sigma})$.
    Notably, they produce families of (piecewise linear) FLTZ skeleta that induce the action.
    However, we note that their arguments rely on the window equivalences of \cite{segal2011equivalences,halpern2015derived, ballard2019variation} and, in contrast to the results here, do not produce the action from the topology of the moduli space.

    For some examples, a similar construction, again using window technology in the proofs, is given in \cite{donovan2021mirror} where the autoequivalences are put in the context of perverse schobers \cite{kapranov2014perverse,bondal2018perverse}, which are meant to fill in the local system of categories over the discriminant locus.
\end{itemize}

\subsection{Conjectures and Questions}
There are several directions that we leave unexplored.
For example, we work with the space of stops with embedded Lagrangian projection to make a combinatorial comparison to the space of embedded graphs, but there is no geometric reason to constrain to this set.
In fact, we conjecture:
\begin{conj} \label{conj:fipsisleg}
    Let $\Sigma$ be a fan so that $\XB_{\Sigma}$ is a toric Calabi--Yau.
    The Fayet--Iliopoulos parameter space \cref{def:fips} is homotopy equivalent to $\mathcal L_{\stp_\Sigma}(Y)$.
\end{conj}
\cref{conj:fipsisleg} has a parallel combinatorial version in dimension two;
we could not find a proof of this (seemingly simpler) statement in the literature: 
\begin{conj}
    The moduli space $\mathcal N_\nn(\RR^2)$ of planar graphs with $\nn$ bounded faces is homotopy equivalent to $\uConf_\nn(\RR^2)$, the configuration space of $\nn$ unlabeled points in the plane.
        \label{conj:graphs}
\end{conj}
The complex structure moduli space is believed to be related to the Bridgeland stability manifold \cite{bridgeland2009spaces}. The relationship to Bridgeland stability conditions has been made explicit in the setting of 3-fold flops \cite{donovan2019stringy} and Dynkin diagrams \cite{bridgeland2009stability}, the latter covering the setting of this paper.

\begin{qs}
    Can the moduli space of mostly Legendrian stops be related in any way to the space of stability conditions on the partially wrapped Fukaya category?
    We note that a generic choice of mostly Legendrian stop gives a preferred choice of generators of the partially wrapped Fukaya categories supported on (small negative pushoffs of) cotangent fibers.
    For the FLTZ stop, these generators coincide with the Bondal--Thomsen line bundles. 
    Thus, it is also natural to ask whether there is a (geometric) stability condition on the partially wrapped Fukaya category for which the Bondal--Thomsen generators are stable.
\end{qs}

\begin{rem}
    There is precedent for studying homological mirror symmetry where the $B$-side is not K\"ahler (and does not even admit a symplectic structure). For instance, HMS for the Hopf surface is studied in \cite{ward2021homologicalmirrorsymmetryelliptic}.
    In that work, the $A$-side category can be understood to be a generalization of the partially wrapped Fukaya category to the non-exact setting. 
    In particular, the moduli space of stops is still a reasonable object.
    Should we interpret the stop moduli space in the non-K\"ahler setting as a stringy ``symplectic'' moduli space on the mirror?
\end{rem}

\subsection{Outline}
The remainder of the paper is organized as follows.
\Cref{sec:SKMSreview} is an in-depth literature review covering both $A$-side and $B$-side perspectives on the toric SKMS; it also contains a worked example of how these different viewpoints fit together in the case of the $A_\n$ singularity. We use some notation from this review throughout the paper.

The new mathematical arguments are contained in \cref{sec:stopModuliSpace,sec:monodromyAction,sec:applications}.
In \cref{sec:stopModuliSpace}, we provide the geometric heart of this paper: a definition of the moduli space of Legendrian stops and the identification of the braid group inside the fundamental group of this moduli space. \Cref{sec:monodromyAction} contains a rather technical construction that allows us to lift paths in this moduli space to symplectomorphisms of the corresponding stopped symplectic manifolds. While important to the main result, this section can be safely skipped by readers who are primarily interested in the geometric and algebraic aspects of the main theorem. 

The last section, \cref{sec:applications}, contains examples and applications; in particular, we examine how the action of the fundamental group of the moduli space of Legendrian stops on the $A$-side translates through the mirror to the spherical and line bundle twists, provide a dictionary to the diagrammatic methods of the Bondal--Thomsen collection, and extend the result to obtain equivalences between VGIT chambers and actions on partial resolutions of the $A_\n$ singularity.

Finally, in \cref{app:HMS}, we extend the Floer-theoretic proof of homological mirror symmetry for toric varieties to toric Deligne--Mumford stacks. We use only the simplest example of this result in this paper, the $A_\n$ quotient stack, but have included the full result as it may be of future use to the community.

\subsection{Acknowledgements}
We are grateful to Ed Segal for encouraging us to produce an $A$-side moduli mirror whose monodromy induces variation of GIT.
We thank Charlotte Bartram, Sheel Ganatra, Oleg Lazarev, Nick Sheridan, and \v{S}pela \v{S}penko, for useful conversations, as well as the organizers of ``Higher categorical methods in algebra and geometry'' at which we were originally inspired to work on this problem. 
We also thank Michael Wemyss for asking us about the partially resolved case, which led to \cref{subsec:partialResolutions}.
Finally, we thank Charlie Egan and Louis Theran for introducing us to Voronoi diagrams as a method to obtain a planar graph from a configuration of points.
Some figures were created using PyPlot \cite{pyplot}.

AH was supported by the National Science Foundation through awards DMS-2549013 (previously as DMS-2412043) and DMS-2549038.
JH was supported by EPSRC grants EP/V049097/1 and EP/Z53528X/1.

\section{A review of the Stringy K\"ahler moduli of toric mirrors}
\label{sec:SKMSreview}

This section is a literature review and a source of motivation for the constructions in \cref{sec:stopModuliSpace,sec:monodromyAction}. The proofs later in the paper do not rely on the heuristic open Gromov--Witten discussion below, except through the standard background and the precise references explicitly cited there.

The homological mirror symmetry framework for toric varieties is well established \cite{abouzaid2006homogeneous,fang2011categorification,kuwagaki2020nonequivariant}. An advantage of this setting is that the complex moduli of the toric DM stack $\XB_\Sigma$ are trivial, as are the symplectic moduli of the symplectic mirror $\XA$.
However, the stringy K\"ahler moduli space $\mathcal M^{\XB}_\omega$ of the toric variety may be non-trivial.

There is no rigorous definition of the stringy K\"ahler moduli space $\mathcal M^{\XB}_\omega$ of a $B$-side model \cite{bridgeland2009spaces}; in this section, we provide an extended review of the literature with an emphasis on the moduli space of $A$-side Landau--Ginzburg superpotentials on the mirror (\cref{subsec:Aside}) and its relationship to variation of GIT (\cref{subsec:Bside}). In \cref{subsec:A1,subsec:braidAn} we give two worked-out examples of the construction of the SKMS of the $A_1$ singularity and the braid group action on the derived category of the $A_{\n}$ singularity.

\subsection{Moduli of \texorpdfstring{$A$}{A}-side Landau--Ginzburg models} 
\label{subsec:Aside}
We first give an overview of the symplectic motivation.
For this section, a smooth toric variety is determined by the data of a smooth fan $\Sigma\subset N_{\RR} = N \otimes \RR$. At the level of precision employed in this literature review, one may replace a toric variety with a smooth toric stack (see \cref{app:HMS}). We will use $\nn$ to denote the rank of $N$.
A mirror to a toric variety $\XB_\Sigma$ defined from this data is a symplectic Landau--Ginzburg model $(\XA,\omegaA, W_\Sigma^{\XA})$, where $\XA$ is symplectomorphic to $T^*T^{n}$ with the standard symplectic form and $W_\Sigma^{\XA}\colon \XA\to \CC$ is a superpotential determined by the data of the fan $\Sigma$.
We note that the superpotential is not unique and, by asking that $W_\Sigma^{\XA}$ be holomorphic, the choice of superpotential can be understood to reflect the variation of the complex structure of $\XA$.
There are several recipes for constructing a potential $W_\Sigma^{\XA}$,
but one method that makes all choices explicit arises from considering SYZ mirror symmetry and open Gromov--Witten theory. When constructed in this fashion, the choices are determined by the data of a K\"ahler structure on the mirror (see \cite[Section 2.2]{auroux2008special} for a summary of this philosophy).
More specifically, a choice of K\"ahler form on $\XB$ determines a moment map $\XB\to \Delta\subset \RR^n$ with generic fibers being $n$-tori, $T^n$.
To our knowledge, this construction has not been rigorously worked out in the literature (see \cite[Section 4.2]{auroux2007mirror} for a discussion on the current state of this viewpoint).
For us, the following heuristic suffices: we think of the complex structure induced on $\XA$ as identifying $\XA$ with the \emph{subset} $\operatorname{Log}^{-1}(\Delta)\subset (\CC^*)^n$ resembling an affinoid domain.
The open Gromov--Witten potential $W_\Sigma^{\XA}$ also depends on the choice of $\omegaB$. This is a formal power series with \emph{real} coefficients which provides an $\omegaB$-weighted count of pseudoholomorphic disks in $\XB$ with boundary on SYZ tori. Ignoring possible issues of convergence, the resulting function $W_\Sigma^{\XA} \colon \XA\to \CC$ is holomorphic with respect to the complex structure on $\XA$ (also determined by $\omegaB$). In practice, these power series are difficult to compute when $\XB$ is not Fano\footnote{These potentials have been computed for toric CY surfaces \cite[Theorem 4.2]{lau2012mirror}, which is the case that we are interested in.}.

It is expected that equipping $\XB$ with a complexified K\"ahler form corresponds to modifying the open Gromov--Witten potential so that it is a formal power series with \emph{complex} coefficients. It is unclear to us how --- or even if --- a complexified K\"ahler form on $\XB$ corresponds to a modification of the complex structure on $\XA$. However, since the $A$-model is determined by $W_\Sigma^{\XA}$ (and only uses the complex structure as auxiliary data), one viewpoint is that the moduli space of suitably complexified open Gromov--Witten potentials $\mathcal M^{\XA}_{J}$ is the appropriate $A$-side analogue to the stringy K\"ahler moduli space $\mathcal M^{\XB}_{\omega}$.
Without dwelling on the mathematical gaps in the story, let us describe some expected features of this moduli space in a worked example (\cref{exm:VGITP1}) and give a purely algebraic prescription for this moduli space in the setting of toric varieties (\cref{subsubsec:FIPS}).
\subsubsection{A local system of categories over the SKMS: the projective line}
\label{exm:VGITP1}
We briefly recall the construction of \cite[Section 1.1]{hanlon2019monodromy}. Consider $\XB_\Sigma=\PP^1$. The open Gromov--Witten potential is $W_{\Sigma}^{\XA}(z)=z+\frac{\exp(-\alpha)}{z}$, where $\alpha$ corresponds to the symplectic area of $\PP^1$. For every $c\in\CC^*$ we have a complexified potential:

   \begin{align*}
 W_{\Sigma,c}^{\XA}\colon \CC^*\to&\CC\\
 z\mapsto &z+\frac{c}{z}.
\end{align*}
We therefore have a ``local system'' of symplectic Landau-Ginzberg models $(\CC^*, W_{\Sigma, c})$ over $\CC^*$. By associating to each symplectic LG model a Fukaya category, we expect to build a local system of categories over $\CC^*$. 
 The model for the Fukaya category with the cleanest description in this setting is the Fukaya--Seidel category, whose objects are Lagrangian submanifolds of $\CC^*$ that fiber over the positive real axis under $W_{\Sigma}^{\XA}$ outside of a compact set. By definition, a generating set for the Fukaya--Seidel category is prescribed by choosing paths in the base $\CC$ from the critical values of the potential $W_{\Sigma}^{\XA}$ which eventually agree with the positive real axis, and taking the Lagrangian lifts (thimbles) of these paths. 
 If we take the loop $c=\exp(i\theta)$ for $\theta \in [0, 2\pi]$, the critical values $\pm 2\sqrt{c}$ twist around each other, and we can ``follow'' the Lagrangians through this loop (see \Cref{fig:FSLoop}). 
\begin{figure}[h!]
\centering
\scalebox{.6}{
\begin{tikzpicture}\begin{scope}[]

    \draw[fill=gray!20]  (0,0) node (v1) {} ellipse (2 and 2);
    \draw[dashed]  (v1) ellipse (0.5 and 0.5);
    \node at (0,-0.5) {$\times$};
    \node at (0,0.5) {$\times$};
    \draw[red] (0,-0.5) .. controls (0,-1) and (1,0) .. (1.5,0);
    \draw[blue] (0,0.5) .. controls (0,1) and (1,0) .. (1.5,0);
    \draw (1.5,0) -- (2,0);
    \draw[green] (0, -0.5) -- (0,0.5);
    \end{scope}
    
    \begin{scope}[shift={(9,0)}]
    
    \draw[fill=gray!20]  (0,0) node (v1) {} ellipse (2 and 2);
    \draw[dashed]  (v1) ellipse (0.5 and 0.5);
    \node at (0,-0.5) {$\times$};
    \node at (0,0.5) {$\times$};
    \draw (1.5,0) -- (2,0);
    \draw[blue] (0,-0.5) .. controls (0,-1) and (-1,-0.5) .. (-1,0) .. controls (-1,0.5) and (-0.5,1) .. (0,1) .. controls (0.5,1) and (1,0) .. (1.5,0);
    \draw[red] (0,0.5) .. controls (0,1) and (1,0) .. (1.5,0);
    \draw[green] (0, -0.5) -- (0,0.5);
    \end{scope}
    
    \begin{scope}[shift={(4.5,0)}]
    
    \draw[fill=gray!20]  (0,0) node (v1) {} ellipse (2 and 2);
    \draw[dashed]  (v1) ellipse (0.5 and 0.5);
    \node at (0.5,0) {$\times$};
    \node at (-0.5,0) {$\times$};
    \draw[red] (0.5,0) .. controls (1,0) and (1,0) .. (1.5,0);
    \draw (1.5,0) -- (2,0);
    \draw[blue] (-0.5,0) .. controls (-1,0) and (-0.5,1) .. (0,1) .. controls (0.5,1) and (1,0) .. (1.5,0);
    \end{scope}
    \end{tikzpicture}
    
    
    
     }
\caption{A loop of potentials $W_{\Sigma}^{\XA}(z)=z+{\exp(i\theta)}z^{-1}$ corresponding to an autoequivalence of Fukaya--Seidel categories.
The figures represent the critical values of the potential in $\CC$.
The red and blue paths lift to Lagrangian submanifolds, and are twisted as we rotate the parameter $\exp(i\theta)\in S^1$. The green path corresponds to a Lagrangian sphere in $\XA$; this autoequivalence can also be described in terms of the spherical twist around this object.}
\label{fig:FSLoop}
\end{figure}
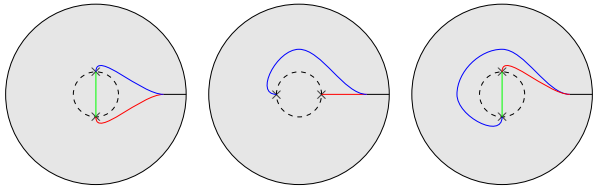

 Under the HMS dictionary, the resulting autoequivalence of the Fukaya--Seidel category is identified with $- \otimes \mathcal O(-1)$. 
This story was generalized in \cite{hanlon2019monodromy}, where it was shown that monodromy arising from rotation of the arguments of the monomials in the Hori--Vafa potential could be used to construct the mirror action to tensor product by a line bundle for the monomial-admissible Fukaya--Seidel category.
 
Something unique to this example (and not present in the case of HMS for $\PP^n$) is that there is a Lagrangian $S^1$ which parameterizes the unit circle in $\CC^*$. The autoequivalence of the Fukaya--Seidel category built by monodromy over the SKMS can be obtained by taking the symplectic Dehn twist around this Lagrangian sphere. This induces a spherical twist functor on the Fukaya--Seidel category.

\subsubsection{Hori--Vafa potential and $\FIPS$}
\label{subsubsec:FIPS}
In practice, homological mirror symmetry for toric varieties does not use the open Gromov--Witten potential, but rather one of several different substitutes --- see \cref{tab:FukayaSeidelCategories} for a summary of these alternative approaches as well as the current state of HMS for each model. 

One of the most commonly used substitutes --- the Hori--Vafa potential \cite{hori2000mirror}--- allows one to define a purely formal stand-in for the SKMS via ``complex variation'' of the coefficients of the superpotential.
For a fan $\Sigma$, it is defined as 
\[W_{\mathrm{HV}, \Sigma}(z):= \sum_{\rho\in \Sigma(1)} c_\rho z^\rho.\]
Here, we have used the identification of characters on the $A$-side with cocharacters on our input $B$-side.
\begin{rem}
 The Hori--Vafa potential can be viewed as a ``first-order'' approximation of the open Gromov--Witten potential.
 To see this relation, we observe that there is a connection between the 1-dimensional cones of $\Sigma$ and homology classes $[\rho]\in H_2(\XB_\Sigma, T^n; \ZZ)$ (where $T^n$ is the orbit of $1$ in the algebraic torus).
  When $\XB_\Sigma$ is Fano and we let $c_\rho=\exp(-\omega([\rho]))$, the Hori--Vafa and open Gromov--Witten potentials agree (subject to the domain of definition).
  In many other cases, the appearance of higher-order terms in the open Gromov--Witten potential does not meaningfully alter the symplectic geometry of the symplectic Landau--Ginzburg model. For this reason, the literature frequently glosses over any potential discrepancies between these two potentials in the non-Fano case.
\end{rem} 
To construct a ``Stringy K\"ahler Moduli space'' for this model, one takes the space of ``valid'' potentials $W_{\mathrm{HV}, \Sigma}(z)$. Let $d$ denote the number of $1$-dimensional cones of $\Sigma$, and let $n=\operatorname{rank}(N)$; equivalently, in a linear toric GIT presentation with torus rank $r$, one has $n=d-r$. 
The $d$ coefficients $c_\rho\in \CC$ can vary, but not all coefficient choices are allowed. In particular, we must avoid those lying on $V(E_A)$, where $E_A$ is the principal $A$-determinant as defined and studied in \cite{gelfand1994discriminants}, a higher-dimensional generalization of a discriminant. 
\begin{landscape}
\begin{table}[p]
\centering
        \rowcolors{2}{gray!50}{white}
\begin{tabular}{p{3cm}|p{3.5cm}|p{5cm}|p{9cm}}
    \textbf{Potential Data} & \textbf{State of HMS} & \textbf{Moduli of Potentials} & \textbf{Local Systems over Moduli space}\\ \hline
 OGW Potential & No $A$-model defined yet in this setting. & This model most accurately captures the physical definition of the SKMS of $\XB_\Sigma$. & No defined $A$-model \\
 Hori--Vafa potential& After suitable deformation to a tropical limit, we have HMS  \cite{abouzaid2006homogeneous}. & By definition, this is the $\FIPS$, which has a completely combinatorial definition & Due to the rigidity of the data used to define FS-category, it is difficult to construct parallel transport maps. Parallel transports are known to exist on the fiber category \cite{spenko2024hmssymmetriestoricboundary}.\\
 Monomial Admissibility Data  & HMS proven \cite{hanlon2019monodromy}.& Moduli space by definition is the torus $(S^1)^{|\Sigma(1)|}$ & There exists a local system of categories over the moduli space, with autoequivalences mirror to twists by line bundles.\\
 FLTZ Skeleton (microlocal)  & HMS is proven \cite{fang2011categorification,kuwagaki2020nonequivariant} &
 Unclear how to construct the moduli space from data
    & Under additional hypothesis, non-characteristic deformations of skeleta correspond to window derived equivalences \cite{zhou2020variationgitvariationlagrangian,huang2022variation}\\
 FLTZ Skeleton  (partially wrapped) & HMS is proven  \cite{hanlon2022aspects}. &\cref{def:legModuliSpace} & For variation of Legendrian stops \cite{sylvan2019orlovviterbofunctorspartially}. For FLTZ stops, \cref{thm:monodromyAction}\\
 SYZ in complement of smooth anticanonical & HMS not yet proved in this setting  & No interesting moduli of LG potential in this setting, but now have non-trivial moduli of complex structure  & In some settings, the $A$-side mirror contains Lagrangian spheres \cite{gross2018homological}, giving rise to spherical autoequivalences of the Fukaya category \cite[Section 1.3]{seidel2001braid} via symplectic Dehn twist.
\end{tabular}
\caption{Different frameworks for HMS for toric varieties.}
\label{tab:FukayaSeidelCategories}
\end{table}
\end{landscape}
The Fayet--Iliopoulos Parameter Space ($\FIPS$) is defined as the parameter space over which these potentials live:
\begin{equation}
    \label{def:fips}
 \FIPS = \frac{\CC^d \backslash V(E_A)}{(\CC^*)^n}
      = \frac{\CC^d \backslash V(E_A)}{(\CC^*)^{d-r}}
\end{equation}
where by construction there is a natural action of $(\CC^*)^n=(\CC^*)^{d-r}$ via $A$ on the coefficients, which we quotient out.

It is worth noting that $\FIPS$ does not always coincide with the understanding of the stringy K\"ahler Moduli space as the complex structure moduli space; see \cite[Remark 2.8]{donovan2015mixed}. 
The space $\FIPS$ is a smooth Deligne--Mumford stack.

\subsection{\texorpdfstring{$B$}{B}-side moduli via Toric Variation of GIT}
\label{subsec:Bside}
Given the patchwork landscape of different $A$-side Landau--Ginzburg models and approaches to defining the SKMS, we turn to the $B$-side to look for corroborating evidence. 
Despite our lack of a formal definition of the $B$-side SKMS, we can recover a shadow of this conjectural space through toric variation of GIT.

Toric geometry is equivalent to the geometry of \textit{linear toric Geometric Invariant Theory (GIT)} problems, the simplest kinds of GIT problems. Linear toric GIT problems are given by an algebraic torus $T \cong (\CC^*)^r$ acting on a vector space $V \cong \CC^d$, where $r$ is the rank and $d$ is the dimension. 
After choosing bases for $T$ and $V$, these actions are always given by
$$
(\lambda_1, \dots, \lambda_r) \cdot (z_1, \dots, z_d) = ({\lambda_1}^{q_{11}}{\lambda_2}^{q_{21}}\dots {\lambda_r}^{q_{r1}} z_1, \dots, {\lambda_1}^{q_{1d}}{\lambda_2}^{q_{2d}}\dots {\lambda_r}^{q_{rd}} z_d)
$$
where $q_{ij} \in \ZZ$ are integers for $i \in \{1, \dots,r\}, j \in \{1,\dots,d\}$.
We obtain an $r \times d$ integer matrix $Q = (q_{ij})$, called the \textit{weight matrix} whose columns $\{q_k\}_{k=1}^d$ are called \textit{weight vectors}.
GIT is an algebraic framework for building well-behaved quotients. In most cases, to obtain well-behaved quotients, we will have to remove some bad orbits (the unstable locus) with non-generic behaviour before taking the quotient.
As a result, GIT quotients are not unique and depend on a choice of stability condition specifying the unstable locus we remove.
In our setup, the stability condition is determined by a character $\theta$ of $T$, making our ``stability space'' the character lattice $M =  \hom(T,\CC^*) \cong \ZZ^r$. There is an entirely combinatorial process which computes the unstable locus $U_{\theta} \subset \CC^d$ to be removed before forming the (stack) quotient $[(\CC^d \setminus U_{\theta}) / (\CC^*)^r]$; see, for instance, \cite[Chapter 14]{cox2024toric}, or \cite[Section 4]{coates2018crepant} for a clear and brief introduction.

The stability space $M$ has a wall-and-chamber decomposition. The chambers of the space correspond to stability conditions with the same unstable loci (and hence yield the same GIT quotients). 
A fundamental result is that these chambers are the interiors of the top-dimensional cones of a toric fan, often called the secondary or GKZ fan \cite[Chapter 14]{cox2024toric}.
The GIT quotients corresponding to characters interior to the maximal cones of the GKZ fan are smooth toric Deligne--Mumford stacks in the sense of \cite{borisov2005orbifold}.
These are the only GIT quotients we will discuss, and from now on, when we say GIT quotients, we are referring to these.

\begin{rem}
 The stability condition space for toric linear GIT is naturally discrete, but admits a straightforward enhancement to a real manifold via symplectic geometry.
 Linear toric GIT can be recast as symplectic reduction \cite{kempf1979length}, see also the survey \cite{woodward2010moment}.
 In this framework, the standard coordinatewise $(\CC^*)^d$-action on $\CC^d$ restricts to a Hamiltonian action of $(S^1)^d$, and the GIT action corresponds to a Hamiltonian action of $(S^1)^r$.
  These Hamiltonian actions give us moment map projections $\pi_1$ and $\pi_2$: 
    \[\begin{tikzcd}V \arrow{r}{\pi_1} \arrow[bend left]{rr}{\pi_2} &  \RR_{\geq 0 }^d \arrow{r}{p} & \Delta\end{tikzcd}\]
 where (under mild conditions on the action) $\Delta \subset M_\RR=M\otimes_\ZZ \RR$ is the image of $p$ inside our continuous real enhancement of the GIT parameter space.
  Associated to a generic point $q\in \Delta$, the symplectic reduction/quotient $\XB_q:=\pi_2^{-1}(q)/(S^1)^r$ carries a Hamiltonian $(S^1)^{d-r}$ action whose moment polytope is $p^{-1}(q)$.
 The combinatorial type of $p^{-1}(q)$ varies with $q$, as drawn in \cref{fig:symplecticReduction}.
 The different chambers of the GIT problem correspond to regions over which the combinatorial type of $p^{-1}(q)$ is constant.
\end{rem}
\begin{figure}
    \centering

\begin{tikzpicture}

\draw [decorate, decoration={ticks, segment length=1cm}](-2,-1) -- (8,-1);
\draw[<->] (-2.5,-1) -- (8.5,-1);
\draw (3,-0.6) -- (3,-1.4);
\node[right] at (8.5, -1) {$M_\RR$};

\begin{scope}[shift={(0.5,1)}, rotate=90, yscale=-1]

\fill[gray!20] (-1,1.5) -- (0,1.5) -- (0,1.5) -- (0.75,0) -- (-1,0) -- cycle;
\draw [help lines, step=0.5cm] (-1,0) grid (1,2);
\draw (-1,1.5) -- (0,1.5) -- (0,1.5) -- (0.75,0);
\end{scope}

\begin{scope}[shift={(-2.5,1)},rotate=90, yscale=-1]

\fill[gray!20] (-1,1.5) -- (0,1.5) -- (0,1.5) -- (0.75,0) -- (-1,0) -- cycle;
\draw [help lines, step=0.5cm] (-1,0) grid (1,2);
\draw (-1,1.5) -- (0,1.5) -- (0,1.5) -- (0.75,0);
\end{scope}

\begin{scope}[shift={(3.5,1)},rotate=90, yscale=-1]

\fill[gray!20] (-1,1.5) -- (-0.25,1.5) -- (0,1.25) -- (0.65,0) -- (-1,0) -- cycle;
\draw [help lines, step=0.5cm] (-1,0) grid (1,2);
\draw (-1,1.5) -- (-0.25,1.5) -- (0,1.25) -- (0.65,0);
\end{scope}

\begin{scope}[shift={(6.5,1)},rotate=90, yscale=-1]

\fill[gray!20] (-1,1.5) -- (-0.5,1.5) -- (0,1) -- (0.5,0) -- (-1,0) -- cycle;
\draw [help lines, step=0.5cm] (-1,0) grid (1,2);
\draw (-1,1.5) -- (-0.5,1.5) -- (0,1) -- (0.5,0);
\end{scope}

\node at (0,-1.5) {$q<0$};
\node at (6,-1.5) {$q>0$};
\end{tikzpicture} \caption{Variation of GIT can also be understood in terms of varying the moment map parameter in symplectic reduction.
Here, we look at the variation of symplectic reduction for the $A_1$ singularity.}
\label{fig:symplecticReduction}
\end{figure}
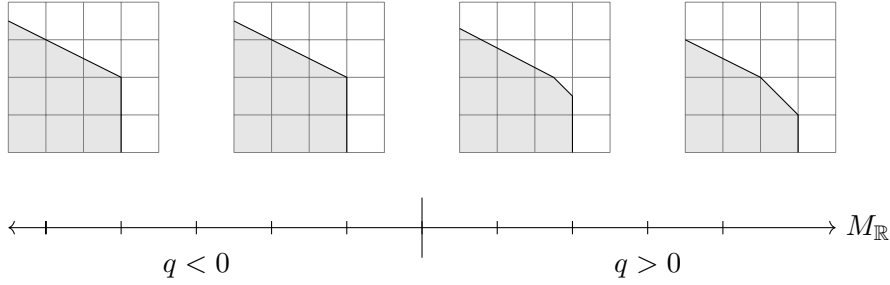
\subsubsection{Tropicalizations of the Hori--Vafa potential and VGIT chambers}
The Hori--Vafa potential can also be read from the presentation of a toric Calabi--Yau variety via linear toric GIT. 
Consider the dual weight matrix $Q^\vee : \ZZ^r \to \ZZ^d$; the cokernel $A: \ZZ^d \to N $ (where $N$ has rank $d-r$ and may have torsion) gives both the toric fans of the quotients and the form of the superpotential $W$. We ignore any torsion part of $N$, consider the remaining cokernel map $A$, and by picking bases, we can view $A$ as a $(d-r) \times d$ matrix.
We can describe $A$ by its set of column vectors $\{\rho_i\}$, which correspond to the rays generating the 1-dimensional cones of $\Sigma$. 
The Hori--Vafa potential is now given by
$$W_A(z):= \sum_i c_{\rho_i} z^{\rho_i}.$$

The Calabi--Yau property implies that the convex hull $P\subset \RR^{d-r}$ of the rays $\{\rho_i\}$ lies in an affine hyperplane. The compact polytope  $P$ is called the \textit{primary polytope}, and coherent triangulations of $P$ give the fans of the GIT quotients by taking cones over the simplices of the triangulation \cite[Chapter 7, Theorem 1.7]{gelfand1994discriminants}.
By the same argument, one sees that the Hori--Vafa potential factors as $W_A(z_1, \ldots, z_\nn)=z_1\cdot f_A(z_2, \ldots, z_\nn)$.
The Newton polytope of $f_A$ is $P$. When the tropicalization of $f_A$ is a smooth tropical hypersurface, the tropical hypersurface is combinatorially dual to a triangulation of $P$;
similarly, every choice of coherent triangulation for $P$ specifies a choice of coefficients for $f_A$ with a dually-tessellated tropicalization.
\subsubsection{Derived autoequivalences on the  \texorpdfstring{$B$}{B}-side}
The previous section suggests that the VGIT parameter space is a ``tropicalization'' of the SKMS. 
Indeed, the $\FIPS$ naturally embeds into the \textit{secondary toric stack}\footnote{One way to see this is that the Newton polytope of the principal $A$-determinant $E_A$ is the secondary polytope, whose normal fan is the secondary fan.}, which is a toric Deligne--Mumford stack (in the sense of \cite{borisov2005orbifold}) whose fan is given by the secondary fan.
Under this tropicalization map, fixed points in the toric variety associated to the secondary fan correspond, via the orbit-cone correspondence, to maximal cones (or chambers) of the secondary fan, and hence to GIT quotients.
The intersections of arbitrarily small neighborhoods around these fixed points with $\FIPS$ are called \emph{large complex structure limits} (LCSLs), with one such LCSL for each GIT quotient.
This gives us a ``clean statement'' of our expectation that there exists a local system of categories over the SKMS: the fundamental groupoid of $\FIPS$, with basepoints chosen at the LCSLs of the GIT quotients, acts on the derived categories of the quotients via line bundle twists and certain equivalences known as \emph{window equivalences}\footnote{Also referred to as grading restriction rules.}, $\Db(\XB) \to \Db(\XB')$.
\begin{figure}[h!]
    \centering
    \includegraphics[width=0.6\linewidth]{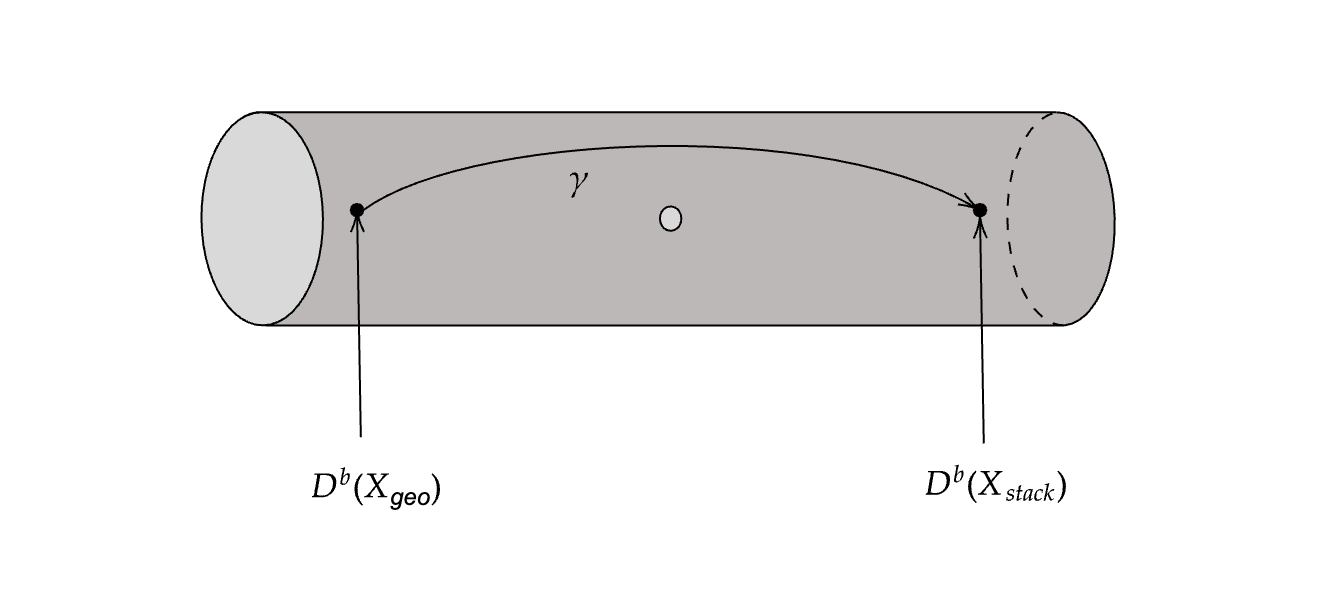}
    \caption{A picture of $\FIPS \cong \CC^* \setminus \{1\}$ for the $A_1$ singularity.
 The path $\gamma$ can be chosen to represent a ``window'' equivalence $\Db(\XB_{\mathrm{geo}}) \cong \Db(\XB_{\mathrm{stack}})$ as in \cite{segal2011equivalences}.
 There are countably infinitely many of these, and the choice is non-canonical.}
\end{figure}
These window equivalences arise via the Artin stack $[V/T]$, as in the framework of \cite{halpern2015derived}, with explicit constructions developed in \cite{segal2011equivalences,halpern2016autoequivalences,ballard2019variation,halpern2020combinatorial}.

In fact, \cite{halpern2016autoequivalences} shows that composing a window equivalence with the inverse of another corresponds to a \textit{spherical twist about a spherical functor}, where a spherical functor is a generalization of the spherical objects introduced in \cite{seidel2001braid}.

\begin{df}\cite[Chapter 8]{huybrechts2006fourier}
    \label{sphericalobject}
 An object $\Es^{\bullet} \in \Db(\XB)$ is called \textit{spherical} if 
    \begin{enumerate}
        \item $\Es^{\bullet} \otimes \omega_{\XB}\cong \Es^{\bullet}$, where $\omega_{\XB}$ is the canonical bundle, and
        \item \[
        \Hom_{\Db(\XB)}(\Es^{\bullet},\Es^{\bullet}[i]) \cong 
        \begin{cases}
 k & \text{if } i = 0 \text{ or } i = \dim(\XB) \\
            0 & \text{otherwise}.
        \end{cases}
        \]
    \end{enumerate}
\end{df}
We define the \textit{spherical twist} $T_{\Es^{\bullet}} : \Db(\XB) \to \Db(\XB)$ about a spherical object to be the cone on the evaluation map:
\begin{equation}
    \label{sphericalobjecttwist}
 T_{\Es^{\bullet}} (\Fs^{\bullet}) = \operatorname{Cone}\left( \mathbf{R}\!\operatorname{Hom}(\Es^{\bullet}, \Fs^{\bullet})\otimes \Es^{\bullet} \to \Fs^{\bullet}\right)
\end{equation}
where $\mathbf{R}\!\operatorname{Hom}(\Es^{\bullet}, \Fs^{\bullet})$ denotes the derived Hom complex.\footnote{There are some technical issues with the definition of the spherical twist because taking cones is not functorial. We bypass any such issues because in our geometric context, where our triangulated category is the derived category of a DM stack, we can describe any spherical twist (even about a spherical functor) as a Fourier-Mukai transform.}
By \cite[Proposition 8.6]{huybrechts2006fourier}, the spherical twist $T_{\Es^{\bullet}}:\Db(\XB) \to \Db(\XB)$ about any spherical object $\Es^{\bullet}\in \Db(\XB)$ is an autoequivalence. The conjecture may be rephrased as saying that the fundamental groupoid of $\FIPS$, with basepoints chosen at the relevant large complex structure limits, acts on $\Db(\XB)$ via spherical twists.
While we do not give any details here, these spherical twists are explicitly determined by the GIT data, and the expected representation can be described quite precisely.
However,  \cref{conj:toricmon} remains unproven even in the toric or abelian GIT setting.

That said, partial results are known.
For quasi-symmetric GIT problems (a special case of Calabi--Yau where $E_A = 0$ defines a hyperplane arrangement), \cref{conj:toricmon} was proven in \cite{halpern2020combinatorial}.
In addition, \cite{alexthesis} verified the conjectural representation when restricted to large radius paths in two-dimensional examples.
The first author's thesis \cite{michelathesis} studies \cref{conj:toricmon} for toric Calabi--Yau 3-folds of Picard rank 2 from this $B$-side perspective.
When the action of the torus $T$ on the vector space $V$ lands in $SL(V)$ (i.e., the weight vectors sum to $0$), the GIT quotients are Calabi--Yau (they have trivial canonical line bundle) and their derived categories are equivalent \cite{halpern2015derived}.

\subsection{Worked Example: SKMS, VGIT, and $\FIPS$ for the \texorpdfstring{$A_1$}{A1} singularity}
\label{subsec:A1}
 Suppose we have $\CC^*$ acting on $\CC^3$ with weight matrix $Q=(-1,2,-1)$, i.e.,
    $$
    \lambda \cdot (x,y,z) = (\lambda^{-1}x, \lambda^2 y, \lambda^{-1}z).
    $$
 The stability space  $M \cong \ZZ$ is an integer lattice of rank 1. We have two GIT quotients:
    \begin{align*}
        \XB_{stack} &= \CC^3 \backslash \{ y=0\} / \CC^* = \left[ \CC^2_{x,z} / \ZZ_2 \right] \\
        \XB_{geo} &= \CC^3 \backslash \{ x=z=0\} / \CC^* = \Tot\, \mathcal{O}(-2)_{\PP^1_{x,z}} = K_{\PP^1}
    \end{align*}
 where $\XB_{stack}$ denotes the GIT quotient for $\theta >0$, while $\XB_{geo}$ denotes the GIT quotient for $\theta<0$.
    $\theta = 0$ is our wall. $\XB_{stack}$ is a smooth stack associated to the $A_1$ surface singularity, and $\XB_{geo}$ is the crepant resolution of the $A_1$ surface singularity.

\subsubsection{$A$-side moduli}
We provide a non-rigorous but instructive example motivating the expectation that the Stringy K\"ahler moduli space is covered by charts corresponding to different toric stacks on the $A$-side. 
The toric K\"ahler form on $\XB_{geo}$ is determined completely by the value $\alpha:=\omegaB(\PP^1)\in \RR_{>0}$, and the open Gromov--Witten potential is 
\begin{equation}
W_\Sigma^{\XA}=x(1+y)(1+\exp(\alpha)y).
\label{eq:ogwPotential}
\end{equation}
While $W_{\Sigma}^{\XA}$ is a priori only defined on $\XA$, it can be extended to a holomorphic function on all of $(\CC^*)^2\supset \XA$. 
\begin{figure}
    \centering
    \includegraphics[scale=.2]{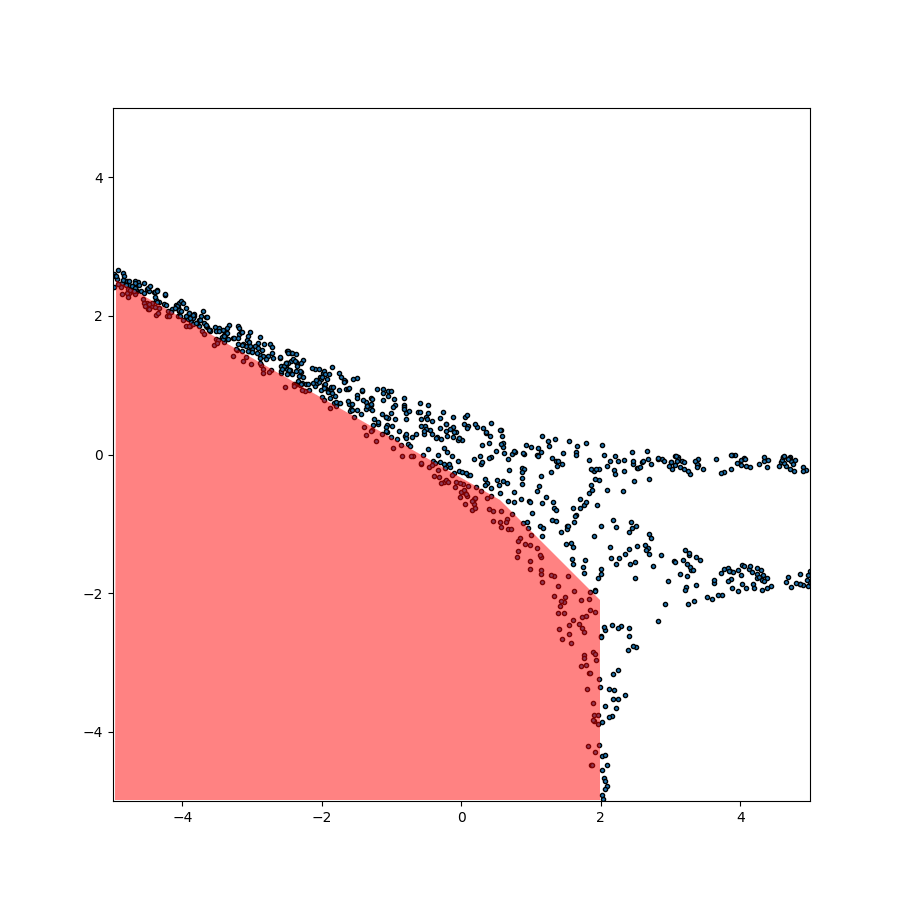}
    \caption{Tropicalization of the fiber of the open Gromov--Witten potential. The region highlighted in red approximates the moment polytope for $K_{\PP^1}$ defined by the symplectic form such that $\omega_\alpha(\PP^1)=\alpha$.}
    \label{fig:tropicalization}
\end{figure}
The relationship between $\XA$ and the potential $W_\Sigma^{\XA}$ can be loosely understood via tropicalization. Observe in \Cref{fig:tropicalization} that the bottom left-hand connected component of the complement of the amoeba approximates the moment polytope for $(\XB,\omegaB)$.
The preimage of this region under the Log map is (via our loose interpretation of \cite[Section 4.2]{auroux2007mirror}) the space $\XA$.
Continuing this line of reasoning, we may view the fiber of $W_{\Sigma}^{\XA}$ as living ``at the boundary'' of $\XA$. 
As $|\alpha|\to 0$, we see several kinds of degeneration occurring:
\begin{itemize}
    \item The moment map of $K_{\PP^1}$, as drawn in \cref{fig:symplecticReduction}, degenerates as the edge corresponding to the $\PP^1$ toric orbit shrinks to zero length.
    \item The zero set of $W_\Sigma^{\XA}$ develops a double root. 
    \item The combinatorial type of the tropicalization of $W_{\Sigma}^\XA$ changes.
\end{itemize}

Note that, at this juncture, only positive values of $\alpha$ arise from the mirror symmetry construction. Negative values of $\alpha$ are recovered by considering the OGW invariants of toric stacks. In this example, a negative value of $\alpha$ arises from the toric stack $\XB_{stack}$. Again, the open Gromov--Witten invariants come from counts of pseudoholomorphic curves with boundary on Lagrangian tori; to the best of our knowledge, the symplectic machinery has not been developed to provide a count of these curves. In the closed Gromov--Witten invariant setting, the agreement of GW invariants on $\XB_{geo}$ and $\XB_{stack}$ has been checked \cite[Corollary 3.4]{bryan2009crepant}, and corresponds to setting the parameter $\exp(\alpha)$ from $W_\Sigma^{\XA}$ to $-1$.

As before, the symplectic viewpoint is agnostic to the sign of $\alpha$, and we may consider any non-zero (even complex) value of $\alpha$.
The natural coordinate to parameterize this space is $c\in \CC^*\setminus \{1\}$, which is equal to $\exp(\alpha)$ on the portion of the stringy K\"ahler moduli space parameterized by OGW invariants of $K_{\PP^1}$ or $\CC^2/\ZZ_2$. Observe that the removed point $c=1$ corresponds to when $\alpha=0$. 
We can then study the ``variation of Landau--Ginzburg model'' via tropicalization of the fiber, as drawn in \cref{fig:variationOfTropicalization}.

\begin{figure}
\centering
\input{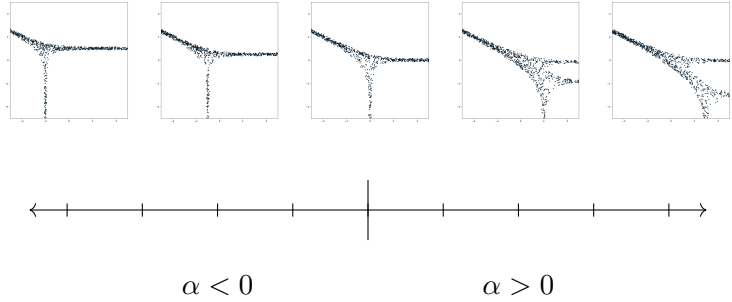} \caption{Variation of the tropicalization.
Suggestively, we draw the real parameter space for $\alpha$ to match that of the VGIT parameter space from \cref{fig:symplecticReduction}.
In the left domain, the ``large connected component'' of the amoeba complement does not vary, and agrees with the moment polytope for the symplectic orbifold $\CC^2/\ZZ_2$ from \cref{fig:symplecticReduction}.
On the right-hand side, the ``large connected component'' of the amoeba complement has a side length given by $\alpha$.
}
\label{fig:variationOfTropicalization}
\end{figure}

\begin{rem}
 When $\XB_\Sigma$ is a toric Calabi--Yau, there is another choice of $A$-side mirror \cite{chan2012syz} (for the setting of the $A_{\n}$ singularity, this is described in \cite{chan2017open}).
 If $W(x_1, \ldots, x_\nn)$ is the open Gromov--Witten potential, the mirror can also be presented as the symplectic manifold 
    \[\{(u,v, x_1, \ldots, x_{\nn})\in \CC_u\times \CC_v\times (\CC^*)^{\nn} \mid uv= W(x_1, \ldots, x_{\nn})\},\]
 with Landau--Ginzburg potential $u$.
 When presented in this manner, one can see in the $A_{\n}$ example that $\XA$ contains an $A_{\n}$ configuration of Lagrangian 2-spheres;
 more generally, the work of \cite{gross2018homological} constructs a similar configuration of Lagrangian 3-spheres in the setting of mirror symmetry for toric Calabi--Yau 3-folds.
 \label{rem:nonToricAnticanonical}
\end{rem}

\subsubsection{$B$-side moduli}
On the $B$-side, we understand from variation of GIT that the $[\AA^2/\ZZ_2]$ stack and its resolution $K_{\PP^1}$ fit into the same moduli space.

\begin{figure}[h!]
        \centering
        \begin{tikzpicture}

\draw [decorate, decoration={ticks, segment length=1cm}](-2,-1) -- (8,-1);
\draw[<->] (-2.5,-1) -- (8.5,-1);
\node[right] at (8.5, -1) {$M$};
\draw (3,-0.6) -- (3,-1.4);
\node at (1,-2) {$\theta<0$};
\node at (5,-2) {$\theta>0$};
\node at (6,-0.5) {$X_+$};
\node at (0,-0.5) {$X_-$};
\begin{scope}[rotate=-90, xscale=-1, shift={(1,-1)}]
\fill[fill=gray!20] (0,2) -- (0,1) -- (1,1.5) -- (1,2) -- cycle;
\draw [help lines, step=0.5cm] (-1,0) grid (1,2);
\draw[->] (0,1) -- (0,1.5);
\draw[->] (0,1)  -- (1,1.5);

\end{scope}

\begin{scope}[rotate=-90, xscale=-1, shift={(-5,5)}]
\fill[fill=gray!20] (6,2) -- (6,1) -- (7,1.5) -- (7,2) -- cycle;
\draw [help lines, step=0.5cm] (4.99,0) grid (7,2);
\draw[->] (6,1) -- (6,1.5);
\draw[->] (6,1) -- (7,1.5);
\draw[->] (6,1) -- (6.5,1.5);

\end{scope}

\end{tikzpicture}         \caption{Secondary fan for the $A_1$ singularity GIT problem.
 Above each chamber, we have added the fan associated with the GIT quotient.
 Note that we take the toric \emph{stack} associated to the fan.}
    \end{figure}
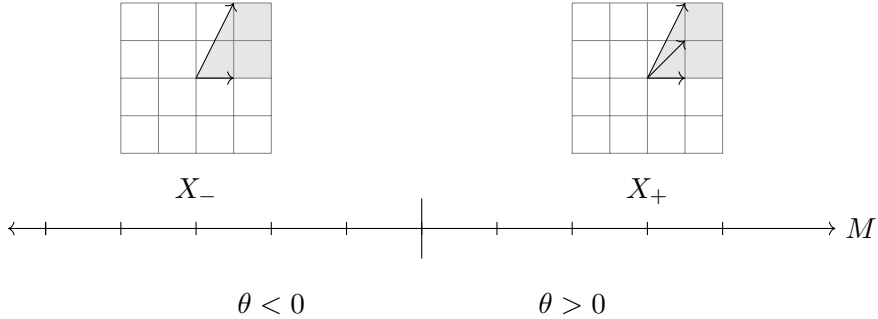

 The dual of the weight matrix has cokernel given by the matrix
    $A = 
    \begin{pmatrix}
        1 & 1 & 1 \\
        0 & 1 & 2
    \end{pmatrix}
    $.
 Hence our mirror is $((\CC^*)^2,W)$, where $W(x,y)=x(a+by+cy^2)$.
In this case, the exceptional locus is $V(E_A)=V(abc(b^2-4ac))$, and 
    $$
 \FIPS = \frac{(\CC^*)^3_{a,b,c} \backslash V(b^2-4ac)}{(\CC^*)^2} \cong \{u\in \CC^* \mid u\neq 1\} = \CC^* \setminus \{1\},
    $$
where $u=4ac/b^2$ is invariant under the $(\CC^*)^2$-action.

\subsection{Worked Example: Braid group action on the derived category of the $A_{\n}$ singularity}
\label{subsec:braidAn}
In fact, for all $A_{\n}$ we can construct a linear toric GIT problem such that the GIT quotients are resolutions of the $A_{\n}$ surface singularity. The problem is rank $\n$ and dimension $\nn+1$: $(\CC^*)^{\n}\curvearrowright \CC^{\nn+1}$ (for details, see \cite{donovan2015mixed}). The stability lattice is hence rank $\n$, and the number of quotients will be equal to the number of possible partitions of the set $\{ 1,\dots,\nn\}$. The GIT quotient $\XB_{stack}$ corresponding to the coarsest partition $\{ 1,\dots,\nn\}$ is isomorphic to the stack $[\CC^2 \: \big/_{1,\n} \: \ZZ_{\nn}]$, and the GIT quotient $\XB_{geo}$ corresponding to the finest set partition $\{ \{1 \} ,\dots, \{\nn\}\}$ is the classical geometric (non-stacky) resolution obtained via repeated blow ups. All the other quotients will be stacky and in some sense intermediate between $\XB_{stack}$ and $\XB_{geo}$.
The Hori--Vafa mirror model to the $A_{\n}$ singularity is $((\CC^*)^2, W: (\CC^*)^2 \to \CC)$ with superpotential
\begin{equation}
W(x,y) = x\sum_{i=0}^{\nn} a_i y^i.
\label{Ansuperpotential}
\end{equation}
In this example, the discriminant locus of $W(x, y)$ is the discriminant $\Delta_{\nn}$ of a generic single-variable polynomial of degree $n$.
These are the coefficients for which the polynomial $\sum_{i=0}^{\nn} a_i y^i$ has a double root.
As a result, there is a homotopy equivalence $\FIPS\to \uConf_{\nn}(\CC)$ which sends each potential $W$ to its set of $\nn$ distinct complex roots.
The fundamental group $\pi_1(\uConf_\nn(\CC))$ is isomorphic to the braid group $\mathcal B_\nn$ on $\nn$ strands.
In their seminal paper \cite{seidel2001braid}, Seidel and Thomas proved, by analogy with the symplectic side, that $\mathcal B_\nn$ acts (faithfully) on the $B$-side of the $A_{\n}$ example via \textit{spherical twists} about \textit{spherical objects}.

In \cite{seidel2001braid}, they define an $(A_{\n})$-configuration to be a collection of spherical objects $\Es_1, \dots, \Es_{\n}$ such that
$$
\text{dim}_{\CC} \Hom_{\Db(\XB)}(\Es_i, \Es_j) = 
    \begin{cases}
            1 & |i-j| = 1, \\
            0 & |i-j| \geq 2.
    \end{cases}
$$
Their theorem states that the twists $T_{\Es_i}$ of an $(A_{\n})$-configuration satisfy the braid relations for the braid group $\mathcal B_\nn$ up to graded natural isomorphism:
\begin{align*}
 T_{\Es_i} T_{\Es_{i+1}} T_{\Es_i} = T_{\Es_{i+1}} T_{\Es_i} T_{\Es_{i+1}} \hspace{0.5cm} &\text{for } i = 1, \dots, n-2,\\
 T_{\Es_i} T_{\Es_j} = T_{\Es_j} T_{\Es_i} \hspace{2cm} &\text{for } |i-j| \geq 2.
\end{align*}
This indexing is intentional: the singularity is of type $A_{\n}$, while the braid group is on $\nn$ strands because it comes from the configuration space of the $\nn$ roots of a degree-$\nn$ polynomial.
It turns out that the GIT quotients for the $A_{\n}$ singularity have $(A_{\n})$-configurations of spherical objects.
Say we look at the stacky quotient $\XB_{stack} \cong [\CC^2 /_{1,\n} \ZZ_\nn]$.
We note that $\Pic(\XB_{stack}) \cong \ZZ_\nn$ and consider line bundle twists of the structure sheaf of the origin $\Os_0 (i)$, for $i \in \{0,1,\dots ,\n\}$.

 We claim that $\Os_0(1), \dots, \Os_0(\n)$ form an $(A_{\n})$-configuration.
We can show this by computing $\Hom_{D^b(\XB_{stack})}(\Os_0 (i),\Os_0 (j)) \cong \Hom_{D^b(\XB_{stack})}(\Os_0,\Os_0(j-i))$, where $i,j\in\{1,\dots,\n\}$.
This is the same as computing Ext groups $\Ext^k(\Os_0,\Os_0(j-i))$, and in this case it is sufficient\footnote{The local-to-global Ext spectral sequence
\cite[p.~85]{huybrechts2006fourier} degenerates in this example: the relevant sheaf Ext groups are supported on the residual gerbe over the origin, and over \(\CC\) the stabilizer \(\ZZ_\nn\) is linearly reductive, so taking invariants is exact and there is no higher cohomology.}
to compute local Ext. For the quotient stack $[\CC^2/_{1,\n}\ZZ_\nn]$, the weighted Koszul resolution of $\Os_0$ is:

\begin{equation}
    0\to \Os(-\nn)\to \Os(-1)\oplus \Os(-(\n))\to \Os \to 0,
    \label{eq:Koszul}
\end{equation}
Applying $\mathcal{H}om_{\Coh(\XB_{stack})}( - , \Os_0 (j-i))$ gives the complex:
$$
0 \leftarrow \Os_0 (j-i) \leftarrow \Os_0(j-i+1) \oplus \Os_0(j-i+\n) \leftarrow \Os_0 (j-i+\nn) \leftarrow 0.
$$
We then take global sections, which in this case means taking $\ZZ_{\nn}$-invariant sections (i.e., degree $0$ modulo $\nn$ components).
The case $i=j$ shows that $\Os_0(i)$ are spherical objects. For $|j-i|=1$ we see a 1-dimensional Hom space, and for $|j-i| \geq 2$ there are no sections.
This is consistent with what we observed at the end of the previous section on paths and loops in $\FIPS$.

\begin{rem}
\label{rem:stabilityConditions}
In \cite{thomas2002stability}, Thomas describes a connected component $Y$ of the space of Bridgeland stability conditions on the $A$-side of $A_{\n}$ mirror examples and shows that $Y$ is the universal cover of the configuration space $\uConf_{\nn}(\CC)$ of $n$ distinct points in $\CC$.
This gives a geometric origin for the braid group actions.
\end{rem}

\section{Legendrian lifts of embedded Lagrangian graphs}
\label{sec:stopModuliSpace}
As \cref{tab:FukayaSeidelCategories} indicates, it is technically difficult to address variation of symplectic Landau-Ginzburg potentials directly in the context of homological mirror symmetry.
Instead, we follow the approach that has been the most computationally successful to date and work with the following substitute for the Fukaya-Seidel category: the partially wrapped Fukaya category stopped at a skeleton of a generic superpotential fiber.
In this setup, the data $(\XA, W_\Sigma^{\XA} \colon {\XA} \to \CC)$ is replaced by $(\XA, \stp_\Sigma)$, where $\stp_\Sigma \subset \partial \XA$ is a mostly Legendrian subset (the FLTZ stop) of the contact boundary of $\XA$.
For a class of toric varieties including Fano $\XB$, $\Lambda_\Sigma$ can be realized as the Lagrangian core of a fiber of $W_\Sigma^{\XA}$ in \cite{gammage2022mirror,zhou2018lagrangian} --- see \Cref{fig:coamoeba} for a sketch (see also \cite[Section 2.2]{ganatra2024homologicalmirrorsymmetrybatyrev} for an overview of the symplectic aspects of the construction).
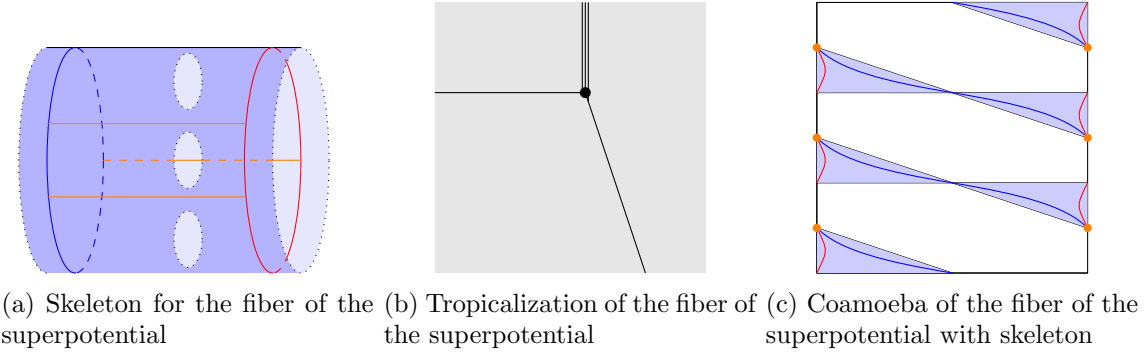
\begin{figure}
   \centering
   \begin{subfigure}{.3\linewidth}
        \centering
        
\begin{tikzpicture}[scale=.75]

    \draw[fill=blue!30,dotted]  (-4,6) ellipse (0.5 and 2);
    \fill[blue!30]  (-4,8) rectangle (0.5,4);
    \begin{scope}[]
    
    \clip  (0,8) rectangle (1.5,4);
    \draw[dashed, red]  (0,6) ellipse (0.5 and 2);
    \end{scope}\begin{scope}[]
    
    \draw[fill=blue!10,dotted]  (0.5,6) ellipse (0.5 and 2);
    \clip (0.5,6) ellipse (0.5 and 2);
    \draw[red]  (0,6) ellipse (0.5 and 2);
    \end{scope}
    \draw[dotted, fill=blue!10]  (-1.5,6) ellipse (0.25 and 0.5);
    \draw[dotted, fill=blue!10]  (-1.5,4.6) ellipse (0.25 and 0.5);
    \draw[dotted, fill=blue!10]  (-1.5,7.4) ellipse (0.25 and 0.5);
    \draw (-4,8) -- (0.5,8) (-4,4);\begin{scope}[]
    
    \clip  (0,8) rectangle (-3.5,4);
    \draw[dashed, blue]  (-3.5,6) ellipse (0.5 and 2);
    \draw[red]  (0,6) ellipse (0.5 and 2);
    \end{scope}\begin{scope}[]
    
    \clip  (-3.5,4) rectangle (-4.5,8);
    \draw[blue]  (-3.5,6) ellipse (0.5 and 2);
    \end{scope}
    \draw[orange] (-4,6.65) -- (-0.5,6.65) (-4,5.35) -- (-0.5,5.35);
    \draw[dashed, orange] (-3,6) -- (-1.75,6);
    \draw[orange] (-1.75,6) -- (-1.25,6);
    \draw[dashed, orange] (-1.25,6) -- (0,6);
    \draw[orange] (0,6) -- (0.5,6);

    \end{tikzpicture}         \caption{Skeleton for the fiber of the superpotential}
   \end{subfigure}
   \begin{subfigure}{.3\linewidth}
        \centering
        \begin{tikzpicture}[scale=.8]
        \fill[gray!20]  (-11.5,8) rectangle (-7,3.5);
        \draw (-11.5,6.5) -- (-9,6.5) -- (-8,3.5);
        \draw (-9.05,8) -- (-9.05,6.5) (-9,8) -- (-9,6.5) (-8.95,8) -- (-8.95,6.35);
        \node[scale=.4, fill=black, circle] at (-9, 6.5) {};
\end{tikzpicture}         \caption{Tropicalization of the fiber of the superpotential}
   \end{subfigure}
   \begin{subfigure}{.3\linewidth}
        \centering
        \begin{tikzpicture}[scale=.6]

\begin{scope}[]
    \begin{scope}[]
    \draw  (3,9.5) rectangle (9,3.5);
    \draw (3,9.5) -- (3,3.5) -- (3,4.5)  -- (6,3.5) (3,6.5) -- (9,4.5) (3,8.5) -- (9,6.5)  (6,9.5)-- (9,8.5);
    \draw (3,7.5) -- (9,7.5) (3,5.5) -- (9,5.5) (3,3.5) -- (9,3.5);
    \begin{scope}[]
    
    \begin{scope}[shift={(0,2)}]
    
    \fill[blue!20] (3,6.5) -- (3,5.5) -- (6,5.5) -- cycle;
    \draw[blue] (3,6.5) .. controls (3.5,6) and (4.5,5.75) .. (6,5.5);
    \draw[red] (3,6.5) .. controls (3.25,6) and (3.25,6) .. (3,5.5);
    \node[scale=.3, fill=orange, circle] at (3,6.5) {};
    \end{scope}
    \begin{scope}[shift={(0,0)}]
    
    \fill[blue!20] (3,6.5) -- (3,5.5) -- (6,5.5) -- cycle;
    \draw[blue] (3,6.5) .. controls (3.5,6) and (4.5,5.75) .. (6,5.5);
    \draw[red] (3,6.5) .. controls (3.25,6) and (3.25,6) .. (3,5.5);
    \node[scale=.3, fill=orange, circle] at (3,6.5) {};
    \end{scope}
    \begin{scope}[shift={(0,-2)}]
    
    \fill[blue!20] (3,6.5) -- (3,5.5) -- (6,5.5) -- cycle;
    \draw[blue] (3,6.5) .. controls (3.5,6) and (4.5,5.75) .. (6,5.5);
    \draw[red] (3,6.5) .. controls (3.25,6) and (3.25,6) .. (3,5.5);
    \node[scale=.3, fill=orange, circle] at (3,6.5) {};
    \end{scope}
    \end{scope}
    \begin{scope}[scale=-1, shift={(-12,-13)}]
    
    \begin{scope}[shift={(0,2)}]
    
    \fill[blue!20] (3,6.5) -- (3,5.5) -- (6,5.5) -- cycle;
    \draw[blue] (3,6.5) .. controls (3.5,6) and (4.5,5.75) .. (6,5.5);
    \draw[red] (3,6.5) .. controls (3.25,6) and (3.25,6) .. (3,5.5);
    \node[scale=.3, fill=orange, circle] at (3,6.5) {};
    \end{scope}
    \begin{scope}[shift={(0,0)}]
    
    \fill[blue!20] (3,6.5) -- (3,5.5) -- (6,5.5) -- cycle;
    \draw[blue] (3,6.5) .. controls (3.5,6) and (4.5,5.75) .. (6,5.5);
    \draw[red] (3,6.5) .. controls (3.25,6) and (3.25,6) .. (3,5.5);
    \node[scale=.3, fill=orange, circle] at (3,6.5) {};
    \end{scope}
    \begin{scope}[shift={(0,-2)}]
    
    \fill[blue!20] (3,6.5) -- (3,5.5) -- (6,5.5) -- cycle;
    \draw[blue] (3,6.5) .. controls (3.5,6) and (4.5,5.75) .. (6,5.5);
    \draw[red] (3,6.5) .. controls (3.25,6) and (3.25,6) .. (3,5.5);
    \node[scale=.3, fill=orange, circle] at (3,6.5) {};
    \end{scope}
    \end{scope}
    
    \end{scope}
    
    \end{scope}

\end{tikzpicture}                         \caption{Coamoeba of the fiber of the superpotential with skeleton}
   \end{subfigure}
   \caption{Homological mirror symmetry for the $A_2$ singularity.
   We have taken the Hori--Vafa potential $W=x(1+y^3)$, and we have drawn the fiber $W^{-1}(1)$.
   The FLTZ skeleton replaces the fiber of the superpotential with a mostly Legendrian stop.}
   \label{fig:coamoeba}
\end{figure}

One could imagine constructing a map
\begin{align*}
\text{skel} \colon \mathcal M^{\XA}_{\JA}\to &\mathcal L(Y)\\
W_\Sigma^\XA\mapsto & \stp_{\Sigma}
\end{align*}
that sends each LG superpotential to a Legendrian skeleton for the superpotential fiber in a consistent manner.
We could then construct a ``local system'' of partially wrapped Fukaya categories over $\mathcal L(Y)$ and subsequently pull it back to a local system of categories over the stringy K\"ahler moduli space or $\FIPS$.
At present, the construction of the Legendrian skeleton depends on many choices, and it is not clear how to make these choices consistently except in the simplest examples. 

We therefore examine $\mathcal L(Y)$ directly, without making reference to $\mathcal M^{\XA}_{\JA}$. Note that $\mathcal L(Y)$ is much larger than $\mathcal M_{J}^{\XA}$ (in the same way that the space of all symplectic forms on $\XB$ is much larger than the set of all K\"ahler forms on $\XB$).

The remainder of this section is dedicated to setting up a method to understand (part of) $\mathcal L(Y)$ in our example of interest by lifting graphs on a cylinder to Legendrians.

\subsection{Space of graph embeddings}
\label{subsec:graph}
We now set up a framework for discussing paths of embedded graphs where the combinatorial type of the graph is allowed to change, as in \Cref{fig:graphIsotopy}.
These will later parameterize certain cell complexes whose top-dimensional cells are Lagrangian or Legendrian submanifolds.
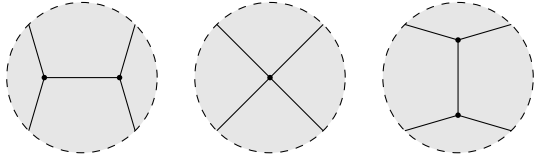
\begin{figure}
\centering
\begin{tikzpicture}
\begin{scope}[]

\draw[dashed, fill=gray!20]  (0,0) ellipse (1 and 1);
\draw (45:1) -- (0,0);
\draw (135:1) -- (0,0);
\draw (225:1) -- (0,0);
\draw (315:1) -- (0,0);
\node[circle, fill=black, scale=.2] at (0,0) {};
\end{scope}
\begin{scope}[shift={(2.5,0)}]

\draw[dashed,  fill=gray!20]  (0,0) ellipse (1 and 1);
\draw (45:1) -- (0,0.5);
\draw (135:1) -- (0,0.5);
\draw (225:1) -- (0,-0.5);
\draw (315:1) -- (0,-0.5);
\node[circle, fill=black, scale=.2] at (0,0.5) {};
\node[circle, fill=black, scale=.2] at (0,-0.5) {};
\draw (0,0.5) -- (0,-0.5);
\end{scope}
\begin{scope}[shift={(-2.5,0)}]

\draw[dashed, fill=gray!20]  (0,0) ellipse (1 and 1);
\draw (45:1) -- (0.5,0);
\draw (135:1) -- (-0.5,0);
\draw (225:1) -- (-0.5,0);
\draw (315:1) -- (0.5,0);
\draw (-0.5,0) -- (0.5,0);
\node[circle, fill=black, scale=.2] at (-0.5,0) {};
\node[circle, fill=black, scale=.2] at (0.5,0) {};
\end{scope}
\end{tikzpicture} \caption{An isotopy of embedded graphs giving a path in $\mathcal N(M)$.}
\label{fig:graphIsotopy}
\end{figure}

\begin{df} \label{df:graphembed}
Let $M$ be a manifold and let $G$ be a finite graph, regarded as a CW complex.
Choose a parametrization of each edge $e$ by $I=[0,1]$, and write $\underline e\colon I\to M$ for the restriction of a continuous map $\phi\colon G\to M$ to that edge.
Let $E_0(\phi)\subset E(G)$ be the set of edges on which $\underline e$ is constant, and let $\sim_{\phi}$ be the equivalence relation on $V(G)$ generated by the endpoints of the edges in $E_0(\phi)$.
A degenerate graph embedding is such a map $\phi\colon G\to M$ satisfying the following properties:
 \begin{itemize}
   \item Each edge restriction $\underline e$ is either an embedding or the constant map;
   \item for vertices $v,w\in V(G)$, one has $\phi(v)=\phi(w)$ if and only if $v\sim_{\phi} w$;
   \item if $\underline e$ is non-constant, then $\underline e((0,1))$ is disjoint from $\phi(V(G))$;
   \item if $e\neq f$ are non-constant edges, then $\underline e((0,1))\cap \underline f((0,1))=\varnothing$; and
   \item no cycle is completely collapsed: for every cycle $c=e_1, e_2, \ldots, e_k$, at least one map $\underline{e_i}$ is non-constant.
 \end{itemize}
When $\dim(M)=2$, we additionally require that no non-constant edge is self-adjacent: if $x$ lies in the interior of a non-constant edge, then a sufficiently small disk $D_x\subset M$ has $D_x\setminus\Im(\phi)$ with two components, and these two local components lie in distinct connected components of $M\setminus\Im(\phi)$.
\end{df}

Equivalently, a degenerate graph embedding is an ordinary graph embedding after contracting the constant edges, with the additional requirement that no cycle of $G$ is entirely contracted.
The set $\{\phi \colon G\to M\}$ of degenerate graph embeddings can be equipped with the subspace topology by considering it as a subset of all continuous maps from the CW complex of $G$ to $M$.
We then denote by 
\[\mathcal N(M):=\left(\bigsqcup_{G}\operatorname{DegEmb}(G,M)\right)\big/ \sim\]
where $\operatorname{DegEmb}(G,M)$ is the smooth space of degenerate graph embeddings of fixed combinatorial type, and two degenerate graph embeddings $\phi\colon G\to M$ and $\phi'\colon G'\to M$ are declared equivalent if their images agree in $M$.\footnote{In fact, after picking a metric on $M$, each equivalence class has a preferred representative whose domain is minimal under the relation of edge contraction and where every edge has a constant speed parameterization.} 
To avoid excessive notation, we will from now on write $\phi\in \mathcal N(M)$ to refer to the equivalence class given by some choice of parameterization $\phi \colon G\to M$.
We now restrict to the setting where $\dim(M)=2$, $M$ is compact (possibly with boundary), and we have chosen an area form $\omega_M\in \Omega^2(M)$.
The faces of a degenerate graph embedding $\phi \colon G\to M$ are the connected components of the complement of the image of $\phi$ and are denoted by $F(\phi)$.
Note that we do not require our faces to be simply connected.
We let $\mathcal N_{1}(M)$ denote the moduli space of degenerate graph embeddings with 1 marked point.

We will use this quotient model throughout.
Concretely, when we say that a map $K\to \mathcal N(M)$ is continuous, we mean that every point of $K$ has a neighborhood on which the map can be represented by a continuous family in some $\operatorname{DegEmb}(G,M)$ after passing, if necessary, to a common subdivision of the source graphs.
Thus all continuity statements below are checked locally on fixed combinatorial models and then descend to $\mathcal N(M)$ by the continuity of the quotient map.
To establish further notation, we will write $\phi^t \colon G^t\to M$ to denote a path $I\to \mathcal N(M)$, understanding that the combinatorial type of $G^t$ may vary with the parameter $t$.
Given an element $\phi\in \mathcal N(M)$, let $\mathcal N_\phi(M)$ (resp. $\mathcal N_{\phi,1}(M)$) denote the pointed space corresponding to the connected component of  $\mathcal N(M)$ (resp. $\mathcal N_1(M)$) containing the embedding $\phi$.

As we forbid collapsing of cycles, given an isotopy of embedded graphs $\phi^t \colon G^t\to M$, for each $t\in I$ we obtain an induced map on the set of faces $F(\phi^t) \colon F(\phi^0)\to F(\phi^t)$.
We say that the isotopy is area-preserving if 
    \[\Area(f)= \Area(F(\phi^t)(f))\]   
for all $f\in F(\phi^0)$ and $t\in I$.
This allows us to specialize the moduli space of graphs to those which preserve area: for a fixed $\phi \colon G\to M$, let $\mathcal N_{\phi,1}^{\Area}(M)$ denote the moduli space of pointed graph embeddings $\psi$ for which there is an area-preserving isotopy from $\phi$ to $\psi$.

We will now further restrict to a particular set of examples.
From here on, let $M= [R^-, R^+]\times S^1$ with coordinates $(r,q_2)$, where $R^-\ll 0 \ll R^+$.
Equip $M$ with the standard area form $dr \wedge dq_2$, and let $\phi_\nn \colon (G_\nn, \bullet)\to M$ denote the embedded graph with vertical lines at $r=0, r=\nn$, and $n$ evenly spaced horizontal line segments $[0, \nn]\times \{k/\nn\}$ for $k=0, \dots, \n$ (see \Cref{fig:graphGn} for this graph when $n=4$).
This graph has $n$ simply connected faces, each of area one.

\begin{figure}[!h]
    \centering
    \begin{tikzpicture}[xscale=-1.5, yscale=1.5,  rotate=90]

\fill[gray!20]  (-3,-0.5) rectangle (1,-4.5);
\draw[red] (-3,-1.5) -- (1,-1.5) ;
\draw[blue] (-3,-2.5) -- (1,-2.5);
\draw[orange] (-3,-1.5) -- (-3,-2.5) (-2,-1.5) -- (-2,-2.5) (-1,-1.5) -- (-1,-2.5) (0,-1.5) -- (0,-2.5);

\draw[>>->>] (-3,-5) -- node[below, rotate=270] {$q_2\in S^1$} (1,-5);
\draw (-3.5,-0.5) -- node[below]{ $r\in [-R, R]$} (-3.5,-4.5);
\node[circle, fill=black, scale=.4] at (-3,-2.5) {};
\end{tikzpicture}     \caption{The basic graph embedding $\phi_\nn \colon G_\nn\to M$ when $n=4$. We have also indicated a marked point. The area of each closed square is one.}
    \label{fig:graphGn}
\end{figure}
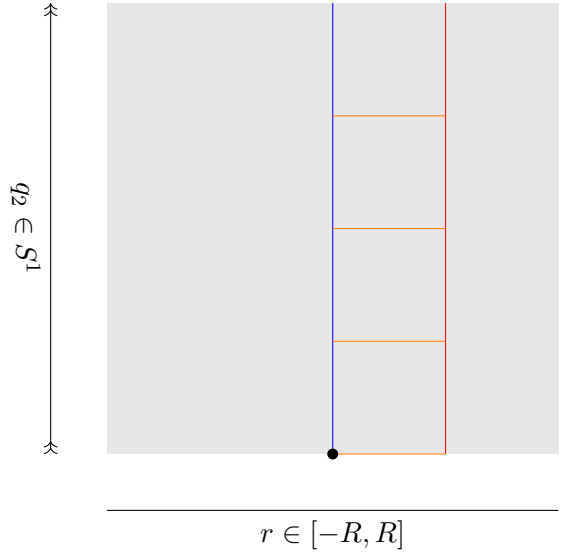

\begin{prop}
Consider the configuration space $\uConf_\nn(M)$ of $n$ unlabeled points in $M$.
There is an inclusion
\[\pi_1(\uConf_\nn(M))\into \pi_1(\mathcal N^{\Area}_{\phi_\nn}(M)).\]
    \label{prop:subgroup}
\end{prop}
\begin{proof}
The proof is based on the following simpler map.
Consider the subset $\mathcal N_{\phi_\nn}^{c}(M)$ of $\mathcal N_{\phi_\nn}(M)$ where every simply connected face is convex with respect to the flat metric on $M$. 
Then we have maps
\begin{equation}\label{eq:centVor}
\begin{tikzcd}
   & \uConf_\nn(M)\\
   \mathcal N_{\phi_\nn}^{c}(M)\arrow{ur}{\Cent} & \uConf_\nn(M) \arrow{l}{\Vor^0} \arrow{u}{\id} \\
   \end{tikzcd}
\end{equation}
where $\Cent$ sends an embedded graph with convex simply connected faces to the centroids of its simply connected faces, and $\Vor^0$ sends a configuration of points to the corresponding Voronoi diagram.
The triangle commutes up to homotopy: observe that the linear interpolation
\begin{align*}
   H^t :\uConf_\nn(M)\times I \to& \uConf_\nn(M) \\
    ((x_1, \ldots, x_\nn),t)\mapsto& \bigl(t(\Cent\circ \Vor^0)_i +(1-t)x_i\bigr)_{i=1}^\nn
\end{align*}
is well defined as $H^t(x)_i$ always belongs to the face $f_i$ due to the convexity of the faces, and therefore $H^t(x)_i=H^t(x)_j$ implies $i=j$.
To extend this proof to the moduli space of planar graphs with faces of area one, we fit \cref{eq:centVor} as face $A$ of the following diagram, which contains maps we have yet to define.
\begin{equation}
 \label{eq:extendedVorCent}
 \begin{tikzpicture}

    \node (NphiM) at (-2,8.5) {$\mathcal N_{\phi_\nn}(M)$};
    \node (UConfnM) at (6,8.5) {$\uConf_\nn(M)$};
    \node (NphicM) at (1.5,5) {$\mathcal N_{\phi_\nn}^{c}(M)$};
    \node (UConfnnM_R2C3) at (6,5) {$\uConf_\nn(M)$};
    \node (NphiAreaM) at (-2,-0.5) {$\mathcal N_{\phi_\nn}^{\Area}(M)$};
    \node (NphiCAreaM) at (1.5,-0.5) {$\mathcal N_{\phi_\nn}^{c,\Area}(M)$};
    \node (UConfnnM_R4C3) at (6,-0.5) {$\uConf_\nn(M)$};

    \draw[->] (NphiM) -- node[fill=white] {$\Cent^+$} (UConfnM);
    \draw[->] (NphicM) -- node[fill=white] {$\Cent$} (UConfnM);
    \draw[->] (NphicM) -- (NphiM);
    \draw[<-] (NphicM) -- node[fill=white] {$\Vor^0$} (UConfnnM_R2C3);
    \draw[->] (UConfnnM_R2C3) -- node[fill=white] {$\id$} (UConfnM);
    \draw[->] (NphiAreaM) -- node[fill=white] {$j$} (NphiM);
    \draw[->] (NphiCAreaM) -- (NphicM);
    \draw[->] (NphiCAreaM) -- node[fill=white] {$i$} (NphiAreaM);
    \draw[->] (UConfnnM_R4C3) -- node[fill=white] {$\id$} (UConfnnM_R2C3);
    \draw[<-] (NphiCAreaM) -- node[fill=white] {$\Vor^{w_x}$} (UConfnnM_R4C3);
    \node at (4.5,6.5) {$A$};
    \node at (1.5,7) {$C$};
    \node at (4,2) {$B$};
\end{tikzpicture}
\end{equation}
We will prove that all the inner faces of the diagram commute up to homotopy.
The proposition follows from the homotopy commutativity of the outer face.
The key tool is a generalization of the Voronoi diagram.
Let $W= \RR^\nn_{\geq 0}\times \RR^2_{\geq 0}$ be a weight space.
The power diagram\footnote{Also called the Laguerre-Voronoi diagram, Dirichlet cell complex, radical Voronoi tessellation or sectional Dirichlet tessellation.} determined by weights $w\in W$
\[
\Vor^w:\uConf_\nn(M)\to \mathcal N(M)
\]
is defined in the following way.
The assignment $\Vor^w(x)$ is the graph whose faces $f_i$ are characterized by
\[
f_i^\circ=\left\{x\in M\ \middle|\ 
\begin{array}{l}
\forall j\neq i,\ d(x, x_i)^2 - w_i\leq d(x, x_j)^2-w_j,\\[2pt]
d(x,x_i)^2-w_i\leq d(x,\{R^+\}\times S^1)^2-w_N,\\[2pt]
d(x,x_i)^2-w_i\leq d(x,\{R^-\}\times S^1)^2-w_S
\end{array}
\right\}.
\]
These generalize Voronoi diagrams in the sense that when $w=0$, we recover the standard Voronoi construction.
For each $x\in \uConf_\nn(M)$, define $W^+_{adm}(x)$ to be the set of weights so that every face of $w$-power diagram has positive area. 
\begin{lemma}
The space $W^+_{adm}(x)$ is convex. As a consequence, for any $w_0, w_1\in W^+_{adm}(x)$, the path $w_s:=(1-s)w_0+sw_1$ gives a path $\Vor^{w_s}(x)\in \mathcal N_{\phi_\nn}^c(M)$.
\end{lemma}
\begin{proof}
Lift the cylinder to its universal cover and choose, for each labeled cell of the two endpoint power diagrams, a lift of an interior point belonging to the corresponding lifted site.  For a fixed choice of lifted sites, each inequality defining a power cell is affine in the spatial variable and in the weight vector, since the quadratic terms in the squared distances cancel.  If $y_0$ and $y_1$ are interior points of the same labeled cell for weights $w_0$ and $w_1$, then $y_s=(1-s)y_0+sy_1$ satisfies the strict inequalities for $w_s=(1-s)w_0+sw_1$.  Hence the labeled cell has non-empty interior, and therefore positive area, for every $s\in[0,1]$.  This argument applies to each labeled face, including the two boundary cells, so $w_s\in W^+_{adm}(x)$ throughout the segment.
\end{proof}
Observe that all maps $\Vor^w$ are homotopic via linear interpolation of the $w$-coordinate.
We are now ready to describe the objects and morphisms in this diagram:
\begin{itemize}
     \item The space $\mathcal N_{\phi_\nn}^{c,\Area}(M)$ is the space of degenerately embedded graphs with convex simply connected faces of area one.
          Given $x\in \uConf_\nn(M)$, we solve the semi-discrete optimal transport problem by yielding a transport map $T\colon M\to M$ that minimizes the transport cost between the measure (dependent on $x$)
          \[
     \frac{1}{2}(R^+-R^- - \nn)(\chi_{\{R^-\}\times S^1}+\chi_{\{R^+\}\times S^1}) + \sum_{i=1}^\nn \delta_{x_i}
     \]
     and the standard area form $dr\wedge dq_2$. Here $\delta_{x_i}$ is the Dirac measure, and $\chi_{\{R^\pm\}\times S^1}$ is the measure which integrates over the boundary circles.
     By general principles of the semi-discrete optimal transport problem \cite{cuesta1993characterization}, the solution is unique (up to a set of measure zero).
     Furthermore, by \cite{aurenhammer1998minkowski}, the regions $T^{-1}(x_i)$ correspond (up to a set of measure zero) to the faces of a power diagram $\Vor^{w_x}$ for the subsets $\{x_i, \{R^-\}\times S^1, \{R^+\}\times S^1\}$ for a certain set of weights $w_x=\{w_i, w_N, w_S\}$, which also depend continuously on the $x_i$.
     In this sense, the weighted Voronoi graph is the minimizing planar graph selected by the transport problem.
     It is immediate from the construction of the power diagram that the diagram varies continuously in both the weight parameters $w_x$ and points $x_i$.
     We define the map $\Vor^{w_x} \colon \uConf_\nn(M)\to \mathcal N_{\phi_\nn}^{c,\Area}(M)$ to assign to each point configuration this power diagram.
     As noted previously, linearly interpolating the weights $w_x \to 0$ gives homotopy commutativity of the square $B$.
     \item The extension theorem for centers of Jordan domains from \cite[Theorem 1.1]{belegradek2025point} lets us extend the centroid choice on convex faces to a center choice on arbitrary Jordan faces in this component, giving a map
     \[
     \Cent^+ \colon \mathcal N_{\phi_\nn}(M)\to \uConf_\nn(M)
     \]
     with the property that $\Cent^{+}(\phi)\subset M\setminus \Im(\phi)$, i.e., the points belong to the faces.
     Furthermore, this map extends the centroid map in the sense that it agrees with the centroid map when restricted to $\mathcal N_{\phi_\nn}^{c}(M)$.
     This gives us commutativity of $C$.
\end{itemize}
The unlabeled face commutes by construction.

After picking basepoints and applying $\pi_1$ to the outer face of \cref{eq:extendedVorCent}, we obtain
\[
\Cent^+_*\circ j_*\circ i_*\circ \Vor^{w_x}_*=\id_{\pi_1(\uConf_\nn(M))},
\]
so $j_*\circ i_*\circ \Vor^{w_x}_*$ is injective.  
\end{proof}

\begin{rem}
In fact, it is possible to prove a homotopy equivalence between $\mathcal N_{\phi_\nn}(M)$ and the smaller space $\mathcal N_{\phi_\nn}^{\Area}(M)$. However, we do not need it in our construction.
\end{rem}
It will be practical to identify the subgroup $\pi_1(\uConf_\nn(M))\subset \pi_1(\mathcal N^{\Area}_{\phi_\nn}(M))$. 
Let $\tau_1\in \pi_1(\mathcal N_{\phi_\nn}^{\Area}(M))$ denote the loop drawn in \cref{fig:graphLoop} and let $\rho\in \pi_1(\mathcal N_{\phi_\nn}^{\Area}(M))$ be the loop induced by a $1/n$-rotation of the $q_2$ coordinate. Denote by $\tau_k=\rho^{k-1} \tau_1 \rho^{-(k-1)}$ (corresponding to the twist of the $k$-th and $(k+1)$-th faces).
\begin{figure}[!h]
    \centering
    \begin{subfigure}{.9\linewidth}
    \centering
    \scalebox{.3}{
    \begin{tikzpicture}[xscale=-1.5, yscale=1.5,  rotate=90]
\fill[gray!20]  (-3,-0.5) rectangle (1,-4.5);
\draw[red] (-3,-1.5) -- (1,-1.5) ;
\draw[blue] (-3,-2.5) -- (1,-2.5);
\draw[orange] (-3,-1.5) -- (-3,-2.5) (-2,-1.5) -- (-2,-2.5) (-1,-1.5) -- (-1,-2.5) (0,-1.5) -- (0,-2.5);
\end{tikzpicture}     \begin{tikzpicture}[xscale=-1.5, yscale=1.5,  rotate=90]

\fill[gray!20]  (-3,-0.5) rectangle (1,-4.5);
\draw[red] (-3,-1.5) -- (1,-1.5) ;
\draw[blue] (-3,-2.5) -- (1,-2.5);
\draw[orange] (-3,-1.5) -- (-3,-2.5) (-1.5,-1.5) -- (-2.5,-2.5) (-1,-1.5) -- (-1,-2.5) (0,-1.5) -- (0,-2.5);
\end{tikzpicture}     \begin{tikzpicture}[xscale=-1.5, yscale=1.5,  rotate=90]

\fill[gray!20]  (-3,-0.5) rectangle (1,-4.5);
\draw[red] (-3,-1.5) -- (1,-1.5) ;
\draw[blue] (-3,-2.5) -- (1,-2.5);
\draw[orange] (-3,-1.5) -- (-3,-2.5) (-1,-1.5) -- (-3,-2.5) (-1,-1.5) -- (-1,-2.5) (0,-1.5) -- (0,-2.5);
\end{tikzpicture}     \begin{tikzpicture}[xscale=-1.5, yscale=1.5,  rotate=90]

\fill[gray!20]  (-3,-0.5) rectangle (1,-4.5);
\draw[red] (-3,-1.5) -- (1,-1.5) ;
\draw[blue] (-3,-2.5) -- (1,-2.5);
\draw[orange] (-3,-1.5) -- (-3,-2.5) (-1,-2) -- (-3,-2) (-1,-1.5) -- (-1,-2.5) (0,-1.5) -- (0,-2.5);
\end{tikzpicture}     \begin{tikzpicture}[xscale=-1.5, yscale=1.5,  rotate=90]

\fill[gray!20]  (-3,-0.5) rectangle (1,-4.5);
\draw[red] (-3,-1.5) -- (1,-1.5) ;
\draw[blue] (-3,-2.5) -- (1,-2.5);
\draw[orange] (-3,-1.5) -- (-3,-2.5) (-1,-2.5) -- (-3,-1.5) (-1,-1.5) -- (-1,-2.5) (0,-1.5) -- (0,-2.5);
\end{tikzpicture}     \begin{tikzpicture}[xscale=-1.5, yscale=1.5,  rotate=90]

\fill[gray!20]  (-3,-0.5) rectangle (1,-4.5);
\draw[red] (-3,-1.5) -- (1,-1.5) ;
\draw[blue] (-3,-2.5) -- (1,-2.5);
\draw[orange] (-3,-1.5) -- (-3,-2.5) (-1.5,-2.5) -- (-2.5,-1.5) (-1,-1.5) -- (-1,-2.5) (0,-1.5) -- (0,-2.5);
\end{tikzpicture}     \begin{tikzpicture}[xscale=-1.5, yscale=1.5,  rotate=90]
\fill[gray!20]  (-3,-0.5) rectangle (1,-4.5);
\draw[red] (-3,-1.5) -- (1,-1.5) ;
\draw[blue] (-3,-2.5) -- (1,-2.5);
\draw[orange] (-3,-1.5) -- (-3,-2.5) (-2,-1.5) -- (-2,-2.5) (-1,-1.5) -- (-1,-2.5) (0,-1.5) -- (0,-2.5);
\end{tikzpicture}     }
    \caption{A loop in $\mathcal N^{\Area}_{\phi_4}(M)$ corresponding to a generator of the braid group, $\tau_1$.}
       \label{fig:graphLoop}
\end{subfigure}
\begin{subfigure}{.9\linewidth}
    \centering
    \scalebox{.3}{
            \begin{tikzpicture}[xscale=-1.5, yscale=1.5,  rotate=90]

    \fill[gray!20]  (-3,-0.5) rectangle (1,-4.5);
    \node[circle, fill=black, scale=.4] at (-3,-2.5) {};
    \clip  (1,-3) rectangle (-3,-1);
    \begin{scope}[]
    \draw[red] (-4,-1.5) -- (1,-1.5) ;
    \draw[blue] (-4,-2.5) -- (1,-2.5);
    \draw[orange] (-3,-1.5) -- (-3,-2.5) (-2,-1.5) -- (-2,-2.5) (-1,-1.5) -- (-1,-2.5) (0,-1.5) -- (0,-2.5);
    \draw[orange] (-4,-2.5) -- (-4,-1.5);
    \end{scope}

    \end{tikzpicture}     \begin{tikzpicture}[xscale=-1.5, yscale=1.5,  rotate=90]

\fill[gray!20]  (-3,-0.5) rectangle (1,-4.5);
\node[circle, fill=black, scale=.4] at (-3,-2.5) {};
\clip  (1,-3) rectangle (-3,-1);
\begin{scope}[shift={(0.166,0)}]
\draw[red] (-4,-1.5) -- (1,-1.5) ;
\draw[blue] (-4,-2.5) -- (1,-2.5);
\draw[orange] (-3,-1.5) -- (-3,-2.5) (-2,-1.5) -- (-2,-2.5) (-1,-1.5) -- (-1,-2.5) (0,-1.5) -- (0,-2.5);
\draw[orange] (-4,-2.5) -- (-4,-1.5);
\end{scope}

\end{tikzpicture}     \begin{tikzpicture}[xscale=-1.5, yscale=1.5,  rotate=90]

\fill[gray!20]  (-3,-0.5) rectangle (1,-4.5);
\node[circle, fill=black, scale=.4] at (-3,-2.5) {};
\clip  (1,-3) rectangle (-3,-1);
\begin{scope}[shift={(0.333,0)}]
\draw[red] (-4,-1.5) -- (1,-1.5) ;
\draw[blue] (-4,-2.5) -- (1,-2.5);
\draw[orange] (-3,-1.5) -- (-3,-2.5) (-2,-1.5) -- (-2,-2.5) (-1,-1.5) -- (-1,-2.5) (0,-1.5) -- (0,-2.5);
\draw[orange] (-4,-2.5) -- (-4,-1.5);
\end{scope}

\end{tikzpicture}     \begin{tikzpicture}[xscale=-1.5, yscale=1.5,  rotate=90]

\fill[gray!20]  (-3,-0.5) rectangle (1,-4.5);
\node[circle, fill=black, scale=.4] at (-3,-2.5) {};
\clip  (1,-3) rectangle (-3,-1);
\begin{scope}[shift={(0.5,0)}]
\draw[red] (-4,-1.5) -- (1,-1.5) ;
\draw[blue] (-4,-2.5) -- (1,-2.5);
\draw[orange] (-3,-1.5) -- (-3,-2.5) (-2,-1.5) -- (-2,-2.5) (-1,-1.5) -- (-1,-2.5) (0,-1.5) -- (0,-2.5);
\draw[orange] (-4,-2.5) -- (-4,-1.5);
\end{scope}

\end{tikzpicture}     \begin{tikzpicture}[xscale=-1.5, yscale=1.5,  rotate=90]

\fill[gray!20]  (-3,-0.5) rectangle (1,-4.5);
\node[circle, fill=black, scale=.4] at (-3,-2.5) {};
\clip  (1,-3) rectangle (-3,-1);
\begin{scope}[shift={(0.6666,0)}]
\draw[red] (-4,-1.5) -- (1,-1.5) ;
\draw[blue] (-4,-2.5) -- (1,-2.5);
\draw[orange] (-3,-1.5) -- (-3,-2.5) (-2,-1.5) -- (-2,-2.5) (-1,-1.5) -- (-1,-2.5) (0,-1.5) -- (0,-2.5);
\draw[orange] (-4,-2.5) -- (-4,-1.5);
\end{scope}

\end{tikzpicture}     \begin{tikzpicture}[xscale=-1.5, yscale=1.5,  rotate=90]

\fill[gray!20]  (-3,-0.5) rectangle (1,-4.5);
\node[circle, fill=black, scale=.4] at (-3,-2.5) {};
\clip  (1,-3) rectangle (-3,-1);
\begin{scope}[shift={(0.83333,0)}]
\draw[red] (-4,-1.5) -- (1,-1.5) ;
\draw[blue] (-4,-2.5) -- (1,-2.5);
\draw[orange] (-3,-1.5) -- (-3,-2.5) (-2,-1.5) -- (-2,-2.5) (-1,-1.5) -- (-1,-2.5) (0,-1.5) -- (0,-2.5);
\draw[orange] (-4,-2.5) -- (-4,-1.5);
\end{scope}

\end{tikzpicture}             \begin{tikzpicture}[xscale=-1.5, yscale=1.5,  rotate=90]

    \fill[gray!20]  (-3,-0.5) rectangle (1,-4.5);
    \node[circle, fill=black, scale=.4] at (-3,-2.5) {};
    \clip  (1,-3) rectangle (-3,-1);
    \begin{scope}[]
    \draw[red] (-4,-1.5) -- (1,-1.5) ;
    \draw[blue] (-4,-2.5) -- (1,-2.5);
    \draw[orange] (-3,-1.5) -- (-3,-2.5) (-2,-1.5) -- (-2,-2.5) (-1,-1.5) -- (-1,-2.5) (0,-1.5) -- (0,-2.5);
    \draw[orange] (-4,-2.5) -- (-4,-1.5);
    \end{scope}

    \end{tikzpicture}     }
    \caption{A loop in $\mathcal N^{\Area}_{\phi_4}(M)$ corresponding to the generator $\rho$.}
       \label{fig:graphLoop2}
\end{subfigure}
\caption{Generators of the braid group via variation of embedded graphs.}
\end{figure}
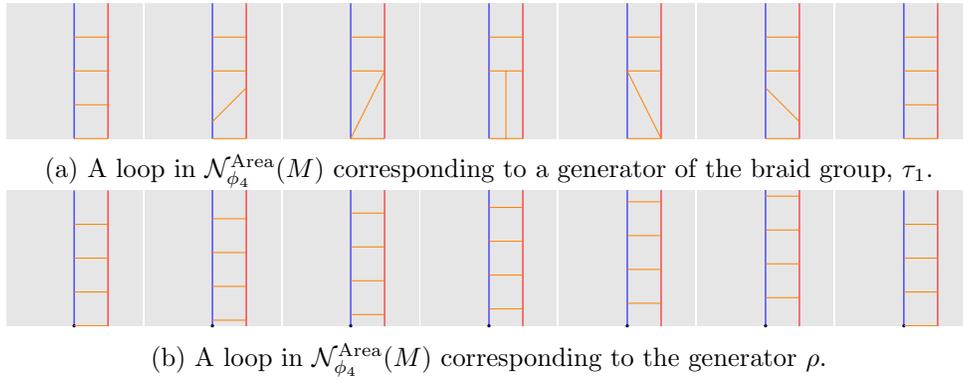

\begin{prop}
\label{prop:Bn1AsConf}
    The elements $\tau_1, \ldots, \tau_{\n}$ and $\rho$ lie in $\pi_1(\uConf_{\nn}(M)) \subset \pi_1(\mathcal N^{\Area}_{\phi_{\nn}}(M))$.
        Moreover, these are identified with the usual generators of $\pi_1(\uConf_\nn(M))\cong \mathcal B_{\nn,1}$,
    where $\mathcal B_{\nn,1}$ is the braid group on $\nn$ uncolored strands and one colored strand, and $\tau_\nn=\rho^{\n}\tau_1\rho^{-(\n)}$ is the cyclic translate used to keep the annular indexing symmetric.
    \end{prop}
\begin{proof}
Let $M^\circ=(R^-,R^+)\times S^1$ be the interior of the cylinder.
Choose a collar of $\partial M$ and an isotopy of embeddings $r_t:M\to M$, $t\in[0,1]$, with $r_0=\id_M$ and $r_t(M)\subset M^\circ$ for every $t>0$.
Applying $r_t$ pointwise gives a deformation retraction of unordered configuration spaces, so the inclusion
\[
\uConf_\nn(M^\circ)\longrightarrow \uConf_\nn(M)
\]
is a homotopy equivalence.
Choose a homeomorphism $h:M^\circ\stackrel{\sim}{\longrightarrow}\RR^2\setminus\{0\}$.
This induces a homeomorphism
\[
\uConf_\nn(M^\circ)\cong \uConf_\nn(\RR^2\setminus\{0\}).
\]
Now let $\Conf_{\nn,1}(\RR^2)$ denote the configuration space of $\nn$ unlabeled points together with one distinguished colored point in the plane, and let
\[
p:\Conf_{\nn,1}(\RR^2)\to \RR^2
\]
record the colored point.
Since the base $\RR^2$ is contractible, the inclusion of any fiber is a homotopy equivalence.
Using the fiber over $0$, we get
\[
\uConf_\nn(M)\simeq \uConf_\nn(M^\circ)\cong \uConf_\nn(\RR^2\setminus\{0\})\simeq \Conf_{\nn,1}(\RR^2),
\]
and therefore
\[
\pi_1(\uConf_\nn(M))\cong \pi_1(\Conf_{\nn,1}(\RR^2))\cong \mathcal B_{\nn,1}.
\]

To identify generators, choose the base configuration in $M$ so that under $h$ it becomes $\nn$ equally spaced points on a small circle around the origin.
Then the loop $\tau_i$ becomes the standard half-twist exchanging the $i$-th and $(i+1)$-st neighboring points inside a disk disjoint from the puncture, while $\rho$ becomes the rigid rotation of the entire configuration by angle $2\pi/\nn$ around the puncture.
In \cite[Proposition 1]{MANFREDINI1997123}, these are exactly the standard generators for the annular braid group: the half-twists $b_i$ for $1\le i<\nn$ together with the annular rotation.
The subgroup generated by $\tau_1,\dots,\tau_{\n}$ is therefore the standard copy of $\mathcal B_{\nn}$ inside $\mathcal B_{\nn,1}$, and adjoining the cyclic translate $\tau_\nn$ does not enlarge this subgroup.
The extra loop $\rho$ records the motion of the uncolored strands around the distinguished colored strand.
\end{proof}

\subsection{Lagrangian projection for mostly Legendrian subsets}
In this section, we fix notation for Legendrian submanifolds and describe their front and Lagrangian projections, which are useful for 2-dimensional visualizations.
We then use Lagrangian projections to lift certain degenerate embeddings of graphs to Legendrians.

Let $(q_1, q_2)$ be standard flat coordinates on $T^2$ and let $Y$ be the contact manifold given by the unit cosphere bundle $S^*T^2\subset T^*T^2$ with respect to the standard metric on $T^2$.
Then $Y\cong T^2\times S^1$, which we give coordinates $(q_1, q_2, \theta)$, where $\theta$ is the coordinate on the $S^1$ factor measured as the angle against the covector $dq_1$.
The contact form on this space is  
\[
\lambda_{std}:=\cos(\theta)dq_1 + \sin(\theta)dq_2.
\]
A 1-submanifold $\Leg\subset Y$ is called Legendrian if $\lambda_{std}|_\Leg=0$.
A mostly 1-Legendrian subset (\cite[Definition 1.7]{ganatra2024sectorial}) is a degenerate graph embedding $\Leg \colon G\to Y$ where every non-constant edge is a Legendrian 1-submanifold.
\begin{df}
\label{def:legModuliSpace}
   Let $\mathcal L_\Leg(Y)$ denote the subspace of 
    $\mathcal N_\Leg(Y)$ consisting of those graphs which are mostly Legendrian subsets of $Y$.
\end{df}
If $(\dot q_1, \dot q_2)\neq (0,0)$ along a Legendrian submanifold of $Y$, the Legendrian can be conveniently described by its projection to the base $T^2$.
This is called the \emph{front projection}.
One recovers the Legendrian $\Leg$ from its front projection $p_{front}(\Leg)\subset T^2$ by specifying an orientation on each edge of the projection;
the Legendrian condition then uniquely determines the lift (see \eqref{eq:Legendrian}).
Suggestively, we draw front projections as (possibly immersed) graphs in $T^2$ with ``hairs'' indicating the cooriented unit conormals, as in \Cref{fig:frontProjectionI}.
If $(\dot q_1, \dot q_2)=(0,0)$ along an edge, we draw multiple covectors to capture the $\theta$-component of the parameterization as in \cref{fig:frontProjectionII}.
This projection is frequently used in the mirror symmetry literature (see, for instance, \cite[Figure 2]{fang2014coherent}).
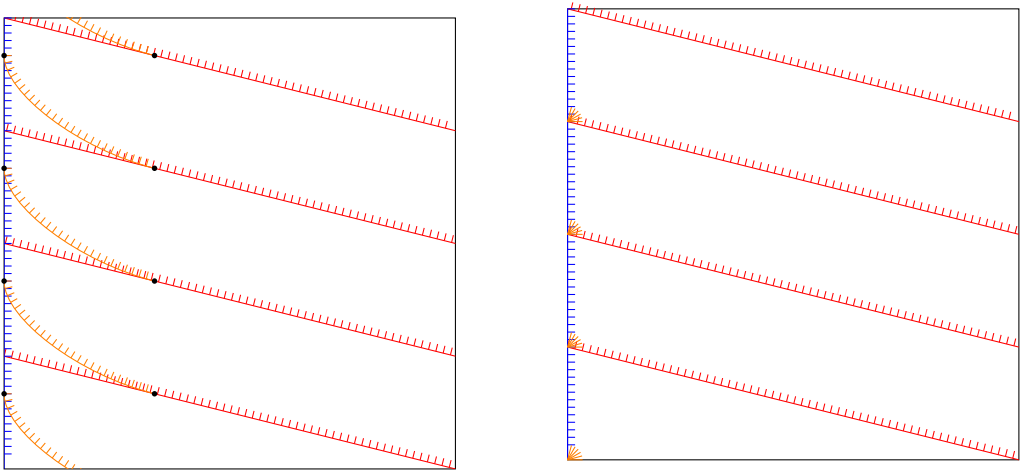
\begin{figure}
   \centering
   \begin{subfigure}{.45\linewidth}
   \centering
   
        \begin{tikzpicture}
    \begin{scope}[]

    \draw  (-3,-3) rectangle (3,3);
    \clip  (-3,-3) rectangle (3,3);
    \draw[red, fuzz] (-3,-1.5) -- (3,-3) (-3,0) -- (3,-1.5) (-3,1.5) -- (3,0) (-3,3) -- (3,1.5);
    \draw[blue, fuzz] (-3,3) -- (-3,-3);

    \begin{scope}[shift={(-3,0)}]

    \draw[fuzz, orange] (0,1) .. controls (0,0.5) and (1,-0.25) .. (2,-0.5);
    \end{scope}

    \begin{scope}[shift={(-3,-3)}]

    \draw[fuzz, orange] (0,1) .. controls (0,0.5) and (1,-0.25) .. (2,-0.5);
    \end{scope}

    \begin{scope}[shift={(-3,3)}]

    \draw[fuzz, orange] (0,1) .. controls (0,0.5) and (1,-0.25) .. (2,-0.5);
    \end{scope}

    \begin{scope}[shift={(-3,1.5)}]

    \draw[fuzz, orange] (0,1) .. controls (0,0.5) and (1,-0.25) .. (2,-0.5);
    \end{scope}

    \begin{scope}[shift={(-3,-1.5)}]

    \draw[fuzz, orange] (0,1) .. controls (0,0.5) and (1,-0.25) .. (2,-0.5);
    \end{scope}

    \end{scope}

    \node[circle, fill=black, scale=.2] at (-1,-2) {};
    \node[circle, fill=black, scale=.2] at (-3,-0.5) {};
    \node[circle, fill=black, scale=.2] at (-1,-0.5) {};
    \node[circle, fill=black, scale=.2] at (-3,1) {};
    \node[circle, fill=black, scale=.2] at (-1,1) {};
    \node[circle, fill=black, scale=.2] at (-3,2.5) {};
    \node[circle, fill=black, scale=.2] at (-1,2.5) {};
    \node[circle, fill=black, scale=.2] at (-3,-2) {};
    \end{tikzpicture}    \caption{The front projection of a mostly 1-Legendrian of $Y_0$, with vertices of the graph indicated by black nodes.
   Observe that even though the image of the front projection is only an immersed graph in $T^2$ (the projections of the red and blue edges to $T^2$ intersect one another), its lift to $Y_0$ is embedded. }
   \label{fig:frontProjectionI}
   \end{subfigure}\;\;\;\;\;
   \begin{subfigure}{.45\linewidth}
   \begin{tikzpicture}
\begin{scope}[]

\draw  (-3,-3) rectangle (3,3);
\draw[red, fuzz] (-3,-1.5) -- (3,-3) (-3,0) -- (3,-1.5) (-3,1.5) -- (3,0) (-3,3) -- (3,1.5);
\draw[blue, fuzz] (-3,3) -- (-3,-3);

\begin{scope}[shift={(-3,0)}]
\foreach \i in {0,...,5}
{
        \pgfmathtruncatemacro{\y}{15*\i}
        \draw[orange] (0,0)-- (\y:.2) ;
}

\node at (0,0) {};
\end{scope}

\begin{scope}[shift={(-3,-3)}]
\foreach \i in {0,...,5}
{
        \pgfmathtruncatemacro{\y}{15*\i}
        \draw[orange] (0,0)-- (\y:.2) ;
}

\node at (0,0) {};
\end{scope}

\begin{scope}[shift={(-3,1.5)}]
\foreach \i in {0,...,5}
{
        \pgfmathtruncatemacro{\y}{15*\i}
        \draw[orange] (0,0)-- (\y:.2) ;
}

\node at (0,0) {};
\end{scope}

\begin{scope}[shift={(-3,-1.5)}]
\foreach \i in {0,...,5}
{
        \pgfmathtruncatemacro{\y}{15*\i}
        \draw[orange] (0,0)-- (\y:.2) ;
}

\node at (0,0) {};
\end{scope}

\end{scope}

\end{tikzpicture}    \caption{The front projection of a mostly 1-Legendrian of $Y_0$.
   Vertices are unmarked. 
   By specifying multiple covectors at a single point, we describe the Legendrian along edges where the tangent space is parallel to the $\theta$-coordinate (see the orange ``spokes'' of covectors joining the red and blue cycles)}
   \label{fig:frontProjectionII}
   \end{subfigure}
   \caption{Two examples of front projections.}
\end{figure}

Consider now a different contact form for the same contact structure on this space:
\[\lambda=s(\theta)dq_1 + r(\theta)dq_2\]
where 
\begin{align*}
 r|_{(-\pi/2+\eps, \pi/2-\eps)} = \tan(\theta), && r|_{S^1\setminus(-\pi/2+\eps/2, \pi/2-\eps/2)}=\sin(\theta), && s(\theta):=r(\theta)\cdot \cot(\theta).
\end{align*}
We will restrict ourselves to the subset $Y_0:=\{(q_1, q_2, r) \st r=r(\theta), \theta\in (-\pi/2+\eps,\pi/2-\eps)\}$.
On this subset, the contact form is given by $\lambda=dq_1+rdq_2$.
If a submanifold is contained in $Y_0$ and is parameterized by $(q_1(t), q_2(t), r(t))$, the Legendrian condition simplifies to 
\begin{equation}
\label{eq:Legendrian}
 \dot q_1 = - r\dot q_2.
\end{equation}

On $Y_0$, we can also consider the projection to the $(q_2, r)$ coordinates, giving us the \emph{Lagrangian projection}:
\begin{align*}
 p_{lag} \colon Y_0\to & S^1\times \RR\\
 (q_1, q_2,r)\mapsto& (q_2,r).
\end{align*}
In the examples we consider, it is easier to visualize the Legendrians via this projection, as the images are embedded graphs rather than immersed ones.
For the Lagrangian projection, we now regard the same cylinder with the coordinates reversed and write $M=S^1\times (R^-, R^+)$, with coordinates $(q_2,r)$ and area form $dr\wedge dq_2$, where $R^-=-\cot\eps$ and $R^+=\cot\eps$.

\begin{df}
\label{def:lemb}
Let $\mathcal L^{emb}_\Lambda(Y)$ denote the set of mostly Legendrians of $Y_0$ based at $\Lambda$ which are homeomorphic to their image under $p_{lag}$.
\end{df}

As with the front projection, the relation \cref{eq:Legendrian} allows us to recover $\Leg$ from $p_{lag}\circ \Leg$ (up to a translation in the $q_1$ coordinate) by fixing a basepoint $z_0 \in \Im(p_{lag}\circ \Leg)$ and defining $q_1(z_1)=\int_{z_0}^{z_1} -r\,dq_2$. 
There is a necessary and sufficient condition determining whether a degenerate embedded graph $\phi \colon G\to M$ can be lifted to a Legendrian.
\begin{prop}
   \label{lem:liftingGraph}
 Let $\phi \colon (G,\bullet)\to M$ be a degenerate embedded pointed graph.
 Suppose for every cycle $c$ in $G$ we have 
\begin{equation}
   \int_{c} r\,dq_2 \in \ZZ.
\label{eq:liftingCondition}
\end{equation}
Then there exists a unique mostly Legendrian $\Leg \colon (G,\bullet)\to Y_0$ satisfying
\begin{align*}
 p_{lag}\circ \Leg = \phi  && \text{and} && q_1(\Leg(\bullet))=0.
\end{align*}
\end{prop}
\begin{proof}
To construct the lift, it suffices to give the $q_1$ coordinate.
For any $x\in G$, pick a path $\gamma_x:[0,1]\to G$ with $\gamma_x(0)=\bullet$ and $\gamma_x(1)=x$.
Then we define 
\[q_{1}(x)= \int_{\phi \circ \gamma_x} -r\,dq_2.\]
For any different choice of path $\gamma_x'$, the closed loop $\gamma_x^{-1}\cdot \gamma_x'$ represents a class in $H_1(G)$, so the condition \eqref{eq:liftingCondition} ensures that the two resulting values of $q_1(x)$ differ by an integer.
Therefore, we have a well-defined function $q_1 \colon G\to \RR/\ZZ$ which does not depend on the choice of path $\gamma_x$.
If $e\colon [0,1]\to G$ is a non-constant oriented edge and we choose, for each $s\in [0,1]$, the path $\gamma_{e(s)}$ to be a fixed path from $\bullet$ to $e(0)$ followed by $e|_{[0,s]}$, then
\[
q_1(e(s))=q_1(e(0))+\int_{\phi\circ e|_{[0,s]}}-r\,dq_2.
\]
Differentiating with respect to $s$ gives
\[
\frac{d}{ds}q_1(e(s))=-r(e(s))\frac{d}{ds}q_2(e(s)),
\]
so \cref{eq:Legendrian} holds along $e$.
Hence each non-constant edge lifts to a Legendrian curve, and the resulting graph is mostly Legendrian.
The marked point fixes the remaining translation ambiguity by imposing $q_1(\bullet)=0$, so the lift is unique as a pointed Legendrian.
\end{proof}

Observe that if a cycle $c$ bounds a face $f\in F(\phi)$, then by Stokes' theorem,
\[\int_{c} r\,dq_2 = \int_f dr \wedge dq_2 = \Area(f).\]
So, the condition in \cref{lem:liftingGraph} is equivalent to checking that all the faces of $\phi \colon G\to M$ have integer area and that \cref{eq:liftingCondition} is satisfied on cycles generating the image of the map $H_1(G)\to H_1(M)$.

\begin{example}
The graph from \Cref{fig:graphGn} lifts to a mostly Legendrian as every face has area one and $r\dot q_2$ integrates to $0$ over the blue cycle.
 In fact, the lift of \Cref{fig:graphGn} has Legendrian front projection drawn in \Cref{fig:frontProjectionII}.
\end{example}

By carrying out \cref{lem:liftingGraph} in families, we obtain a lift of moduli spaces of embedded graphs.
For later use, define
\[
\mathcal N^{\Area,\mathrm{lift}}_{\phi,1}(M)
:=
\left\{
\psi\in \mathcal N^{\Area}_{\phi,1}(M)\ \middle|\ 
\int_c r\,dq_2\in \ZZ \text{ for every cycle }c\subset G
\right\},
\]
and let $\mathcal N^{\Area,\mathrm{lift}}_{\phi}(M)$ denote its image under the map forgetting the marked point.

\begin{cor}
    \label{cor:liftingSpace}
 Suppose that $\phi \colon (G, \bullet)\to M$ lifts to $\Leg \colon (G, \bullet)\to Y_0$.
 Then there exists a lifting map $\Leg_{(-)} \colon \mathcal N^{\Area,\mathrm{lift}}_{\phi,1}(M)\to \mathcal L_{\Leg}^{emb}(Y_0)$.
\end{cor}
\begin{proof}
We must show that the lift from \cref{lem:liftingGraph} varies continuously in families.
By the quotient-topology convention above, it is enough to work locally on a parameter space $K$ where a map $K\to \mathcal N^{\Area,\mathrm{lift}}_{\phi,1}(M)$ is represented by a continuous family of pointed degenerate embeddings
\[
\phi_k \colon (G,\bullet)\to M
\]
for one fixed subdivided graph $G$.
Because every $\phi_k$ lies in $\mathcal N^{\Area,\mathrm{lift}}_{\phi,1}(M)$, the integer period condition \cref{eq:liftingCondition} holds for all $k$.

Choose once and for all, for each vertex $x\in G$, a combinatorial path $\gamma_x$ in $G$ from the marked point $\bullet$ to $x$.
For every parameter $k\in K$, define
\[
q_{1,k}(x):=\int_{\phi_k\circ \gamma_x} -r\,dq_2 \in \RR/\ZZ.
\]
Since the $1$-form $-r\,dq_2$ is smooth on $M$ and the maps $\phi_k\circ \gamma_x$ vary continuously in the compact-open topology, the value $q_{1,k}(x)$ depends continuously on $k$ for each vertex $x$.
Choose an orientation on each edge $e\colon [0,1]\to G$ with initial vertex $v_e=e(0)$.
For $x=e(s)$ on such an edge, define
\[
q_{1,k}(x):=q_{1,k}(v_e)+\int_{\phi_k\circ e|_{[0,s]}}-r\,dq_2 \in \RR/\ZZ.
\]
For fixed $e$, the right-hand side depends continuously on $(k,s)$ because it is the sum of the continuous vertex value $q_{1,k}(v_e)$ and a path integral over the continuously varying family of edge segments $\phi_k\circ e|_{[0,s]}$.
At $s=1$, this agrees with the vertex value $q_{1,k}(e(1))$ modulo $\ZZ$ by the integer period condition \cref{eq:liftingCondition}.
Hence these edgewise formulas glue to a continuous map
\[
K\times G\longrightarrow Y_0,\qquad (k,x)\longmapsto \Leg_{\phi_k}(x),
\]
whose restriction over each $k$ satisfies \cref{eq:Legendrian} on every non-constant edge and is therefore precisely the lift supplied by \cref{lem:liftingGraph}.
The construction is independent of the chosen representative family, because two local lifts differ by a $q_1$-translation and the marked point normalization forces that translation to be zero.
Hence the local constructions glue on overlaps and descend to a continuous map
\[
\Leg_{(-)}: \mathcal N^{\Area,\mathrm{lift}}_{\phi,1}(M)\to \mathcal L_{\Leg}^{emb}(Y_0).
\]
\end{proof}

Note that the marking makes this map canonical. The marking fixes the $q_1$-translation ambiguity of the Legendrian lift so that the target is the embedded-projection subspace rather than a lift only defined up to deck transformation.

\begin{prop}
\label{prop:lagProjectionLeftInverse}
The Lagrangian projection induces a continuous map
\[
\mathrm{pr}_{lag}\colon \mathcal L_{\Leg}^{emb}(Y_0)\to \mathcal N^{\Area,\mathrm{lift}}_{\phi}(M),\qquad \Lambda\mapsto p_{lag}\circ \Lambda.
\]
Moreover, $\mathrm{pr}_{lag}\circ \Leg_{(-)}$ is the forgetful map
\[
\mathcal N^{\Area,\mathrm{lift}}_{\phi,1}(M)\to \mathcal N^{\Area,\mathrm{lift}}_{\phi}(M)
\]
that drops the marked point.
\end{prop}
\begin{proof}
Let $\Lambda^1\in \mathcal L_{\Leg}^{emb}(Y_0)$ and choose a path $\Lambda^t$ in $\mathcal L_{\Leg}^{emb}(Y_0)$ from the base Legendrian $\Leg$ to $\Lambda^1$.
Because each $\Lambda^t$ has embedded Lagrangian projection, $\phi^t:=p_{lag}\circ \Lambda^t$ is a path of embedded graphs in $M$.
For every bounded face $f_t$ of $\phi^t$, let $c_t=\partial f_t$ be the positively oriented boundary cycle.
As in the discussion following \cref{lem:liftingGraph}, Stokes' theorem and the Legendrian condition give
\[
\Area(f_t)=\int_{c_t} r\,dq_2\in \ZZ.
\]
The left-hand side depends continuously on $t$, hence is constant; thus every bounded face of $\phi^t$ has the same area as the corresponding face of $\phi^0$ for all $t$.
Moreover, every $\phi^t$ lifts by construction, so \cref{eq:liftingCondition} holds for each $t$.
Thus the projected path is area-preserving and liftable, so $\phi^1$ belongs to $\mathcal N^{\Area,\mathrm{lift}}_{\phi}(M)$.

Continuity of $\mathrm{pr}_{lag}$ is immediate from the quotient-topology model for $\mathcal L_{\Leg}^{emb}(Y_0)$ and the continuity of the pointwise projection $p_{lag} \colon Y_0\to M$.
Finally, if $\psi\in \mathcal N^{\Area,\mathrm{lift}}_{\phi,1}(M)$, then the lift $\Leg_{\psi}$ was defined precisely so that $p_{lag}\circ \Leg_{\psi}$ is the underlying unpointed graph.
Hence $\mathrm{pr}_{lag}\circ \Leg_{(-)}$ is the forgetful map.
\end{proof}
\begin{cor}
    \label{cor:braidLegendrian}
    There is an inclusion
    \[
    \mathcal B_{\nn, 1}\into \pi_1(\mathcal N^{\Area,\mathrm{lift}}_{\phi_\nn,1}(M),\phi_\nn).
    \]
    Composing with $\Leg_{(-)}$ represents these classes by loops in $\mathcal L_{\Leg_{\phi_\nn}}^{emb}(Y)$.
\end{cor}
\begin{proof}
    By \cref{prop:subgroup,prop:Bn1AsConf}, the classes of the explicit loops $\tau_1,\dots,\tau_{\n},\rho$ generate a subgroup
    \[
    B\subset \pi_1(\mathcal N^{\Area}_{\phi_\nn}(M))
    \]
    canonically identified with $\mathcal B_{\nn,1}$.
    We choose the marked-point section on the bottom annular cycle and use it to lift $B$ to the pointed graph space.
    Consider the subset $C\subset \mathcal N^{\Area}_{\phi_\nn, 1}(M)$ of graphs where the marked point belongs to the ``bottom'' cycle of $G_\nn$, that is, the cycle which is the boundary of the bottom annulus.
    The forgetful map $C\to \mathcal N^{\Area}_{\phi_\nn}(M)$ is a trivial $S^1$ fibration, so choosing the marked point on the bottom cycle gives a section and hence a split injection on $\pi_1$.
    The explicit braid loops lie in the liftable locus, so this split injection lands in $\pi_1(\mathcal N^{\Area,\mathrm{lift}}_{\phi_\nn,1}(M),\phi_\nn)$.
    Composing these loops with $\Leg_{(-)}$ gives the stated loops of Legendrian stops.
\end{proof}
While we only consider Legendrian stops that are the lifts of embedded graphs, it is natural to conjecture the following.
\begin{conj}
The spaces $\mathcal L^{emb}_{\Leg_\nn}(Y)$ and $\mathcal L_{\Leg_\nn}(Y)$ are homotopy equivalent.
\end{conj}

\section{Induced action on partially wrapped Fukaya categories}
\label{sec:monodromyAction}
Over each point of the form $\Lambda=\Leg_{\phi}$ with $\phi\in\mathcal N^{\Area,\mathrm{lift}}_{\phi_\nn,1}(M)$, we have an associated partially wrapped Fukaya category $\mathcal W(\XA, \Lambda)$.
In this section, we will build geometric equivalences between these categories using Lagrangian correspondences.
The ambient contact boundary remains $Y=\partial_\infty \XA$ throughout, while $Y_0\subset Y$ denotes the standard coordinate chart used for our lifted graph model.
For brevity of notation, we fix $\nn$ in this section and denote $\mathcal N:=\mathcal N^{\Area,\mathrm{lift}}_{\phi_\nn, 1}(M)$.
Let $\mathcal N^I$ be the space of smooth paths in $\mathcal N$ (recall that these are smooth paths factoring through $\operatorname{DegEmb}(G, M)$ for some combinatorial type $G$).
For $\rho>0$, write $\Phi^\rho$ for the time-$\rho$ Reeb flow on $Y$ and define
\[
\mathcal L^{\rho}:=
\left\{
\Lambda \in \mathcal L_{\stp_\nn}(Y)\ \middle|
\Phi^{\rho'}(\Lambda)\cap \Lambda = \emptyset \text{ for all $0<\rho'\le \rho$}
\right\}.
\]
Observe that $\bigcup_{\rho}\mathcal L^{\rho}=\mathcal L_{\stp_\nn}(Y)$.

The entirety of this section will prove the following.

\begin{thm}
\label{thm:monodromyAction}
Using the construction below, we define a continuous map
    \begin{align*}
        \Corr_{A,\rho}\colon \mathcal N^I\to \Lag(\XA^-\times \XA)
    \end{align*}
    to the space of cylindrical Lagrangian submanifolds, smoothly depending on the parameters $(A,\rho)\in (0,1)\times \RR_{>0}$. 
    We additionally construct filtrations 
    \begin{itemize}
    \item $U^\rho\subset \mathcal N^I$ with $U^\rho\subset U^{\rho'}$ whenever $\rho>\rho'$ and   
    \item $U_{A,\rho}\subset U^\rho$ with $U_{A,\rho}\subset U_{A', \rho}$ whenever $A<A'$.
    \end{itemize}
    The maps $\Corr_{A,\rho}$ and filtrations satisfy the following properties:
    \begin{enumerate}[(a)]
    \item The filtrations are exhaustive,\label{item:exhaust}
    \item There is compatibility with the lifting map $\Leg_{(-)}: \mathcal N\to \mathcal L_{\Leg_\nn}^{emb}(Y_0)$ in the sense that for $\phi^t \in U_{A,\rho}$
    \[\Corr_{A,\rho}(\phi^t)\in \Ob(\mathcal W(\XA^-\times \XA,\Leg_{\phi^0}\times \Leg_{\phi^1})). \] \label{item:compatibility}
    \end{enumerate}
    The next items provide compatibility between the different choices of $\rho, A$. Pick $I_A\times I_\rho\subset (0,1)\times \RR_{>0}$ parameters we would like to interpolate over. 
    \begin{enumerate}[(a)]
        \setcounter{enumi}{2} 
    \item Given a smooth path $\phi^t\in \bigcap_{(A,\rho)\in I_A\times I_\rho}U_{A,\rho}$ and any smooth path $(A_s, \rho_s): I\to I_A\times I_\rho$, $\Corr_{A_s, \rho_s}(\phi^t)$ parameterizes a Hamiltonian isotopy of Lagrangian submanifolds.
     \label{item:preferredContinuation}
    \item Whenever $\phi^{t, s}\in U_{A,\rho}$ is a smooth homotopy of smooth paths relative endpoints, $\Corr_{A,\rho}(\phi^{t, 0})\sim \Corr_{A,\rho}(\phi^{t, 1})$ (where $\sim$ represents Hamiltonian isotopy in $\XA^- \times \XA$ avoiding the product stop $(\Leg_{\phi^0}\times \Leg_{\phi^1})$). \label{item:hamiltonianIsotopy}
    \item Whenever $\phi_{0}^t, \phi_{1}^t\in \mathcal N^I$ satisfy $\phi_0^{1}=\phi_1^{0}$ and can therefore be concatenated, there is a preferred Hamiltonian isotopy between $\Corr_{A,\rho}(\phi_0^t\cdot \phi_1^t)$ and $\Corr_{A,\rho}(\phi_0^t)\circ \Corr_{A,\rho}(\phi_1^t)$, where the former operation is concatenation of paths and the latter is geometric composition in the middle factor. 
    Furthermore, after choosing $\rho$ sufficiently small so that $\Leg_{\phi_1^1}\in \mathcal L^{2\rho}$ and then choosing $A$ sufficiently close to $1$, this isotopy provides an equivalence of objects in $\Ob(\mathcal W(\XA^-\times \XA,\Leg_{\phi_0^0}\times \Leg_{\phi_1^1})).$
    \label{item:composition}
    \end{enumerate}
\end{thm}
\begin{cor}
For any smooth path $\phi^t\in\mathcal N^I$ with endpoints $\phi^0,\phi^1$, the geometric composition with the correspondence $\Corr_{A,\rho}(\phi^t)$ gives an exact functor
\[
\mathcal F_{A,\rho}(\phi^t): H^0\mathcal W(\XA,\Leg_{\phi^0})\longrightarrow H^0\mathcal W(\XA,\Leg_{\phi^1}),
\]
well-defined up to natural isomorphism for $A$ sufficiently close to $1$ and $\rho$ sufficiently small.
These functors assemble over the subgroupoid whose morphisms are represented by smooth paths in $\mathcal N^I$.
In particular, for any basepoint $\phi\in\mathcal N$ there is a group homomorphism
\[
\pi_1(\mathcal N,\phi)\longrightarrow \Aut(H^0\mathcal W(\XA,\Leg_{\phi})).
\]
\end{cor}
\begin{proof}
By the stopped-correspondence formalism for exact cylindrical Lagrangians in product sectors, together with \cref{item:compatibility,item:composition} (see, for instance, \cite[Theorem 1.4 and Proposition 1.37]{ganatra2024sectorial}), the correspondence $\Corr_{A,\rho}(\phi^t)$ defines the functor.
The smoothness and well-definedness of the construction are recorded in \cref{prop:corrChoiceIndependence}.
\Cref{item:hamiltonianIsotopy} shows that it depends only on the endpoint-fixed homotopy class of $\phi^t$, and \cref{item:composition} identifies
\[
\mathcal F_{A,\rho}(\phi_0^t\cdot \phi_1^t)\cong \mathcal F_{A,\rho}(\phi_1^t)\circ \mathcal F_{A,\rho}(\phi_0^t)
\]
up to natural isomorphism for $A$ sufficiently close to $1$ and $\rho$ sufficiently small (by \cref{item:exhaust}); furthermore, replacing $(A, \rho)$ with choices closer to $(1, 0)$ defines an equivalent functor.
Taking based loops gives the stated homomorphism on $\pi_1(\mathcal N,\phi)$.
\end{proof}

Of these items, the most delicate are \cref{item:exhaust} and \cref{item:compatibility}, which motivate the whole construction.
We now describe the role that $(A,\rho)$ play in the construction.
The map $\Corr_{A,\rho}$ will be defined as the composition of several maps, each smoothly dependent on the parameters $A\in (0, 1)$ and $\rho\in \RR_{>0}$.
The first stage of the construction assigns to a path of graphs a time-dependent Hamiltonian on $M$:
\begin{align*}
(H_A)_* \colon \mathcal N^I\to \Ham^I(M),
\end{align*}
where $\Ham(M)$ is the space of Hamiltonian functions $H \colon M\to \RR$ and $\Ham^I(M)$ is its free path space. For a path $\phi^t\in\mathcal N^I$, the output $(H_A)_*(\phi^t)$ is the time-dependent Hamiltonian produced from the infinitesimal flux of the isotopy at each time $t$.
This is then upgraded to the composition:
\begin{equation}
\label{eq:composition}
\begin{tikzcd}
    \mathcal N^I \arrow{r}{(H_A)_*}& \Ham^I(M)\arrow{r}{C} &  P \mathcal C\arrow{r}{\Gamma^\rho} &  \Lag(\XA^-\times \XA)
\end{tikzcd}
\end{equation}
where $\mathcal C$ is the space of contactomorphisms $\tilde \Psi^t \colon Y\to Y$ that preserve the contact form $(\tilde{\Psi}^t)^*\lambda=\lambda$, $P\mathcal C$ is the path space (based at the identity), and $\Lag(\XA^-\times \XA)$ is the space of cylindrical Lagrangian submanifolds of $\XA^-\times \XA$.
Given any basepoint $\phi^0\in \mathcal N$, we now have a diagram which does not commute on the nose:
\begin{equation}
\begin{tikzcd}
P(\mathcal N, \phi^0)\arrow{r}{C\circ (H_{A})_*}  \arrow{dr}[swap]{\Leg_{\phi^1}} &P \mathcal C\arrow{d}{\Psi^1(\Leg_{\phi^0})}\\
& \mathcal L^{emb}_{\stp_\nn}(Y)
\end{tikzcd}\label{eq:desirable}
\end{equation}
where $P(\mathcal N, \phi^0)$ is the path space of $\mathcal N$ based at $\phi^0$. 
If the diagram did strictly commute, we could apply, for instance, \cite[Theorem 1.4]{ganatra2024sectorial} to obtain equivalences between the categories we construct.
However, such an equality cannot hold, as the combinatorial type of $\phi^t \colon G^t\to M$ changes discontinuously over the $t$ parameter (so the image of $\Leg_{\phi^1}$ may not be homeomorphic to $\Leg_{\phi^0}$; on the other hand the image of $\Psi^1(\Leg_{\phi^0})$ is obtained by a flow and is therefore homeomorphic to $\Leg_{\phi^0}$).

The parameters $(A,\rho)$ allow us to control the failure of \cref{eq:desirable} as follows (see \Cref{fig:fudged}).
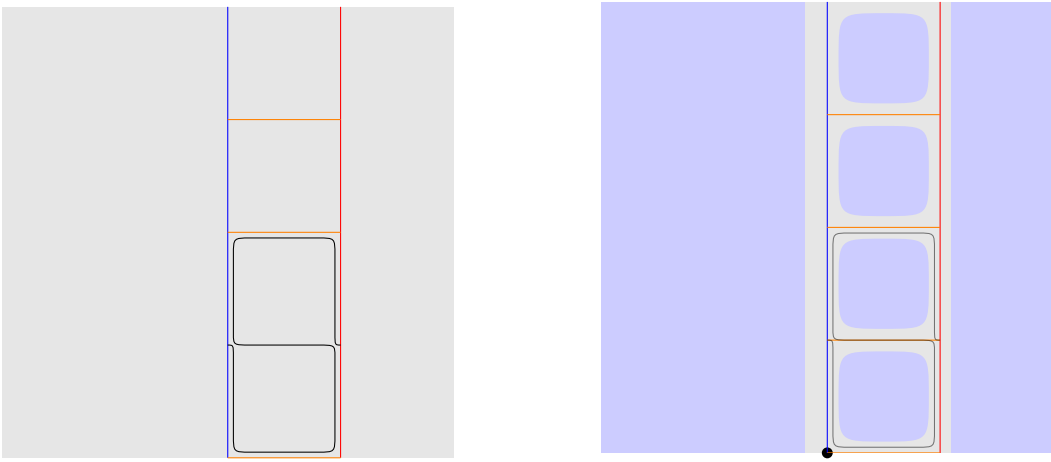
\begin{figure}[h]
\centering\begin{subfigure}{.5\linewidth}
    \centering
\begin{tikzpicture}[xscale=-1.5, yscale=1.5,  rotate=90]
\fill[gray!20]  (-3,-0.5) rectangle (1,-4.5);
\draw[red] (-3,-1.5) -- (1,-1.5) ;
\draw[blue] (-3,-2.5) -- (1,-2.5);
\draw[orange] (-3,-1.5) -- (-3,-2.5)  (-1,-1.5) -- (-1,-2.5) (0,-1.5) -- (0,-2.5);
\draw[black] (-2,-2.5) .. controls (-2,-2.45) and (-2,-2.45) .. (-2.1,-2.45) .. controls (-2.2,-2.45) and (-2.7,-2.45) .. (-2.8,-2.45) .. controls (-2.95,-2.45) and (-2.95,-2.45) .. (-2.95,-2.3) .. controls (-2.95,-2.2) and (-2.95,-1.8) .. (-2.95,-1.7) .. controls (-2.95,-1.55) and (-2.95,-1.55) .. (-2.8,-1.55) .. controls (-2.7,-1.55) and (-2.2,-1.55) .. (-2.1,-1.55) .. controls (-2,-1.55) and (-2,-1.55) .. (-2,-1.7) .. controls (-2,-1.8) and (-2,-2.2) .. (-2,-2.3) .. controls (-2,-2.45) and (-2,-2.45) .. (-1.9,-2.45) .. controls (-1.8,-2.45) and (-1.3,-2.45) .. (-1.2,-2.45) .. controls (-1.05,-2.45) and (-1.05,-2.45) .. (-1.05,-2.3) .. controls (-1.05,-2.2) and (-1.05,-1.8) .. (-1.05,-1.7) .. controls (-1.05,-1.55) and (-1.05,-1.55) .. (-1.2,-1.55) .. controls (-1.3,-1.55) and (-1.8,-1.55) .. (-1.9,-1.55) .. controls (-2,-1.55) and (-2,-1.55) .. (-2,-1.5);
\end{tikzpicture} \caption{Our isotopies do not send loops of graphs to loops of Legendrians, but come close.
Here is a hypothetical Lagrangian projection of $\tilde \Psi^1(\Leg_{\phi^0})$, where $\phi^t$ is the loop from \Cref{fig:graphLoop}.}
\label{fig:fudged}
\end{subfigure}
\begin{subfigure}{.45\linewidth}
    \centering
    \begin{tikzpicture}[xscale=-1.5, yscale=1.5,  rotate=90]

\fill[gray!20]  (-3,-0.5) rectangle (1,-4.5);
\node[circle, fill=black, scale=.4] at (-3,-2.5) {};

\fill[fill=blue!20]  (1,-2.7) rectangle (-3,-4.5);

\fill[blue!20]  (1,-0.5) rectangle (-3,-1.4);
\clip  (1,-3) rectangle (-3,-1);
\begin{scope}[]
\draw[red] (-4,-1.5) -- (1,-1.5) ;
\draw[blue] (-4,-2.5) -- (1,-2.5);
\draw[orange] (-3,-1.5) -- (-3,-2.5) (-2,-1.5) -- (-2,-2.5) (-1,-1.5) -- (-1,-2.5) (0,-1.5) -- (0,-2.5);
\draw[orange] (-4,-2.5) -- (-4,-1.5);
\end{scope}

\begin{scope}[]

\fill[blue!20] (-2.5,-2.4) .. controls (-2.9,-2.4) and (-2.9,-2.4) .. (-2.9,-2) .. controls (-2.9,-1.6) and (-2.9,-1.6) .. (-2.5,-1.6) .. controls (-2.1,-1.6) and (-2.1,-1.6) .. (-2.1,-2) .. controls (-2.1,-2.4) and (-2.1,-2.4) .. (-2.5,-2.4);

\end{scope}

\begin{scope}[shift={(3,0)}]

\fill[blue!20] (-2.5,-2.4) .. controls (-2.9,-2.4) and (-2.9,-2.4) .. (-2.9,-2) .. controls (-2.9,-1.6) and (-2.9,-1.6) .. (-2.5,-1.6) .. controls (-2.1,-1.6) and (-2.1,-1.6) .. (-2.1,-2) .. controls (-2.1,-2.4) and (-2.1,-2.4) .. (-2.5,-2.4);

\end{scope}

\begin{scope}[shift={(2,0)}]

\fill[blue!20] (-2.5,-2.4) .. controls (-2.9,-2.4) and (-2.9,-2.4) .. (-2.9,-2) .. controls (-2.9,-1.6) and (-2.9,-1.6) .. (-2.5,-1.6) .. controls (-2.1,-1.6) and (-2.1,-1.6) .. (-2.1,-2) .. controls (-2.1,-2.4) and (-2.1,-2.4) .. (-2.5,-2.4);

\end{scope}

\begin{scope}[shift={(1,0)}]

\fill[blue!20] (-2.5,-2.4) .. controls (-2.9,-2.4) and (-2.9,-2.4) .. (-2.9,-2) .. controls (-2.9,-1.6) and (-2.9,-1.6) .. (-2.5,-1.6) .. controls (-2.1,-1.6) and (-2.1,-1.6) .. (-2.1,-2) .. controls (-2.1,-2.4) and (-2.1,-2.4) .. (-2.5,-2.4);

\end{scope}
\draw[ black, opacity=.5] (-2,-2.5) .. controls (-2,-2.45) and (-2,-2.45) .. (-2.1,-2.45) .. controls (-2.2,-2.45) and (-2.7,-2.45) .. (-2.8,-2.45) .. controls (-2.95,-2.45) and (-2.95,-2.45) .. (-2.95,-2.3) .. controls (-2.95,-2.2) and (-2.95,-1.8) .. (-2.95,-1.7) .. controls (-2.95,-1.55) and (-2.95,-1.55) .. (-2.8,-1.55) .. controls (-2.7,-1.55) and (-2.2,-1.55) .. (-2.1,-1.55) .. controls (-2,-1.55) and (-2,-1.55) .. (-2,-1.7) .. controls (-2,-1.8) and (-2,-2.2) .. (-2,-2.3) .. controls (-2,-2.45) and (-2,-2.45) .. (-1.9,-2.45) .. controls (-1.8,-2.45) and (-1.3,-2.45) .. (-1.2,-2.45) .. controls (-1.05,-2.45) and (-1.05,-2.45) .. (-1.05,-2.3) .. controls (-1.05,-2.2) and (-1.05,-1.8) .. (-1.05,-1.7) .. controls (-1.05,-1.55) and (-1.05,-1.55) .. (-1.2,-1.55) .. controls (-1.3,-1.55) and (-1.8,-1.55) .. (-1.9,-1.55) .. controls (-2,-1.55) and (-2,-1.55) .. (-2,-1.5);
\end{tikzpicture}
 \caption{A controlled thickening of $\phi^1$ is defined by the blue region; throughout, the ``thickening'' of the graph refers to the complementary region.
Each blue region has area $A$, and $\tilde \Psi^1(\Leg_{\phi^0})$ belongs to the complement.}
\label{fig:thickening}
\end{subfigure}
\caption{Fudging the commutativity of the diagram \eqref{eq:desirable}.}
\end{figure}
\begin{itemize}
    \item The parameter $A$ describes how close the diagram \cref{eq:desirable} is to commuting.
    In our setting the stop attached to $\phi$ is the Legendrian lift $\Leg_{\phi}$, and \emph{the main content of the proof of \cref{thm:monodromyAction}} is the endpoint estimate
    \begin{equation}
        \label{eq:hausdorffDistance}
        \tag{Item b$'$}
        \lim_{A\to 1} d_{Haus}(\Leg_{\phi^{1}},(C\circ (H_A)_*)^1(\Leg_{\phi^0})) = 0.
    \end{equation}
    where $d_{Haus}$ is the Hausdorff distance on $Y_0$.
    \item The parameter $\rho$ controls the time-$\rho$ Reeb flow $\Phi^\rho\colon Y \to Y$.
    Instead of asking for \cref{eq:desirable} to commute, we observe that for any mostly Legendrian stop, there exists $\rho_0$ such that
    \[\Phi^\rho(\Lambda)\cap \Lambda=\emptyset\] for all $0<\rho<\rho_0$.
    More generally, for each compact interval $J\subset(0,\rho_0]$ there exists $\eps(\Lambda,J)>0$ such that
    \[
    d_{Haus}(\Lambda,\Lambda')<\eps(\Lambda,J)\implies \Phi^{\rho'}(\Lambda')\cap \Lambda=\emptyset
    \quad\text{for all }\rho'\in J.
    \]
\end{itemize}
This is the reason for the definition of $\mathcal L^\rho$ given at the start of the section.
For later reference, we record the actual filtrations:
\begin{align*}
U^\rho
&:=
\{\phi^t\in \mathcal N^I \mid \Leg_{\phi^1}\in \mathcal L^{\rho}\},\\
U_{A,\rho}
&:=
\left\{
\phi^t\in U^\rho \ \middle|\ 
d_{Haus}\bigl((C\circ(H_{A'})_*)^1(\Leg_{\phi^0}),\Leg_{\phi^1}\bigr)
<
\delta(\Leg_{\phi^1},\rho)
\text{ for every $A'\in[A,1)$}
\right\},
\end{align*}
where $\delta(\Lambda,\rho)$ denotes the Hausdorff-stability constant from \cref{lem:fixedRhoHausdorff} applied to the pair $(\Lambda,\Lambda)$; this is well defined because $\Lambda\in \mathcal L^\rho$ implies $\Phi^\rho(\Lambda)\cap \Lambda=\emptyset$.
The threshold condition immediately gives $U_{A,\rho}\subset U_{A',\rho}$ whenever $A<A'$.
Then \cref{eq:hausdorffDistance} proves that these filtrations are exhaustive (\cref{item:exhaust}) while being a strong enough bound to produce \cref{item:compatibility}.
The remainder of the section constructs the maps from \cref{eq:composition}, proves \cref{eq:hausdorffDistance} (and that it implies \cref{item:compatibility}), and proves \cref{item:preferredContinuation,item:composition}.
Each portion will describe one of the maps $H_A, C$ or $\Gamma^\rho$, and the various properties of those maps which give us \cref{item:compatibility,item:preferredContinuation,item:composition}.

\subsection{Controlled graph thickenings}
To avoid having to deal with changes in combinatorial types of graph, our argument relies on a thickening of the graph instead. 
Each element $V\in \text{Thick}_{\phi_\nn}(M)$ is a subset $V\subset M$ with $\nn$ disjoint simply connected components with smooth boundary, and two connected components with topology $I\times S^1$ containing the boundary components of $M$.
Once a graph $\phi$ is fixed, we write $V=\bigsqcup_{f\in F(\phi)} V_f$ when we decompose over the connected components.
We say that $V$ is a thickening of $\phi\colon G\to M$ if $V_f\subset f$ for every face $f\in F(\phi)$.
It is an $A$-thickening if $\Area(V_f)=A\cdot \Area(f)$ for every $f\in F(\phi)$.

\begin{lem}[Controlled thickenings]
\label{lem:controlledThickenings}
Fix a compact family of smooth graph paths in $\mathcal N^I$.
For $A$ sufficiently close to $1$, there are $A$-thickenings $V_A^t\subset M$, smoothly depending on $(A,t)$, such that:
\begin{enumerate}
\item $V_A^t$ has one component $V_{A,f}^t$ in each face of $M\setminus\Im(\phi^t)$, including the two annular faces;
\item $\Area(V_{A,f}^t)=A\Area(f_t)$;
\item $K_A^t:=M\setminus\Int(V_A^t)$ admits a facewise retraction $\Ret_A^t:K_A^t\to\Im(\phi^t)$ and a homotopy $\Ret_{A,s}^t$, $s\in[0,1]$, from the identity to $\Ret_A^t$;
\item there are functions $\ell_A\to0$ such that all retraction fibers have diameter at most $\ell_A$ and length at most $\ell_A$;
\item for every $R>0$ there is a constant $L_R$, independent of $A$, such that if $y_A(t)\in K_A^t$ is any path with $|\dot y_A|\le R$, then the strip $S_A(t,s):=\Ret_{A,s}^t(y_A(t))$ satisfies
\[
|\partial_tS_A|\le L_R,
\qquad
\operatorname{length}\bigl(S_A(t,\cdot)\bigr)=\int_0^1|\partial_sS_A(t,s)|\,ds\le\ell_A .
\]
\end{enumerate}
\end{lem}
\begin{proof}
By the quotient topology discussion in \cref{subsec:graph}, a compact family is covered by finitely many charts on which the graphs are represented on a fixed subdivided source.
On such a chart, the boundary of each face is a continuous family of piecewise smooth Jordan curves, smooth on each open edge-stratum and with fixed cyclic ordering of the boundary arcs after passing to a common subdivision.
Thus each bounded face is a continuous family of piecewise smooth disks, and each annular face is a continuous family of piecewise smooth annuli with one fixed boundary component in $\partial M$.

For each bounded face, fix the point-selection map from \cite{belegradek2025point}.
This gives a continuous interior point
\[
  \Cent_f^+(\phi^t)\in f_t .
\]
We use this selection only at the continuous level.
The Riemann map normalized by this interior point is continuous in the parameter by the Carath\'eodory--Rad\'o continuity theorem for Jordan domains; on compact subdisks this convergence is $C^\infty$ in the source variable by the Cauchy integral formula.
We choose such a Riemann map
\[
  u_{f,t}\colon \Int(D^2)\to f_t
\]
with $u_{f,t}(0)=\Cent_f^+(\phi^t)$.
By Carath\'eodory's theorem we regard $u_{f,t}$ as extending continuously to $\overline{D^2}$; this extension is used only to define the radial retraction to $\partial f_t$.
The residual rotation of $D^2$ does not affect the images of concentric disks or the radial retraction.
A similar construction works for the annular faces. 
Let $r_{A,f}(t)$ be the unique radius satisfying
\[
\Area\bigl(u_{f,t}(\{|z|\le r_{A,f}(t)\})\bigr)=A\Area(f_t),
\]
and set
\[
V_{A,f}^t:=u_{f,t}(\{|z|\le r_{A,f}(t)\}).
\]
Before smoothing, the radial level curves are smooth for each fixed $t$, but only vary continuously in the graph-path parameter. We therefore smooth the family in $t$ on each compact source chart, using a mollifier radius depending on $A$.
We then make the conformal radial correction given by the implicit function theorem to restore the exact area condition.
Because both operations occur inside the omitted collar, the thickenings remain in their faces, the radial retraction estimates persist after replacing $\ell_A$ by another function tending to $0$, and the resulting thickenings, retractions, and homotopies vary smoothly on compact parameter ranges with $A<1$.
The maps $\Ret_A^t$ and $\Ret_{A,s}^t$ are induced by radial projection in the resulting disk and annulus coordinates.
Sufficiently close smoothing choices are connected by the same construction with one additional parameter, and hence give Hamiltonian-isotopic data in the continuation arguments below.
Since the omitted collars shrink radially as $A\to 1$, the boundary estimates above give a number $\ell_A\to0$ bounding the diameters and lengths of the radial retraction paths, while the homotopies $\Ret_{A,s}^t$ have uniformly bounded $C^1$ dependence on the path parameter for $A$ close to $1$.
If $y_A(t)\in K_A^t$ has $|\dot y_A|\le R$, this $C^1$ bound gives $|\partial_t S_A|\le L_R$ for some $L_R$ independent of $A$, while $\int_0^1|\partial_sS_A(t,s)|\,ds\le \ell_A$.
\end{proof}

\subsection{The graph flux primitive}
Given a smooth path $\phi^t\in\mathcal N^I$ the infinitesimal flux primitive along the graph is
\begin{align}
\underline H^t \colon \Im(\phi^t)\to& \RR \\
x\mapsto& -r\dot q_2(\bullet)+\frac{d}{d\tau}\left.\left(\int_{[t,\tau]\times [\bullet, x]}(\phi^\sigma)^* \omega\right)\right|_{\tau=t},
\label{eq:normalization}
\end{align}
where $[\bullet, x]$ is any path in the source graph from the marked point to the point mapping to $x$.
The next lemma shows that this definition is intrinsic.

\begin{lem}
\label{lem:graphFluxPrimitiveWellDefined}
Given a smooth path $\phi^t\in\mathcal N^I$, the formula \eqref{eq:normalization} defines a piecewise smooth function $\underline H^t \colon \Im(\phi^t)\to \RR$ which is independent of the chosen source path $[\bullet,x]$.
It is invariant under subdivision of the source graph, so it is well defined on the quotient-topology path space $\mathcal N^I$.
Moreover, if $e \colon [0,1]\to G$ is an oriented edge, then along the corresponding arc of $\Im(\phi^t)$ we have
    \[
    \frac{d}{ds}\underline H^t(\phi^t(e(s)))=\omega\!\left(\partial_t\phi^t(e(s)),\partial_s\phi^t(e(s))\right).
    \]
\end{lem}
\begin{proof}
    If $\gamma,\gamma'$ are two source paths from $\bullet$ to $x$, their difference is a cycle $c:=\gamma^{-1}\cdot \gamma'$ in $G$.
    The corresponding values of \eqref{eq:normalization} differ by
    \[
    \frac{d}{d\tau}\left.\left(\int_{[t,\tau]\times c}(\phi^\sigma)^*\omega\right)\right|_{\tau=t}.
    \]
    Since $\omega=d(r\,dq_2)$, Stokes' theorem rewrites this as
    \[
    \frac{d}{d\tau}\left.\left(\int_{\phi^\tau(c)} r\,dq_2-\int_{\phi^t(c)} r\,dq_2\right)\right|_{\tau=t}.
    \]
    Every time-slice $\phi^\tau$ belongs to $\mathcal N$, so \cref{eq:liftingCondition} implies that $\int_{\phi^\tau(c)} r\,dq_2$ is an integer.
    As a continuous integer-valued function of $\tau$, it is locally constant.
    Hence the displayed derivative is zero, proving independence of the source path.

    Subdividing the source graph does not change the geometric strip in $M$ traced by the chosen path, and any inserted constant edge contributes zero to the integral.
    Thus \eqref{eq:normalization} is unchanged under subdivision.
    By the quotient-topology convention for $\mathcal N^I$, two local representatives agree after passing to a common subdivision, so the construction descends to the quotient.

    Finally, if $0\le s_0<s_1\le 1$, subtracting the formulas for $\underline H^t(\phi^t(e(s_1)))$ and $\underline H^t(\phi^t(e(s_0)))$ gives
    \[
    \underline H^t(\phi^t(e(s_1)))-\underline H^t(\phi^t(e(s_0)))
    =
    \frac{d}{d\tau}\left.\left(\int_{[t,\tau]\times e|_{[s_0,s_1]}}(\phi^\sigma)^*\omega\right)\right|_{\tau=t}.
    \]
    Differentiating under the integral sign yields
    \[
    \underline H^t(\phi^t(e(s_1)))-\underline H^t(\phi^t(e(s_0)))
    =
    \int_{s_0}^{s_1}\omega\!\left(\partial_t\phi^t(e(s)),\partial_s\phi^t(e(s))\right)\,ds,
    \]
    which is equivalent to the displayed edgewise derivative formula.
\end{proof}

\begin{lem}[Variation of the canonical Legendrian lift]
\label{lem:variationLegendrianLift}
Let $\phi^t\in\mathcal N^I$, and let $x^t\in\Im(\phi^t)$ be a smooth path.
Let $\tilde x^t\in\Leg_{\phi^t}$ be the corresponding path of Legendrian lifts, and choose a continuous real lift of its $q_1\in\RR/\ZZ$ coordinate.
On every interval where $x^t$ is represented by a smooth point of a single non-constant edge, one has
\[
\frac{d}{dt}q_1(\tilde x^t)
=
-\underline H^t(x^t)-\alpha(\dot x^t).
\]
For later use, if $\gamma_x\colon[0,t]\to M$ is the path $\gamma_x(\tau)=x^\tau$, then
\[
q_1(\tilde x^t)-q_1(\tilde x^0)
=
-\int_0^t \underline H^\tau(x^\tau)\,d\tau-\int_{\gamma_x}\alpha .
\]
\end{lem}
\begin{proof}
Choose a fixed source edge $e$ and write $x^t=\phi^t(e(s(t)))$ on such an interval.
Let $\eta_t$ be a source path from the marked point to $e(s(t))$, obtained by a fixed combinatorial path to $e(0)$ followed by the segment $e|_{[0,s(t)]}$.
By the definition of the Legendrian lift,
\[
q_1(\tilde x^t)=-\int_{\phi^t\circ\eta_t}\alpha .
\]
First suppose that $s(t)$ is constant.
For small $h$, let $S_h=[t,t+h]\times\eta_t$ be the strip swept out by the fixed source path.
Stokes' theorem gives
\[
\frac{d}{dh}\left.\left(-\int_{\phi^{t+h}\circ\eta_t}\alpha+\int_{\phi^t\circ\eta_t}\alpha\right)\right|_{h=0}
=
-\frac{d}{dh}\left.\int_{S_h}(\phi^\sigma)^*\omega\right|_{h=0}
-\alpha(\partial_t\phi^t(e(s)))+\alpha(\partial_t\phi^t(\bullet)).
\]
Using \cref{eq:normalization}, this is $-\underline H^t(x^t)-\alpha(\partial_t\phi^t(e(s)))$.
If $s(t)$ varies, the extra infinitesimal segment in the time-slice graph contributes $-\alpha(\partial_s\phi^t(e(s))\dot s)$.
Adding this term gives the differential formula.
The integral formula follows by integrating in $t$.
If the path passes through vertices, the formula is first proved away from the vertex times and then extended by continuity of the endpoint values.
\end{proof}

\subsection{Controlled Hamiltonian extension}
We now extend the graph flux primitive to compactly supported Hamiltonians on $M$.

\begin{prop}[Controlled Hamiltonian extension]
\label{prop:controlledHamiltonianExtension}
Fix a compact family of smooth graph paths in $\mathcal N^I$ and the thickenings from \cref{lem:controlledThickenings}.
For $A$ sufficiently close to $1$, there are Hamiltonians $H_A^t:M\to\RR$, smoothly depending on all parameters, such that:
\begin{enumerate}
\item $H_A^t=0$ on a fixed collar of $\partial M$;
\item the Hamiltonian flow $\Psi_A^t$ carries $V_{A,f}^0$ to $V_{A,f}^t$ for every face $f$;\label{prop:agreement}
\item
\[
\delta_A:=\sup_{\tau\in[0,1]}\sup_{x\in\Im(\phi^\tau)}|H_A^\tau(x)-\underline H^\tau(x)|\longrightarrow0
\qquad\text{as }A\to1;
\]
\item there is a constant $B$, depending on the compact family but not on $A$, such that $|dH_A^t|\le B$ on $K_A^t:=M\setminus\Int(V_A^t)$;\label{prop:boundedDerivative}
\end{enumerate}
\end{prop}
\begin{proof}
The compact graph family stays a positive distance from $\partial M$, so for $A$ close to $1$ the annular thickenings contain a fixed collar of $\partial M$.
The infinitesimal motion of $\partial V_{A,f}^t$ determines a closed $1$-form on that boundary by contraction with $\omega$.
Its period is the time derivative of $\Area(V_{A,f}^t)=A\Area(f_t)$, hence is zero on disk faces.
On an annular face, the fixed boundary component in $\partial M$ has zero period because it is stationary, and the zero derivative of the annular area forces the moving graph-side boundary component to have zero period.
Thus the boundary flux is exact on every component of $\partial V_A^t$.
Define $H_{A,\partial}^t$ to be the unique primitive on $\partial V_A^t$ with the following normalizations.
On the fixed boundary components in $\partial M$, we put $H_{A,\partial}^t=0$.
On each moving boundary component, fix the additive constant by minimizing the $L^2$ norm of $H_{A,\partial}^t-\underline H^t\circ\Ret_A^t$ in the conformal angular coordinate, equivalently by requiring its angular average to vanish.
Using the radial coordinates in \cref{lem:controlledThickenings}, extend $H_{A,\partial}^t$ across $K_A^t$ by a fixed cutoff interpolation to the graph primitive $\underline H^t$ along the retraction fibers.
Because these fibers shrink uniformly and the graph family is compact, the normalized boundary primitive converges in $C^0$ to $\underline H^t\circ\Ret_A^t$ and the interpolation gives a uniform bound for $|dH_A^t|$ on $K_A^t$.
Extend the function across $V_A^t$ by a fixed radial cutoff in the conformal coordinates, preserving the prescribed boundary values and making the function vanish on a fixed collar of $\partial M$.
Near graph vertices, apply the fixed smoothing convention obtained by convolution with a fixed mollifier in the fixed source-chart coordinates at a scale smaller than the conformal collar width.
This preserves the boundary flux and the stated estimates after increasing $B$.
Since no non-constant edge is self-adjacent, each local edge collar is bounded by two distinct face components, so the facewise boundary data glue across vertices also when an edge has zero length at an endpoint.
The resulting Hamiltonian flow carries $\partial V_{A,f}^0$ to $\partial V_{A,f}^t$, hence carries $V_{A,f}^0$ to $V_{A,f}^t$.
\end{proof}

Given $\phi^t\in \mathcal N^I$, denote by $(H_A)_*(\phi^t)$ the time-dependent Hamiltonian $t\mapsto H_A^t$.

\begin{lem}[Projection convergence]
\label{lem:projectionConvergence}
For the Hamiltonians above,
\[
\lim_{A\to1}d_{Haus}\bigl(\Psi_A^1(\Im(\phi^0)),\Im(\phi^1)\bigr)=0.
\]
The same statement holds at any fixed time $t\in[0,1]$.
\end{lem}
\begin{proof}
We prove the endpoint statement; the fixed-time variant is identical.
By \cref{prop:agreement}, the Hamiltonian flow carries $K_A^0$ to $K_A^1$.
Since $\Im(\phi^0)\subset K_A^0$, every point of $\Psi_A^1(\Im(\phi^0))$ lies in $K_A^1$ and is within the retraction-fiber diameter of $\Im(\phi^1)$.
For the reverse estimate, we claim that
\[
\Ret_A^1\bigl(\Psi_A^1(\Im(\phi^0))\bigr)=\Im(\phi^1).
\]
To show this, take a point in the interior of a non-constant edge of $\phi^1$.
The retraction fiber through this point crosses the corresponding thin edge-neighborhood in $K_A^1$ between the two local thickened sectors adjacent to the edge.
By the restriction above, these sectors belong to the time-one faces indexed by distinct face labels $f$ and $g$.
Extend the fiber slightly at its endpoints into $V_{A,f}^1$ and $V_{A,g}^1$.
The inverse image of this extended arc under $\Psi_A^1$ has endpoints in $V_{A,f}^0$ and $V_{A,g}^0$, which lie in distinct components of $M\setminus\Im(\phi^0)$.
It must therefore meet $\Im(\phi^0)$.
Since $\Psi_A^1(\Im(\phi^0))\subset K_A^1$, the original retraction fiber already meets $\Psi_A^1(\Im(\phi^0))$.
Vertex points follow by closure.
Thus every point of $\Im(\phi^1)$ is within the retraction-fiber diameter of $\Psi_A^1(\Im(\phi^0))$, and the two one-sided estimates prove the claim.
\end{proof}

\subsection{Construction of the contactomorphism lift \texorpdfstring{$C$}{C}} 
Throughout this section, we will frequently use the notation $x:=p_{lag}(\tilde x)$ for $\tilde x\in Y_0$.
Let $H^t\in \Ham(M)$ be a time-dependent Hamiltonian function which vanishes on a collar neighborhood of $\partial M$. 
Via the identification $Y_0\cong S^1_{q_1}\times M$, we pull $H^t$ back to a function on $Y_0$ by
\[
(q_1,q_2,\theta)\longmapsto H^t(q_2,r(\theta)).
\]
Because $H^t$ vanishes on a collar of $\partial M$, this pullback has support in a compact subset of $Y_0$.
Write $\alpha:=r\,dq_2$, so $\lambda=dq_1+\alpha$ and $\omega=d\alpha$.
Let $X_{H^t}$ be the Hamiltonian vector field on $M$ in the convention $\iota_{X_{H^t}}\omega=dH^t$.
The strict contact lift of $X_{H^t}$ is
\[
X_{H^t}-\bigl(H^t+\alpha(X_{H^t})\bigr)\partial_{q_1},
\]
since $\lambda(X_{H^t}-(H^t+\alpha(X_{H^t}))\partial_{q_1})=-H^t$ and $\iota_{X_{H^t}}\omega=dH^t$, so its Lie derivative on $\lambda$ is zero.
Because $H^t$ vanishes on a collar of $\partial M$, this vector field has compact support in $Y_0$ and extends by zero to all of $Y$.
Thus the resulting flow $\tilde \Psi^t_{H^t} \colon Y\to Y$ preserves the contact form $\lambda$, projects to the Hamiltonian flow $\Psi^t_H$ on $M$, and changes the $q_1$--coordinate by
    \begin{equation}
        \label{eq:q1EvolutionContact}
        \tilde \Psi^t_{H^t}(q_1,z)=\left(q_1-\int_0^t \left(H^\tau+\alpha(X_{H^\tau})\right)(\Psi^\tau_H(z))\,d\tau,\; \Psi^t_H(z)\right).
    \end{equation}
The normalization term in \cref{eq:normalization} is chosen so that the canonical Legendrian lift satisfies the variation formula in \cref{lem:variationLegendrianLift}.
If an ideal Hamiltonian extension satisfied $H^\tau|_{\Im(\phi^\tau)}=\underline H^\tau$ and carried $x^0$ to $x^\tau$, then \cref{eq:q1EvolutionContact} and \cref{lem:variationLegendrianLift} would agree along the lifted path.

Although the smoothed Hamiltonians $H_A^t$ constructed above only approximate this ideal extension, the error vanishes as $A\to 1$, which is enough to prove the required Hausdorff convergence of the Legendrian lifts.

\begin{lem}[Convergence of contact lifts]
\label{lem:convergenceContactLifts}
Let $\phi^t\in\mathcal N^I$ be a smooth path, and let $H_A^t$ be the Hamiltonians from \cref{prop:controlledHamiltonianExtension}.
For $t=1$,
\[
\lim_{A\to 1} d_{Haus}(\tilde \Psi^{1}_A(\Leg_{\phi^0}), \Leg_{\phi^1}) =0.
\]
The same conclusion holds at any fixed time $t\in[0,1]$.
\end{lem}
\begin{proof}
Fix $t\in[0,1]$; the endpoint statement is the case $t=1$.
We need to examine the case where our time-dependent Hamiltonian is 
$(H_A)_*(\phi^t)=\{H_A^\tau\}_{\tau\in[0,1]}$.
To reduce notation, write $\Psi^t_A:= \Psi^t_{(H_A)_*(\phi^t)}$ and $\tilde\Psi_A^t:=\tilde\Psi^t_{(H_A)_*(\phi^t)}$ in the remainder of this section.
Use a product Riemannian metric on $Y_0=S^1\times\overline M$; write $d_Y$, $d_M$, and $d_{Haus}$ for the induced distances and Hausdorff distance.
After increasing a constant if needed, there is $C_d>1$ such that
\begin{equation}
    \label{eq:metricComparisonY0}
    d_Y((q_1,z),(q_1',z'))\le C_d\left(|q_1-q_1'|+d_M(z,z')\right)
\end{equation}
for all $(q_1,z),(q_1',z')\in S^1\times \overline M$, where $|q_1-q_1'|$ denotes the distance between the $q_1$--coordinates in $\RR/\ZZ$.
When real lifts of the $q_1$--coordinates are chosen below, the absolute value of their real difference bounds this circle distance.
By \cref{lem:projectionConvergence}, the Hausdorff distance between the Lagrangian projections of these sets goes to zero, so it remains to control the $q_1$ coordinate.
Set
\[
\alpha_A(t):=d_{Haus}\bigl(\Psi_A^t(\Im(\phi^0)),\Im(\phi^t)\bigr),
\]
so $\alpha_A(t)\to 0$ as $A\to 1$.

Let
\[
K_A^\tau:=M\setminus \bigcup_f \Int(V_{A,f}(\phi^\tau)).
\]
The point of the thickening estimates is that the $q_1$ coordinate is controlled by the area of strips in $K_A^\tau$ and by the lengths of the radial retraction paths; both errors vanish as $A\to1$.

By \cref{prop:agreement}, the Hamiltonian flow carries $K_A^0$ to $K_A^\tau$, so $\Psi_A^\tau(\Im(\phi^0))\subset K_A^\tau$.
Let $\Ret_A^\tau\colon K_A^\tau\to \Im(\phi^\tau)$ be the facewise retraction from \cref{lem:controlledThickenings}, and let $\Ret_{A,s}^\tau$, $s\in[0,1]$, be the corresponding homotopy from the identity to $\Ret_A^\tau$.
Since the family $\phi^\tau$, $\tau\in[0,t]$, is compact, the retraction fibers shrink uniformly as $A\to 1$:
\[
\ell_A:=
\sup_{\tau\in[0,t]}\sup_{z\in K_A^\tau}
\operatorname{length}\bigl(s\mapsto \Ret_{A,s}^\tau(z)\bigr)
\longrightarrow 0 .
\]
In particular, $d_M(z,\Ret_A^\tau(z))\le \ell_A$ for every
$\tau\in[0,t]$ and $z\in K_A^\tau$.

We first prove a uniform one-sided estimate.  Fix $\tilde y_0\in \Leg_{\phi^0}$ and write
\[
\tilde y_A(\tau):=\tilde\Psi_A^\tau(\tilde y_0),
\qquad
y_A(\tau):=p_{lag}(\tilde y_A(\tau))=\Psi_A^\tau(p_{lag}(\tilde y_0)).
\]
Set
\[
x_A(\tau):=\Ret_A^\tau(y_A(\tau))\in \Im(\phi^\tau),
\]
and let $\tilde x_A(\tau)\in\Leg_{\phi^\tau}$ be the lifted path with
$\tilde x_A(0)=\tilde y_0$.
The calculation below is first made on intervals where $x_A(\tau)$ lies in a single smooth edge and then extended across vertex times by continuity of the endpoint values.

Consider the strip
\[
S_A\colon [0,t]\times[0,1]\to M,\qquad
S_A(\tau,s):=\Ret_{A,s}^\tau(y_A(\tau)).
\]
Let $\alpha:=r\,dq_2$, so that $\omega=d\alpha$.

Write
\[
\gamma_y(\tau):=y_A(\tau),\qquad
\gamma_x(\tau):=x_A(\tau),\qquad
\beta_\sigma(s):=S_A(\sigma,s)\quad(\sigma=0,t).
\]
Choose real lifts of the $q_1$--coordinates starting from the common value at
$\tau=0$.
\begin{align}
q_1(\tilde y_A(t))-q_1(\tilde x_A(t))
&=
-\int_0^t
\bigl(H_A^\tau(y_A(\tau))-\underline H^\tau(x_A(\tau))\bigr)\,d\tau \notag\\
&\quad+
\int_{\gamma_x}\alpha-\int_{\gamma_y}\alpha . \notag
\intertext{Here the $H_A$--term comes from the contact evolution formula
\eqref{eq:q1EvolutionContact}, and the $\underline H$--term comes from
\cref{lem:variationLegendrianLift}. The oriented boundary of the strip is}
\partial S_A
&=\gamma_y+\beta_t-\gamma_x-\beta_0 . \notag
\intertext{Since $\omega=d\alpha$, Stokes' theorem gives}
\int_{\gamma_x}\alpha-\int_{\gamma_y}\alpha
&=
-\int_{[0,t]\times[0,1]}S_A^*\omega
+
\int_{\beta_t}\alpha-\int_{\beta_0}\alpha . \notag
\intertext{Substituting this into the first line and taking absolute values gives}
\left|q_1(\tilde y_A(t))-q_1(\tilde x_A(t))\right|
&\le
\int_0^t
\left|H_A^\tau(y_A(\tau))-\underline H^\tau(x_A(\tau))\right|\,d\tau \notag\\
&\quad+
\left|\int_{[0,t]\times[0,1]}S_A^*\omega\right|
+
\left|\int_{\beta_t}\alpha\right|
+
\left|\int_{\beta_0}\alpha\right|.
\label{eq:q1RadialComparison}
\end{align}
The first term is controlled by the approximation of $H_A^\tau$ to the graph flux primitive and the uniform derivative bound. Indeed, after enlarging the constant in \cref{prop:boundedDerivative}, we have on $K_A^\tau$
\[
\left|H_A^\tau(y_A(\tau))-\underline H^\tau(x_A(\tau))\right|
\le
\int_0^1\left|dH_A^\tau(\partial_s\Ret_{A,s}^\tau(y_A(\tau)))\right|\,ds
+
\left|H_A^\tau(x_A(\tau))-\underline H^\tau(x_A(\tau))\right|
\le B\ell_A+\delta_A .
\]

It remains to bound the strip-area term and the two endpoint retraction integrals in \eqref{eq:q1RadialComparison}.
The vector fields $X_{H_A^\tau}$ are uniformly bounded on the compact sets $K_A^\tau$, again by \cref{prop:boundedDerivative} and compactness of $\overline M$.
Thus $|\dot y_A|\le R$ for some $R$ independent of $A$, and the strip estimate in \cref{lem:controlledThickenings} gives a constant $L_R$ such that $|\partial_\tau S_A|\le L_R$, while $\int_0^1|\partial_sS_A(\tau,s)|\,ds\le \ell_A$.
If $C_\omega$ is the operator norm of $\omega$ with respect to the chosen product metric, then
\[
\left|\int_{[0,t]\times[0,1]}S_A^*\omega\right|
\le C_\omega L_R t\,\ell_A .
\]
The initial retraction path is constant, because $y_A(0)\in\Im(\phi^0)$ and the retraction fixes $\Im(\phi^0)$.
If $C_\alpha:=\sup_{\overline M}|\alpha|$, the final retraction path has length at most $\ell_A$, and therefore
\[
\left|\int_{\beta_t}\alpha\right|
+
\left|\int_{\beta_0}\alpha\right|
\le C_\alpha \ell_A .
\]
Combining these estimates gives
\[
\left|q_1(\tilde y_A(t))-q_1(\tilde x_A(t))\right|
\le t\delta_A+(Bt+C_\omega L_Rt+C_\alpha)\ell_A\longrightarrow 0,
\]
uniformly in the initial point $\tilde y_0$.
Since $d_M(y_A(t),x_A(t))\le \ell_A$, \eqref{eq:metricComparisonY0} implies
\[
\sup_{\tilde y\in \tilde\Psi_A^t(\Leg_{\phi^0})}
d_Y\bigl(\tilde y,\Leg_{\phi^t}\bigr)\longrightarrow 0 .
\]

For the reverse Hausdorff estimate, fix $\tilde x\in\Leg_{\phi^t}$ and put $x=p_{lag}(\tilde x)$.
By the convergence of Lagrangian projections, choose $y_A\in \Psi_A^t(\Im(\phi^0))$ with
$d_M(y_A,x)\le \alpha_A(t)$, and let $\tilde y_A\in\tilde\Psi_A^t(\Leg_{\phi^0})$ be the point over $y_A$.
This point is unique because $p_{lag}$ restricts to a homeomorphism on $\tilde\Psi_A^t(\Leg_{\phi^0})$.
Set $\bar x_A:=\Ret_A^t(y_A)$ and let $\tilde{\bar x}_A\in\Leg_{\phi^t}$ be its lift.
The one-sided estimate just proved gives
$d_Y(\tilde y_A,\tilde{\bar x}_A)\to 0$ uniformly, while
\[
d_M(\bar x_A,x)\le \ell_A+\alpha_A(t)\longrightarrow 0 .
\]
Since $p_{lag}\colon \Leg_{\phi^t}\to \Im(\phi^t)$ is a homeomorphism from a compact space to its image, its inverse is uniformly continuous. Therefore
$d_Y(\tilde{\bar x}_A,\tilde x)\to 0$, uniformly in $\tilde x$.
Thus
\[
\sup_{\tilde x\in \Leg_{\phi^t}}
d_Y\bigl(\tilde x,\tilde\Psi_A^t(\Leg_{\phi^0})\bigr)\longrightarrow 0 .
\]
Together with the one-sided estimate, this proves the claim.
\end{proof}

\noindent \textbf{Proof of \cref{eq:hausdorffDistance}:}
Apply \cref{lem:convergenceContactLifts} with $t=1$.

\subsection{Constructing the correspondence \texorpdfstring{$\Gamma^\rho$}{Gamma rho}}
\label{subsubsec:buildingGamma}
Let $\tilde \Psi^t$ be a contact isotopy starting at the identity.
Take a collar neighborhood of the boundary of $\XA$ so that $\XA \cong \XA_{\mathrm{int}} \cup_{\partial \XA} (\partial \XA \times [0,1))$.
We write the collar coordinate as $s$, normalized in the outward direction: $s=0$ is the finite gluing locus with $\XA_{\mathrm{int}}$, while $s\to 1$ approaches the contact boundary, equivalently the positive end after completing. Thus the interpolation is the identity where it is glued to the interior, and its limiting contact map at infinity is $\tilde\Psi^1$.
Let $h_t\colon Y\to\RR$ be the contact Hamiltonian of the isotopy, so if
$X_t=\partial_t\tilde\Psi^t\circ(\tilde\Psi^t)^{-1}$ and $\lambda_Y$ denotes the contact form on $Y$, then
$h_t=\lambda_Y(X_t)$.
On the completed collar, write the Liouville form as $r\lambda_Y$.
Choose a smooth cutoff $\chi(r)$ which is $0$ near the finite gluing locus and $1$ on the cylindrical end, and define the time-dependent Hamiltonian
\[
K_t(r,y)=-\chi(r)r h_t(y),
\]
extended by zero over $\XA_{\mathrm{int}}$.
With the convention $\iota_{X_{K_t}}\omega=dK_t$, the time-one flow $\Phi$ of $K_t$ is the identity over $\XA_{\mathrm{int}}$ and, where $\chi=1$, is the standard symplectization lift of $\tilde\Psi^1$.
In particular $\Phi$ is cylindrical at infinity with boundary contactomorphism $\tilde\Psi^1$.
Moreover, Hamiltonian isotopy gives
\[
\Phi^*\lambda_\XA-\lambda_\XA
=
d\!\left(\int_0^1(\Phi_t)^*\bigl(K_t+\lambda_\XA(X_{K_t})\bigr)\,dt \right),
\]
and this primitive is constant on the cylindrical end.
Thus, suppressing the completion notation, we obtain an exact cylindrical symplectomorphism $\Phi\colon \XA \to \XA$.
Consider the graph of $\Phi$, which is a Lagrangian submanifold $\Gamma(\Phi) \subset \XA^- \times \XA$ parameterized by
\[
x \mapsto (x, \Phi(x)).
\]
Since $\Phi$ is exact, choose $f_\Phi\colon \XA\to \RR$ with
\[
df_\Phi=\Phi^*\lambda_\XA-\lambda_\XA.
\]
Pulling back the Liouville form $\lambda_{\XA^- \times \XA}=-\pi_1^*\lambda_\XA+\pi_2^*\lambda_\XA$ along the parameterization of $\Gamma(\Phi)$ gives
\[
\lambda_{\XA^- \times \XA}|_{\Gamma(\Phi)}
=-\lambda_\XA+\Phi^*\lambda_\XA
=df_\Phi.
\]
Thus $f_\Phi$ is an exact primitive for $\Gamma(\Phi)$.
Therefore, $\Gamma(\Phi)$ is exact and cylindrical. The chosen polarization gives the diagonal its canonical grading and orientation data, and the graph isotopy between the diagonal and $\Gamma(\Phi)$ induced by the chosen isotopy transports this brane structure to canonical grading and orientation data on $\Gamma(\Phi)$. With this primitive and brane structure, $\Gamma(\Phi)$ supports an object of $\mathcal W(\XA^- \times \XA)$.
For $\rho > 0$, let $w^\rho:\XA\to \XA$ denote the time-$\rho$ flow of the wrapping Hamiltonian, and set
\begin{equation}
\label{eq:gammarho}
\Gamma^\rho(\Phi):=\Gamma(w^\rho\circ \Phi).
\end{equation}

We will also write $\Gamma^\rho(\tilde \Psi^t):=\Gamma^\rho(\Phi)$ for the correspondence determined by a contact isotopy $\tilde \Psi^t$ with $\tilde \Psi^0=\mathrm{id}$.

\begin{lem}[Smoothness of the canonical data]
\label{lem:continuousAuxChoices}
After fixing the smoothing convention used in the proof of \cref{lem:controlledThickenings}, cutoff functions, mollifiers, and vertex smoothing convention used in the definition of $H_A$, the thickenings, retractions, Hamiltonians, contact lifts, exact symplectomorphisms, and wrapped graph correspondences above vary smoothly over compact families in $\mathcal N^I$ and compact parameter ranges in $(A,\rho)$.
\end{lem}
\begin{proof}
On each fixed source chart, the thickenings and retractions are the smoothed conformal level thickenings from the proof of \cref{lem:controlledThickenings}.
The rotation ambiguities in the disk and annular maps do not change the concentric domains, the radial retractions, or the conformal angular averages used to normalize $H_{A,\partial}^t$.
The exact-area corrections are smooth by the implicit function theorem, and the formulas in \cref{lem:controlledThickenings,prop:controlledHamiltonianExtension} use fixed cutoff functions in these radial coordinates.
This smoothness is asserted on compact parameter ranges with $A<1$; as $A\to1$, we only use the uniform vanishing of the corresponding retraction-fiber diameters and lengths.
The fixed vertex smoothing convention is local on the fixed source chart and depends smoothly on the input function.
The strict contact lift, the collar extension over $\XA$, and the time-$\rho$ wrapping are smooth operations on these data.
\end{proof}

\begin{prop}[Well-definedness of $\Corr_{A,\rho}$]
\label{prop:corrChoiceIndependence}
Using the data fixed above, the recipe \eqref{eq:composition} defines a smooth map on triples $(A,\rho,\phi^t)$ with $\phi^t$ a smooth path in $\operatorname{DegEmb}(G, M)$:
\[
(A,\rho,\phi^t)\longmapsto \Corr_{A,\rho}(\phi^t)\in \Lag(\XA^-\times \XA).
\]
\end{prop}
\begin{proof}
This follows from \cref{lem:continuousAuxChoices} and from the explicit formulas for the strict contact lift, the exact cylindrical extension, and the wrapped graph construction.
\end{proof}

\begin{lem}[Hausdorff openness of Reeb disjointness]
\label{lem:fixedRhoHausdorff}
Fix an auxiliary Riemannian metric on $Y$, and let $d_{Haus}$ denote the induced Hausdorff distance on compact subsets. Let $\Lambda_-,\Lambda_+\subset Y$ be compact and let $J\subset(0,\infty)$ be compact. If
\[
\Phi^\tau(\Lambda_-)\cap \Lambda_+=\emptyset
\qquad\text{for every $\tau\in J$,}
\]
then there exists $\delta=\delta(\Lambda_-,\Lambda_+;J)>0$ such that every compact $\Lambda_-'\subset Y$ with
\[
d_{Haus}(\Lambda_-',\Lambda_-)<\delta
\]
still satisfies
\[
\Phi^\tau(\Lambda_-')\cap \Lambda_+=\emptyset
\qquad\text{for every $\tau\in J$.}
\]
For $J=\{\rho\}$, we write $\delta(\Lambda_-,\Lambda_+,\rho)$.
\end{lem}
\begin{proof}
Since $\bigcup_{\tau\in J}\Phi^\tau(\Lambda_-)$ and $\Lambda_+$ are compact and disjoint, their distance
\[
m:=\mathrm{dist}\left(\bigcup_{\tau\in J}\Phi^\tau(\Lambda_-),\Lambda_+\right)
\]
is strictly positive. Because $J\times Y$ is compact, the Reeb flow is uniformly continuous on this set, so there exists $\delta>0$ such that
\[
d(y,y')<\delta \implies d(\Phi^\tau(y),\Phi^\tau(y'))<m/2
\qquad\text{for every $\tau\in J$.}
\]
If $y\in \Lambda_-'$, choose $x\in \Lambda_-$ with $d(x,y)<\delta$. Then for every $\tau\in J$,
\[
\mathrm{dist}(\Phi^\tau(y),\Lambda_+)
\ge \mathrm{dist}(\Phi^\tau(x),\Lambda_+)-d(\Phi^\tau(x),\Phi^\tau(y))
> m-m/2>0.
\]
Hence $\Phi^\tau(y)\notin \Lambda_+$ for all $y\in \Lambda_-'$ and all $\tau\in J$, proving the claim.
\end{proof}

\begin{lem}[Wrapped graph criterion for the product stop]
\label{lem:wrappedGraphCriterion}
Let $\tilde\Psi^t$ be a contact isotopy of $Y$ starting at the identity, and let $F \colon \XA\to \XA$ be the associated exact symplectomorphism constructed above.
Then, for compact $\Lambda_-,\Lambda_+\subset Y$, the wrapped graph $\Gamma^\rho(F)$ is disjoint at infinity from the product stop $\Lambda_-\times \Lambda_+$ if and only if
\[
\Phi^\rho(\tilde\Psi^1(\Lambda_-))\cap \Lambda_+=\emptyset.
\]
In particular, if the condition on the right-hand side holds, then $\Gamma^\rho(F)$ supports an object of
\[
\mathcal W(\XA^-\times \XA,\Lambda_-\times \Lambda_+).
\]
\end{lem}
\begin{proof}
Near infinity, $F$ is the symplectization of the contactomorphism $\tilde\Psi^1$, so $\Gamma(F)$ is cylindrical over the Legendrian graph
\[
\{(y,\tilde\Psi^1(y))\mid y\in Y\}\subset Y\times Y.
\]
The wrapping $w^\rho$ acts on this end by the time-$\rho$ Reeb flow in the second factor.
Moreover, because $F$ preserves the cylindrical end, points of $\Gamma(F)$ reach infinity in one factor if and only if they do so in the other, so the only possible intersection with the product stop occurs in the $Y\times Y$ corner.
The equivalence therefore follows immediately, and since $\Gamma^\rho(F)$ is exact and cylindrical by construction, this is exactly the condition needed for it to define an object of the stopped category.
\end{proof}

\noindent\textbf{Proof of \cref{item:compatibility}:}
Set
\[
\Lambda_0:=\Leg_{\phi^0},
\qquad
\Lambda_1:=\Leg_{\phi^1},
\]
and write $\tilde\Psi_A^t:=(C\circ (H_A)_*)(\phi^t)$ for the contact isotopy produced above. Let $F_A \colon \XA\to \XA$ be the associated exact symplectomorphism, so that by definition
\[
\Corr_{A,\rho}(\phi^t)=\Gamma^\rho(F_A).
\]

Assume $\phi^t\in U_{A,\rho}$. Since $\Lambda_1\in \mathcal L^\rho$, we have
\[
\Phi^\rho(\Lambda_1)\cap \Lambda_1=\emptyset.
\]
Let $\delta(\Lambda_1,\rho)$ be the constant from \cref{lem:fixedRhoHausdorff} applied to the pair $(\Lambda_1,\Lambda_1)$.
By the definition of $U_{A,\rho}$, we have
\[
d_{Haus}(\tilde\Psi_A^1(\Lambda_0),\Lambda_1)<\delta(\Lambda_1,\rho).
\]
Therefore \cref{lem:fixedRhoHausdorff} gives
\[
\Phi^\rho(\tilde\Psi_A^1(\Lambda_0))\cap \Lambda_1=\emptyset.
\]
Applying \cref{lem:wrappedGraphCriterion} with $\Lambda_-=\Lambda_0$ and $\Lambda_+=\Lambda_1$, we conclude that
\[
\Corr_{A,\rho}(\phi^t)\in \Ob\bigl(\mathcal W(\XA^-\times \XA,\Lambda_0\times \Lambda_1)\bigr)
\]
which proves \cref{item:compatibility}.

\noindent\textbf{Proof of \cref{item:preferredContinuation}:}
Let $\phi^t\in \bigcap_{(A,\rho)\in I_A\times I_\rho}U_{A,\rho}$ and $(A_s,\rho_s):I\to I_A\times I_\rho$.
By smooth dependence of the controlled thickenings, $H_A$, $C$, and $\Gamma^\rho$, the family
\[
s\longmapsto \Corr_{A_s,\rho_s}(\phi^t)
\]
is a smooth exact isotopy of cylindrical Lagrangians.
Because $\phi^t$ lies in every $U_{A_s,\rho_s}$, the proof of \cref{item:compatibility} shows that each member of this family is disjoint from the product stop.
Any exact isotopy of exact cylindrical Lagrangians is Hamiltonian, so this gives the required Hamiltonian isotopy.

\noindent\textbf{Proof of \cref{item:hamiltonianIsotopy}:}
Let $\phi^{t,s}\in U_{A,\rho}$ be a smooth homotopy of smooth paths relative endpoints.
Smooth dependence of the controlled thickenings, $H_A$, $C$, and $\Gamma^\rho$ gives a smooth exact cylindrical Lagrangian isotopy
\[
s\longmapsto \Corr_{A,\rho}(\phi^{t,s}).
\]
The endpoints of the graph path are fixed in $s$, so the product stop $\Leg_{\phi^0}\times\Leg_{\phi^1}$ is fixed.
Since every $\phi^{t,s}$ lies in $U_{A,\rho}$, the proof of \cref{item:compatibility} shows that every member of the isotopy is disjoint from this product stop.
Any exact isotopy of exact cylindrical Lagrangians is Hamiltonian, giving the claimed Hamiltonian isotopy in the stopped category.

\noindent\textbf{Proof of \cref{item:composition}:}
Let $\tilde \Psi^t_0,\tilde \Psi^t_1$ be the contact isotopies associated to $\phi_0^t,\phi_1^t$, and let $\Phi_0,\Phi_1 \colon \XA\to \XA$ be the corresponding Liouville-domain maps.
These isotopies start at the identity, but their endpoints need not.
Since our graphs are parameterized by
\[
x\longmapsto (x,\Phi(x)),
\]
geometric composition in the middle factor satisfies
\[
\Gamma(F_0)\circ \Gamma(F_1)=\Gamma(F_1\circ F_0).
\]
Thus the order of correspondences agrees with the order of path concatenation, but reverses the written order of the underlying endomorphisms of $\XA$.
\begin{align*}
    \Gamma^\rho(\tilde \Psi^t_0) \circ \Gamma^\rho(\tilde \Psi^t_1) &= \Gamma(w^\rho \circ \Phi_1 \circ w^\rho \circ \Phi_0).
\end{align*}
For $0\le s\le 1$, define
\[
K_s:=\Gamma(w^\rho \circ \Phi_1 \circ w^{s\rho} \circ \Phi_0).
\]
Then
\[
K_1=\Gamma^\rho(\tilde \Psi^t_0) \circ \Gamma^\rho(\tilde \Psi^t_1),
\qquad
K_0=\Gamma(w^\rho \circ \Phi_1 \circ \Phi_0).
\]
Since $s\mapsto w^{s\rho}$ is a Hamiltonian isotopy of $\XA$, the family $K_s$ is an exact isotopy of exact graphs in $\XA^-\times \XA$, hence a Hamiltonian isotopy.

We take $\phi_0^t\cdot\phi_1^t$ to be the usual smooth temporal concatenation obtained from a fixed reparameterization of the two half-intervals.
The conformal thickening at time $t$ depends only on the graph at that time, while the Hamiltonian extension depends linearly on the first time derivative of the graph path.
Under the fixed reparameterization, the Hamiltonian is therefore rescaled by the derivative of the reparameterization, so the exact symplectic isotopy for the concatenated path is the temporal concatenation of the two exact symplectic isotopies.
Its time-one map is $\Phi_1\circ\Phi_0$.
Then
\[
K_0=\Corr_{A,\rho}(\phi_0^t\cdot \phi_1^t).
\]
This proves the unstopped composition statement.

For the final stopped statement, let
\[
\Lambda_0:=\Leg_{\phi_0^0},
\qquad
\Lambda_1:=\Leg_{\phi_0^1}=\Leg_{\phi_1^0},
\qquad
\Lambda_2:=\Leg_{\phi_1^1}.
\]
Choose $\rho>0$ so small that $\Lambda_2\in\mathcal L^{2\rho}$.
By \cref{item:exhaust}, after taking $A$ sufficiently close to $1$ if necessary, the concatenated path $\phi_0^t\cdot\phi_1^t$ lies in $U_{A,\rho}$.

It remains to consider the isotopy obtained by turning off the middle wrapping. The corresponding boundary contactomorphisms before the final wrap form the compact family
\[
\Xi_{A,s}:=\tPsi_{1,A}^1\circ \Phi^{s\rho}\circ \tPsi_{0,A}^1,
\qquad s\in[0,1],
\]
so that the full wrapped graph criterion applies to $\Phi^\rho\circ \Xi_{A,s}$.
Since the contactomorphisms $\tPsi_{i,A}^1$ are strict, they commute with the Reeb flow.
Thus
\[
\Phi^\rho\circ\Xi_{A,s}(\Lambda_0)
=
\Phi^{(1+s)\rho}\bigl((\tPsi_{1,A}^1\circ\tPsi_{0,A}^1)(\Lambda_0)\bigr).
\]
As $A\to1$, the set $\tPsi_{0,A}^1(\Lambda_0)$ converges to $\Lambda_1$ and $\tPsi_{1,A}^1(\Lambda_1)$ converges to $\Lambda_2$ by \cref{eq:hausdorffDistance}.
The maps $\tPsi_{1,A}^1$ are uniformly continuous on the compact family under consideration, so the image of the first convergence under $\tPsi_{1,A}^1$ remains convergent after applying the second estimate.
Therefore
\[
d_{Haus}\bigl((\tPsi_{1,A}^1\circ\tPsi_{0,A}^1)(\Lambda_0),\Lambda_2\bigr)
<
\delta(\Lambda_2,\Lambda_2;[\rho,2\rho])
\]
for all $A$ sufficiently close to $1$.
By \cref{lem:fixedRhoHausdorff}, this implies
\[
\Phi^{(1+s)\rho}\bigl((\tPsi_{1,A}^1\circ\tPsi_{0,A}^1)(\Lambda_0)\bigr)\cap\Lambda_2=\emptyset
\qquad\text{for all }s\in[0,1],
\]
and then \cref{lem:wrappedGraphCriterion} shows that every intermediate correspondence in the preferred isotopy avoids the product stop $\Lambda_0\times \Lambda_2$.
Therefore the isotopy from $\Corr_{A,\rho}(\phi_0^t\cdot \phi_1^t)$ to $\Corr_{A,\rho}(\phi_0^t)\circ \Corr_{A,\rho}(\phi_1^t)$ takes place inside
\[
\Ob\bigl(\mathcal W(\XA^-\times \XA,\Lambda_0\times \Lambda_2)\bigr)
\]
as needed.\footnote{Technically speaking, these Lagrangians need to be equipped with grading/orientation data in the sense of \cite[Section 5.3]{ganatra2024microlocal}. However, as our Lagrangians are isotopic to conormals they have canonical grading/orientation data relative to the tautological polarization of $\XA$.} 

\section{Applications}
\label{sec:applications}

With the technical work of establishing a monodromy action via \cref{thm:monodromyAction} in hand, we now examine this action in the context of homological mirror symmetry. 
In particular, we demonstrate several ways in which it is explicitly computable.

\subsection{Isotopies of the FLTZ Legendrian mirror to the $A_{\n}$ singularity}
\label{sec:mainThm}
We first turn to our main example, the FLTZ stop mirror to the $A_{\n}$ singularity $\XB_{\nn}$.
Let $\stp_\nn$ be the FLTZ stop (\cref{def:FLTZ})
associated to the stacky fan given by the pair consisting of the fan $\Sigma_\nn$ in $\RR^2$ whose cones are
\[0, \RR_{\geq 0}\cdot (1, 0), \RR_{\geq 0}\cdot (1, n), \RR_{\geq 0}\cdot (1, 0) + \RR_{\geq 0}\cdot (1, \nn) \]
and the homomorphism $\beta \colon \ZZ^2 \to \ZZ^2$ satisfying $\beta(1,0) = (1,0)$ and $\beta(0,1) = (1,\nn)$.
Note that $\mathrm{Pic}(\XB_{\nn}) \cong \ZZ/\nn\ZZ$ generated by $\Os(1) := \Os(D_{(1,\nn)})$ where $D_{(1,\nn)}$ is the toric divisor corresponding to $\RR_{\geq 0}\cdot (1, \nn)$.
\begin{figure}
    \centering
    \begin{tikzpicture}
\begin{scope}[]

\draw  (-3,-3) rectangle (3,3);
\draw[red, fuzz] (-3,-1.5) -- (3,-3) (-3,0) -- (3,-1.5) (-3,1.5) -- (3,0) (-3,3) -- (3,1.5);
\draw[blue, fuzz] (-3,3) -- (-3,-3);

\begin{scope}[shift={(-3,0)}]
\foreach \i in {0,...,5}
{
        \pgfmathtruncatemacro{\y}{15*\i}
        \draw[orange] (0,0)-- (\y:.2) ;
}

\node at (0,0) {};
\end{scope}

\begin{scope}[shift={(-3,-3)}]
\foreach \i in {0,...,5}
{
        \pgfmathtruncatemacro{\y}{15*\i}
        \draw[orange] (0,0)-- (\y:.2) ;
}

\node at (0,0) {};
\end{scope}

\begin{scope}[shift={(-3,1.5)}]
\foreach \i in {0,...,5}
{
        \pgfmathtruncatemacro{\y}{15*\i}
        \draw[orange] (0,0)-- (\y:.2) ;
}

\node at (0,0) {};
\end{scope}

\begin{scope}[shift={(-3,-1.5)}]
\foreach \i in {0,...,5}
{
        \pgfmathtruncatemacro{\y}{15*\i}
        \draw[orange] (0,0)-- (\y:.2) ;
}

\node at (0,0) {};
\end{scope}

\end{scope}

\end{tikzpicture}     \caption{The Lagrangian projection of the FLTZ stop $\stp_4$ for the fan of the $A_3$ singularity.}
    \label{fig:FLTZSkeleton}
\end{figure}
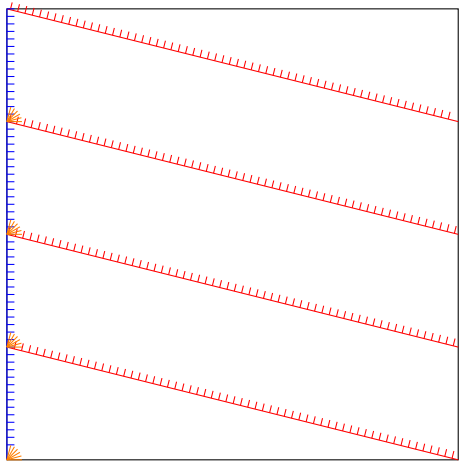

\begin{figure}
\centering
\scalebox{.35}{
\begin{tikzpicture}
\begin{scope}[]

\draw  (-3,-3) rectangle (3,3);
\draw[red, fuzz] (-3,-1.5) -- (3,-3) (-3,0) -- (3,-1.5) (-3,1.5) -- (3,0) (-3,3) -- (3,1.5);
\draw[blue, fuzz] (-3,3) -- (-3,-3);

\begin{scope}[shift={(-3,0)}]
\foreach \i in {0,...,5}
{
        \pgfmathtruncatemacro{\y}{15*\i}
        \draw[orange] (0,0)-- (\y:.2) ;
}

\node at (0,0) {};
\end{scope}

\begin{scope}[shift={(-3,-3)}]
\foreach \i in {0,...,5}
{
        \pgfmathtruncatemacro{\y}{15*\i}
        \draw[orange] (0,0)-- (\y:.2) ;
}

\node at (0,0) {};
\end{scope}

\begin{scope}[shift={(-3,1.5)}]
\foreach \i in {0,...,5}
{
        \pgfmathtruncatemacro{\y}{15*\i}
        \draw[orange] (0,0)-- (\y:.2) ;
}

\node at (0,0) {};
\end{scope}

\begin{scope}[shift={(-3,-1.5)}]
\foreach \i in {0,...,5}
{
        \pgfmathtruncatemacro{\y}{15*\i}
        \draw[orange] (0,0)-- (\y:.2) ;
}

\node at (0,0) {};
\end{scope}

\end{scope}

\end{tikzpicture} \begin{tikzpicture}
\begin{scope}[]

\draw  (-3,-3) rectangle (3,3);
\draw[red, fuzz] (-3,-1.5) -- (3,-3) (-3,0) -- (3,-1.5) (-3,1.5) -- (3,0) (-3,3) -- (3,1.5);
\draw[blue, fuzz] (-3,3) -- (-3,-3);

\begin{scope}[shift={(-3,0)}]
\foreach \i in {0,...,5}
{
        \pgfmathtruncatemacro{\y}{15*\i}
        \draw[orange] (0,0)-- (\y:.2) ;
}

\node at (0,0) {};
\end{scope}

\begin{scope}[shift={(-3,-3)}]
\foreach \i in {0,...,5}
{
        \pgfmathtruncatemacro{\y}{15*\i}
        \draw[orange] (0,0)-- (\y:.2) ;
}

\node at (0,0) {};
\end{scope}

\begin{scope}[shift={(-3,1.5)}]
\foreach \i in {0,...,5}
{
        \pgfmathtruncatemacro{\y}{15*\i}
        \draw[orange] (0,0)-- (\y:.2) ;
}

\node at (0,0) {};
\end{scope}

\end{scope}

\draw[orange, fuzz] (-3,-0.5) .. controls (-3,-1) and (-3,-1.5) .. (-1,-2);
\end{tikzpicture} \begin{tikzpicture}
\begin{scope}[]

\draw  (-3,-3) rectangle (3,3);
\draw[red, fuzz] (-3,-1.5) -- (3,-3) (-3,0) -- (3,-1.5) (-3,1.5) -- (3,0) (-3,3) -- (3,1.5);
\draw[blue, fuzz] (-3,3) -- (-3,-3);

\begin{scope}[shift={(-3,0)}]
\foreach \i in {0,...,5}
{
        \pgfmathtruncatemacro{\y}{15*\i}
        \draw[orange] (0,0)-- (\y:.2) ;
}

\node at (0,0) {};
\end{scope}

\begin{scope}[shift={(-3,-3)}]
\foreach \i in {0,...,5}
{
        \pgfmathtruncatemacro{\y}{15*\i}
        \draw[orange] (0,0)-- (\y:.2) ;
}

\node at (0,0) {};
\end{scope}

\begin{scope}[shift={(-3,1.5)}]
\foreach \i in {0,...,5}
{
        \pgfmathtruncatemacro{\y}{15*\i}
        \draw[orange] (0,0)-- (\y:.2) ;
}

\node at (0,0) {};
\end{scope}

\end{scope}

\draw[orange, fuzz] (-3,0) .. controls (-3,-0.5) and (1,-2.5) .. (3,-3);
\end{tikzpicture} 

\begin{tikzpicture}
\begin{scope}[]

\draw  (-3,-3) rectangle (3,3);
\draw[red, fuzz] (-3,-1.5) -- (3,-3) (-3,0) -- (3,-1.5) (-3,1.5) -- (3,0) (-3,3) -- (3,1.5);
\draw[blue, fuzz] (-3,3) -- (-3,-3);

\begin{scope}[shift={(-3,0)}]
\foreach \i in {0,...,5}
{
        \pgfmathtruncatemacro{\y}{15*\i}
        \draw[orange] (0,0)-- (\y:.2) ;
}

\node at (0,0) {};
\end{scope}

\begin{scope}[shift={(-3,-3)}]
\foreach \i in {0,...,5}
{
        \pgfmathtruncatemacro{\y}{15*\i}
        \draw[orange] (0,0)-- (\y:.2) ;
}

\node at (0,0) {};
\end{scope}

\begin{scope}[shift={(-3,1.5)}]
\foreach \i in {0,...,5}
{
        \pgfmathtruncatemacro{\y}{15*\i}
        \draw[orange] (0,0)-- (\y:.2) ;
}

\node at (0,0) {};
\end{scope}

\end{scope}

\draw[orange, fuzz] (-3,0) .. controls (0,-1.5) and (0,-1.5) .. (3,-3);
\end{tikzpicture} \begin{tikzpicture}
\begin{scope}[]

\draw  (-3,-3) rectangle (3,3);
\draw[red, fuzz] (-3,-1.5) -- (3,-3) (-3,0) -- (3,-1.5) (-3,1.5) -- (3,0) (-3,3) -- (3,1.5);
\draw[blue, fuzz] (-3,3) -- (-3,-3);

\begin{scope}[shift={(-3,0)}]
\foreach \i in {0,...,5}
{
        \pgfmathtruncatemacro{\y}{15*\i}
        \draw[orange] (0,0)-- (\y:.2) ;
}

\node at (0,0) {};
\end{scope}

\begin{scope}[shift={(-3,-3)}]
\foreach \i in {0,...,5}
{
        \pgfmathtruncatemacro{\y}{15*\i}
        \draw[orange] (0,0)-- (\y:.2) ;
}

\node at (0,0) {};
\end{scope}

\begin{scope}[shift={(-3,1.5)}]
\foreach \i in {0,...,5}
{
        \pgfmathtruncatemacro{\y}{15*\i}
        \draw[orange] (0,0)-- (\y:.2) ;
}

\node at (0,0) {};
\end{scope}

\end{scope}

\draw[orange, fuzz] (-3,0) .. controls (-1,-0.5) and (3,-2) .. (3,-3);
\end{tikzpicture} \begin{tikzpicture}
\begin{scope}[]

\draw  (-3,-3) rectangle (3,3);
\draw[red, fuzz] (-3,-1.5) -- (3,-3) (-3,0) -- (3,-1.5) (-3,1.5) -- (3,0) (-3,3) -- (3,1.5);
\draw[blue, fuzz] (-3,3) -- (-3,-3);

\begin{scope}[shift={(-3,0)}]
\foreach \i in {0,...,5}
{
        \pgfmathtruncatemacro{\y}{15*\i}
        \draw[orange] (0,0)-- (\y:.2) ;
}

\node at (0,0) {};
\end{scope}

\begin{scope}[shift={(-3,-3)}]
\foreach \i in {0,...,5}
{
        \pgfmathtruncatemacro{\y}{15*\i}
        \draw[orange] (0,0)-- (\y:.2) ;
}

\node at (0,0) {};
\end{scope}

\begin{scope}[shift={(-3,1.5)}]
\foreach \i in {0,...,5}
{
        \pgfmathtruncatemacro{\y}{15*\i}
        \draw[orange] (0,0)-- (\y:.2) ;
}

\node at (0,0) {};
\end{scope}

\end{scope}

\draw[orange, fuzz] (1,-1) .. controls (3,-1.5) and (3,-2) .. (3,-2.5);
\end{tikzpicture} \begin{tikzpicture}
\begin{scope}[]

\draw  (-3,-3) rectangle (3,3);
\draw[red, fuzz] (-3,-1.5) -- (3,-3) (-3,0) -- (3,-1.5) (-3,1.5) -- (3,0) (-3,3) -- (3,1.5);
\draw[blue, fuzz] (-3,3) -- (-3,-3);

\begin{scope}[shift={(-3,0)}]
\foreach \i in {0,...,5}
{
        \pgfmathtruncatemacro{\y}{15*\i}
        \draw[orange] (0,0)-- (\y:.2) ;
}

\node at (0,0) {};
\end{scope}

\begin{scope}[shift={(-3,-3)}]
\foreach \i in {0,...,5}
{
        \pgfmathtruncatemacro{\y}{15*\i}
        \draw[orange] (0,0)-- (\y:.2) ;
}

\node at (0,0) {};
\end{scope}

\begin{scope}[shift={(-3,1.5)}]
\foreach \i in {0,...,5}
{
        \pgfmathtruncatemacro{\y}{15*\i}
        \draw[orange] (0,0)-- (\y:.2) ;
}

\node at (0,0) {};
\end{scope}

\begin{scope}[shift={(-3,-1.5)}]
\foreach \i in {0,...,5}
{
        \pgfmathtruncatemacro{\y}{15*\i}
        \draw[orange] (0,0)-- (\y:.2) ;
}

\node at (0,0) {};
\end{scope}

\end{scope}

\end{tikzpicture} }
\caption{A loop of mostly Legendrians starting at the FLTZ Legendrian $\stp_{4}$.}
\label{fig:twistFrontProjection}
\end{figure}
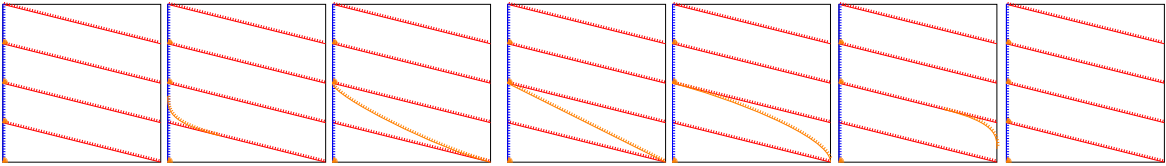
Combining \cref{thm:monodromyAction,cor:braidLegendrian} gives a new proof of the following result (see \cref{subsec:previousResults} for other approaches).
\begin{cor}
There exists an action of $\mathcal B_{\nn,1}$ on the triangulated category $H^0\mathcal W(\XA, \stp_\nn)$ by exact autoequivalences (well-defined up to natural isomorphism).
\label{cor:action}
\end{cor}

In fact, we can go further and intrinsically identify these autoequivalences of the partially wrapped Fukaya category as geometric operations.
Recall that we have preferred generators $\tau_1, \ldots, \tau_{\nn-1}$ and $\rho$ of $\mathcal B_{\nn,1}$ and a cyclic translate $\tau_n$ inside the moduli space of degenerate graphs that are identified in \cref{prop:Bn1AsConf}.
We let $\mathcal F_{\tau_i}$ and $\mathcal{F}_\rho$ be the corresponding autoequivalences from \cref{cor:action}. 
We first describe the $\mathcal F_{\tau_i}$. 

\begin{thm}
For each $i\in\{1,\dots,\n\}$, let $K_i$ denote the linking disk of the Legendrian component of $\stp_{\nn}$ whose projection under $p_{lag}$ is the horizontal line separating the two bounded faces where $\tau_i$ is supported. Then
\[
\mathcal F_{\tau_i}\cong T_{K_i}.
\]
Here $T_{K_i}$ is the spherical twist on $K_i$.
\label{thm:sphericalTwists}
\end{thm}
\begin{proof}
Fix $i$, and consider the Lagrangian correspondence $\Corr_{A, \rho}(\tau_i)$.
Let $\tilde\Psi_{\tau_i}^t$ denote the corresponding contact isotopy, and let $\Phi_{\tau_i}\colon \XA\to \XA$ be the associated exact symplectomorphism from \cref{subsubsec:buildingGamma}.
Observe that the restriction of this correspondence to $\XA^-_{\mathrm{int}}\times \XA_{\mathrm{int}}$ has cylindrical boundary; in the notation of \cref{subsubsec:buildingGamma},
\[
\Gamma^\rho(\Phi_{\tau_i})|_{\XA^-_{\mathrm{int}}\times \XA_{\mathrm{int}}}=\Delta^+_{\XA_{\mathrm{int}}}\subset \XA^-_{\mathrm{int}}\times \XA_{\mathrm{int}}
\]
is a pushoff of the diagonal Lagrangian.
Hence \cite[Proposition 3.5]{hanlon2023relatingcategoricaldimensionstopology} (see also \cite[Proposition 1.37]{ganatra2024sectorial}) gives an isomorphism in the Fukaya category
\[
\Gamma^\rho(\Phi_{\tau_i})\cong [\mathcal E_i \to \Delta_{\XA_{\mathrm{int}}}],
\]
where the term $\mathcal E_i$ corresponds to the intersection between $\Gamma^\rho(\Phi_{\tau_i})$ and the skeleton of $\XA^-\times \XA$ outside of $\XA^-_{\mathrm{int}}\times \XA_{\mathrm{int}}$.
By the same argument as in \cite[Proposition 4.7]{hanlon2023relatingcategoricaldimensionstopology}, this intersection corresponds to the unique intersection between $\Phi^\rho(\tilde \Psi_{\tau_i}^t(\stp_\nn))$ and $\stp_\nn$. For the loop $\tau_i$, there is exactly one such intersection, namely the one indicated in \Cref{fig:stopIntersection}. By construction of the loop $\tau_i$, this intersection lies on the horizontal stop component separating the two bounded faces braided by $\tau_i$. 
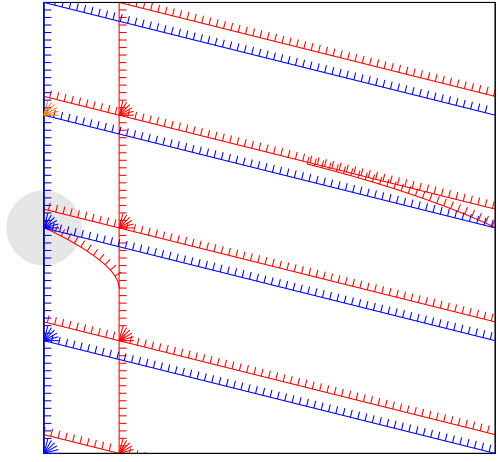
\begin{figure}
        \centering

\begin{tikzpicture}

\fill[gray!20]  (2,-1.5) ellipse (0.5 and 0.5);
\draw  (2,1.5) rectangle (8,-4.5);

\clip  (2,1.5) rectangle (8,-4.5);

\begin{scope}[]

\begin{scope}[]

\draw[red, fuzz] (-3,-1.5) -- (3,-3) (-3,0) -- (3,-1.5) (-3,1.5) -- (3,0) (-3,3) -- (3,1.5);
\draw[red, fuzz] (-3,3) -- (-3,-3);

\begin{scope}[shift={(-3,0)}]
\foreach \i in {0,...,5}
{
        \pgfmathtruncatemacro{\y}{15*\i}
        \draw[red] (0,0)-- (\y:.2) ;
}

\node at (0,0) {};
\end{scope}

\begin{scope}[shift={(-3,-3)}]
\foreach \i in {0,...,5}
{
        \pgfmathtruncatemacro{\y}{15*\i}
        \draw[red] (0,0)-- (\y:.2) ;
}

\node at (0,0) {};
\end{scope}

\begin{scope}[shift={(-3,1.5)}]
\foreach \i in {0,...,5}
{
        \pgfmathtruncatemacro{\y}{15*\i}
        \draw[red] (0,0)-- (\y:.2) ;
}

\node at (0,0) {};
\end{scope}

\begin{scope}[shift={(-3,-1.5)}]
\foreach \i in {0,...,5}
{
        \pgfmathtruncatemacro{\y}{15*\i}
        \draw[red] (0,0)-- (\y:.2) ;
}

\node at (0,0) {};
\end{scope}

\draw[red, fuzz] (-0.5,-0.65) .. controls (1.5,-1.15) and (3,-1.85) .. (3,-2.3);
\end{scope}

\begin{scope}[shift={(6,-6)}]

\draw[red, fuzz] (-3,-1.5) -- (3,-3) (-3,0) -- (3,-1.5) (-3,1.5) -- (3,0) (-3,3) -- (3,1.5);
\draw[red, fuzz] (-3,3) -- (-3,-3);

\begin{scope}[shift={(-3,0)}]
\foreach \i in {0,...,5}
{
        \pgfmathtruncatemacro{\y}{15*\i}
        \draw[red] (0,0)-- (\y:.2) ;
}

\node at (0,0) {};
\end{scope}

\begin{scope}[shift={(-3,-3)}]
\foreach \i in {0,...,5}
{
        \pgfmathtruncatemacro{\y}{15*\i}
        \draw[red] (0,0)-- (\y:.2) ;
}

\node at (0,0) {};
\end{scope}

\begin{scope}[shift={(-3,1.5)}]
\foreach \i in {0,...,5}
{
        \pgfmathtruncatemacro{\y}{15*\i}
        \draw[red] (0,0)-- (\y:.2) ;
}

\node at (0,0) {};
\end{scope}

\begin{scope}[shift={(-3,-1.5)}]
\foreach \i in {0,...,5}
{
        \pgfmathtruncatemacro{\y}{15*\i}
        \draw[red] (0,0)-- (\y:.2) ;
}

\node at (0,0) {};
\end{scope}

\end{scope}

\begin{scope}[shift={(0,-6)}]

\draw[red, fuzz] (-3,-1.5) -- (3,-3) (-3,0) -- (3,-1.5) (-3,1.5) -- (3,0) (-3,3) -- (3,1.5);
\draw[red, fuzz] (-3,3) -- (-3,-3);

\begin{scope}[shift={(-3,0)}]
\foreach \i in {0,...,5}
{
        \pgfmathtruncatemacro{\y}{15*\i}
        \draw[red] (0,0)-- (\y:.2) ;
}

\node at (0,0) {};
\end{scope}

\begin{scope}[shift={(-3,-3)}]
\foreach \i in {0,...,5}
{
        \pgfmathtruncatemacro{\y}{15*\i}
        \draw[red] (0,0)-- (\y:.2) ;
}

\node at (0,0) {};
\end{scope}

\begin{scope}[shift={(-3,1.5)}]
\foreach \i in {0,...,5}
{
        \pgfmathtruncatemacro{\y}{15*\i}
        \draw[red] (0,0)-- (\y:.2) ;
}

\node at (0,0) {};
\end{scope}

\begin{scope}[shift={(-3,-1.5)}]
\foreach \i in {0,...,5}
{
        \pgfmathtruncatemacro{\y}{15*\i}
        \draw[red] (0,0)-- (\y:.2) ;
}

\node at (0,0) {};
\end{scope}

\end{scope}
\begin{scope}[shift={(6,0)}]

\draw[red, fuzz] (-3,-1.5) -- (3,-3) (-3,0) -- (3,-1.5) (-3,1.5) -- (3,0) (-3,3) -- (3,1.5);
\draw[red, fuzz] (-3,3) -- (-3,-3);

\draw[red, fuzz] (-0.5,-0.65) .. controls (1.5,-1.15) and (3,-1.85) .. (3,-2.3);

\begin{scope}[shift={(-3,0)}]
\foreach \i in {0,...,5}
{
        \pgfmathtruncatemacro{\y}{15*\i}
        \draw[red] (0,0)-- (\y:.2) ;
}

\node at (0,0) {};
\end{scope}

\begin{scope}[shift={(-3,-3)}]
\foreach \i in {0,...,5}
{
        \pgfmathtruncatemacro{\y}{15*\i}
        \draw[red] (0,0)-- (\y:.2) ;
}

\node at (0,0) {};

\end{scope}

\begin{scope}[shift={(-3,1.5)}]
\foreach \i in {0,...,5}
{
        \pgfmathtruncatemacro{\y}{15*\i}
        \draw[red] (0,0)-- (\y:.2) ;
}

\node at (0,0) {};
\end{scope}

\begin{scope}[shift={(-3,-1.5)}]
\foreach \i in {0,...,5}
{
        \pgfmathtruncatemacro{\y}{15*\i}
        \draw[red] (0,0)-- (\y:.2) ;
}

\node at (0,0) {};
\end{scope}

\end{scope}

\end{scope}

\begin{scope}[]
\begin{scope}[shift={(5,-1.5)}]

\draw  (-3,-3) rectangle (3,3);
\draw[blue, fuzz] (-3,-1.5) -- (3,-3) (-3,0) -- (3,-1.5) (-3,1.5) -- (3,0) (-3,3) -- (3,1.5);
\draw[blue, fuzz] (-3,3) -- (-3,-3);

\begin{scope}[shift={(-3,0)}]
\foreach \i in {0,...,5}
{
        \pgfmathtruncatemacro{\y}{15*\i}
        \draw[blue] (0,0)-- (\y:.2) ;
}

\node at (0,0) {};
\end{scope}

\begin{scope}[shift={(-3,-3)}]
\foreach \i in {0,...,5}
{
        \pgfmathtruncatemacro{\y}{15*\i}
        \draw[blue] (0,0)-- (\y:.2) ;
}

\node at (0,0) {};
\end{scope}

\begin{scope}[shift={(-3,1.5)}]
\foreach \i in {0,...,5}
{
        \pgfmathtruncatemacro{\y}{15*\i}
        \draw[orange] (0,0)-- (\y:.2) ;
}

\node at (0,0) {};
\end{scope}

\begin{scope}[shift={(-3,-1.5)}]
\foreach \i in {0,...,5}
{
        \pgfmathtruncatemacro{\y}{15*\i}
        \draw[blue] (0,0)-- (\y:.2) ;
}

\node at (0,0) {};
\end{scope}

\end{scope}

\end{scope}

\end{tikzpicture}
         \caption{The unique intersection between the FLTZ stop $\stp_\nn$ and the positive pushoff $\Phi^\rho(\tilde\Psi_{\tau_i}^t(\stp_\nn))$ is circled in the figure.}
        \label{fig:stopIntersection}
    \end{figure}
Therefore $\mathcal E_i = K_i \times K_i$. Since a linking disk is spherical, \cref{sphericalobjecttwist} identifies the kernel $[K_i \times K_i \to \Delta_{\XA_{\mathrm{int}}}]$ with the spherical twist $T_{K_i}$. Hence
\[
\mathcal F_{\tau_i}\cong T_{K_i}. \qedhere
\]
\end{proof}

We will show in \cref{prop:horizontalLinkingDisks} that $K_i$ corresponds under \cref{thm:hms} to the twist $\Os_0 (-i)$ of the structure sheaf of the origin.
However, we do not do so immediately as our argument will use the Bondal--Thomsen generators and a local analysis of the FLTZ stop from \cref{subsec:BTcollection}.

Instead, we now turn to the simpler $\mathcal{F}_\rho$. 
From \cref{prop:Bn1AsConf}, the underlying graph isotopy is given by a $1/n$ translation in the $q_2$ direction and hence lifts to an isotopy of $\stp_\nn$ that moves the component corresponding to the ray generated by $(1,n)$ according to $q_1 + n q_2 = t$ for $t \in [0,1]$. 
On Lagrangians wrapped up to the stop, it follows that $\Phi_\rho$ acts by ``dragging'' the end of the Lagrangian along this isotopy. 
However, we have another description of such a Hamiltonian in \cref{subsec:troplag}, namely, $H_{cyl}^F$ where $F$ is a support function for the divisor $-D_{(1,n)}$. 
While this observation gives a rather explicit geometric description of the autoequivalence, it also allows us to identify its mirror.

\begin{prop}
\label{prop:rhoCompat}
Let $\mathcal F_\rho$ denote the autoequivalence induced by the annular generator $\rho$.
Under the HMS equivalence of \cref{thm:hms}, one has
\[
\mathcal F_\rho \longleftrightarrow -\otimes \mathcal O(-1).
\]
\end{prop}
\begin{proof}
From the discussion preceding the proposition, we are computing the autoequivalence induced by the Hamiltonian $H^F_{cyl}$ in the notation of \cref{subsec:troplag}, where $F$ is the support function for the divisor $-D_{(1,n)}$.
If we were working with a toric variety, \cite[Theorem 3.13]{hanlon2022aspects}, which is a translation of \cite[Theorem 4.5]{hanlon2019monodromy} to the partially wrapped setting, would imply that this is mirror to tensoring by $\Os (-D_{(1,n)})$. 
However, \cref{thm:hms} extends the HMS equivalence used in \cite{hanlon2019monodromy,hanlon2022aspects} to toric Deligne--Mumford stacks following the same proof scheme. 
In particular, tracing morphism spaces between tropical Lagrangian sections as in the proof of \cite[Theorem 4.5]{hanlon2019monodromy} shows the functor induced by the time-$1$ flow of $H^F_{cyl}$ corresponds under \cref{thm:hms} to tensoring by the line bundle corresponding to $F$.

Therefore, in our setting, we indeed obtain that $\mathcal F_\rho$ is mirror to $-\otimes \Os(-D_{(1,n)})$, and under our conventions, $\Os(-D_{(1,n)}) = \Os(-1)$.
\end{proof}

\subsection{The Bondal--Thomsen collection and a local model for mapping cones}
\label{subsec:BTcollection}
We next turn to the Bondal--Thomsen collection, which we will use to identify the linking disks in \cref{thm:sphericalTwists} on the coherent side.
The underlying stratification of the torus induced by the FLTZ skeleton is called the Bondal stratification \cite{Bondal}.
One can combinatorially assign a line bundle to every stratum of the Bondal stratification following \cite[Section 2.3]{hanlon2024resolutions} or \cite{favero2023rouquier}.
In our example of interest, this assignment is depicted on the far left image of \Cref{fig:tcollection} and described in \cref{example:tcollection}.

This subsection records a local 2-dimensional mutation rule rather than a global theorem: an algorithm for more general stratifications would be desirable, but here we only explain how the Bondal--Thomsen collection changes under a single local modification.
It would be interesting to relate this move (and possible generalizations) to the theory developed in \cite{berkesch2025cellularfreeresolutionsnormalizations}.

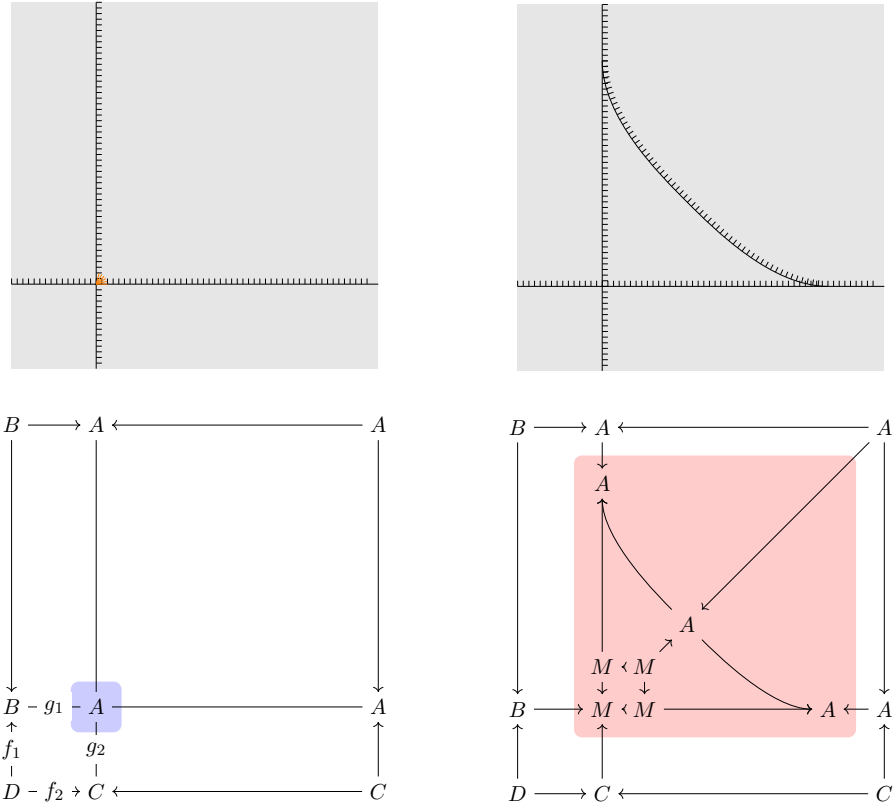
\begin{figure}[h]
    \centering
    \begin{subfigure}{.4\linewidth}
        \centering
        \scalebox{.75}{\begin{tikzpicture}
    
\fill[rounded corners, blue!20]  (0.55,-1.55) rectangle (1.45,-2.45);
\node (v2) at (-0.5,-2) {$B$};
\node (v1) at (-0.5,3) {$B$};
\node (v3) at (1,3) {$A$};
\node (v5) at (6,3) {$A$};
\node (v6) at (6,-2) {$A$};
\node (v7) at (6,-3.5) {$C$};
\node (v8) at (1,-3.5) {$C$};
\node (v9) at (-0.5,-3.5) {$D$};
\draw[->]  (v1) edge (v2);
\draw[->]  (v5) edge (v6);
\draw[->]  (v7) edge (v6);
\draw[->]  (v9) edge node[midway,fill=white]{$f_1$} (v2);
\draw[->]  (v9) edge node[midway,fill=white]{$f_2$} (v8);
\draw[->]  (v1) edge (v3);
\draw[->]  (v5) edge (v3);
\draw[->]  (v7) edge (v8);

\node (v16) at (1,-2) {$A$};
\draw  (v3) edge (v16);
\draw  (v6) edge (v16);
\draw  (v8) edge node[midway,fill=white]{$g_2$} (v16);
\draw  (v2) edge node[midway,fill=white]{$g_1$} (v16);

\begin{scope}[shift={(8,7.5)}]

    \fill[fill=gray!20]  (-8.5,3) rectangle (-2,-3.5);
    \draw[fuzz] (-7,3) -- (-7,-3.5) (-8.5,-2) -- (-2,-2);

    \begin{scope}[shift={(-7,-2)}]
    \foreach \i in {0,...,5}
    {
            \pgfmathtruncatemacro{\y}{15*\i}
            \draw[orange] (0,0)-- (\y:.2) ;
    }
    
    \node at (0,0) {};
    \end{scope}

\end{scope}

\end{tikzpicture} }
        \caption{The local model for the resolution $C_\bullet(S^\phi, O^\phi)$ at a 0-stratum.}
        \label{fig:beforeMutation}
    \end{subfigure}
    \begin{subfigure}{.4\linewidth}
        \centering
        \scalebox{.75}{\begin{tikzpicture}
    
\fill[red!20, rounded corners]  (0.5,2.5) rectangle (5.5,-2.5);
\node (v2) at (-0.5,-2) {$B$};
\node (v1) at (-0.5,3) {$B$};
\node (v3) at (1,3) {$A$};
\node (v5) at (6,3) {$A$};
\node (v6) at (6,-2) {$A$};
\node (v7) at (6,-3.5) {$C$};
\node (v8) at (1,-3.5) {$C$};
\node (v9) at (-0.5,-3.5) {$D$};
\draw[->]  (v1) edge (v2);
\draw[->]  (v5) edge (v6);
\draw[->]  (v7) edge (v6);
\draw[->]  (v9) edge (v2);
\draw[->]  (v9) edge (v8);
\draw[->]  (v1) edge (v3);
\draw[->]  (v5) edge (v3);
\draw[->]  (v7) edge (v8);

\node (v13) at (1,2) {$A$};
\node (v14)at (2.5,-0.5) {$A$};
\node (v10) at (5,-2) {$A$};
\node (v11) at (1.75,-2) {$M$};
\node (v4) at (1,-2) {$M$};
\node (v12) at (1,-1.25) {$M$};
\node (v15) at (1.75,-1.25) {$M$};
\draw[->]  (v2) edge (v4);
\draw[->]  (v8) edge (v4);
\draw[->]  (v6) edge (v10);
\draw[->]  (v11) edge (v10);
\draw[->]  (v11) edge (v4);
\draw[->]  (v12) edge (v4);
\draw[->]  (v12) edge (v13);
\draw[->]  (v3) edge (v13);
\draw[->] (v14) .. controls (2,0) and (1,1) .. (v13);
\draw[->] (v14) .. controls (3,-1) and (4,-2) .. (v10);
\draw[->]  (v15) edge (v14);
\draw[->]  (v15) edge (v11);
\draw[->]  (v15) edge (v12);
\draw[->] (v5) edge (v14);

\begin{scope}[shift={(8,7.5)}]

    \fill[fill=gray!20]  (-8.5,3) rectangle (-2,-3.5);
    \draw[fuzz] (-7,3) -- (-7,-3.5) (-8.5,-2) -- (-2,-2);
    
    \draw[fuzz] (-7,2) .. controls (-7,1) and (-6,0) .. (-5.5,-0.5) .. controls (-5,-1) and (-4,-2) .. (-3,-2);
    
    \end{scope}
\end{tikzpicture} }
        \caption{Modifying the stratification and labeling.}
        \label{fig:afterMutation}
    \end{subfigure}
 
   \caption{Variation of the Legendrian skeleton can be reinterpreted as a variation of stratification for $C_\bullet(S^\phi, O^\phi)$.}
\end{figure}
Suppose that we are in the situation where the dimension of our torus is two and we have a zero-dimensional stratum whose FLTZ stop locally looks like that for affine space, so that there is a resolution $C_\bullet(S^\phi, O^\phi)$ that has the shape as drawn in \Cref{fig:beforeMutation}.
Consider the object of the derived category given by the mapping cone 
\[M:=  \cone(B\oplus C \xleftarrow{f_1\oplus f_2} D)\]
\begin{rem}
    We make a quick remark about the existence of a quasi-isomorphism between $M$ and $A$, which is equivalent to having a short exact sequence
    \[A\xleftarrow{g_1\oplus g_2}  B\oplus C \xleftarrow{f_1\oplus f_2} D.\]
    This occurs when we modify our local model so that the two 1-dimensional strata do not share a higher-dimensional cone (i.e., we are at a $0$-stratum that does not correspond to a 2-cone of $\Sigma$ in the FLTZ construction). This differs from the local model drawn in \Cref{fig:beforeMutation} via removal of the orange stratum.
\end{rem}
Consider the chain maps  
\begin{align*}
    g\colon M\to A \text{ induced by } g_1\oplus g_2 && \iota_B: B\to M && \iota_C: C\to M.
\end{align*}
Then the following sequence (read from right to left) is exact (with chain nullhomotopy indicated from left to right):
\[
\begin{tikzpicture}

    \fill[blue!20, rounded corners]  (-10.5,2) rectangle (-9.5,1);
    \fill[red!20, rounded corners]  (-7.2581,2) rectangle (3,1);
    \node (v2) at (-6,1.5) {$A^{\oplus 2}\oplus M$};
    \node (v1) at (-2,1.5) {$A\oplus M^{\oplus 2}$};
    \node (v3) at (2.5,1.5) {$M$};
    \node (v4) at (-10,1.5) {$A$};
    \draw  (v1) edge[bend right, ->] node[midway, above]{$d^1=\begin{pmatrix} 1 & 0 &- g \\ 1 & g & 0\\ 0& 1& 1 \end{pmatrix}$}(v2);
    \draw  (v3) edge[bend right, ->] node[midway, above]{$d^2=\begin{pmatrix}g\\-1\\ 1\end{pmatrix}$}(v1);
    
    \draw  (v2) edge[bend right, ->] node[midway, below] {$h^1=\begin{pmatrix}0&1 &0\\ 0& 0& 0\\ 0& 0& 1 \end{pmatrix}$} (v1);
    \draw  (v1) edge[bend right, ->] node[midway, below]{$h^2=\begin{pmatrix} 0& -1& 0\end{pmatrix}$} (v3);
    \draw  (v2) edge[bend right, ->] node[midway, above]{$d^0=\begin{pmatrix}1 & -1 & g\end{pmatrix}$} (v4);
    \draw  (v4) edge[bend right, ->] node[midway, below]{$h^0=\begin{pmatrix}1 \\ 0 \\ 0\end{pmatrix}$} (v2);
    \end{tikzpicture}
\]
As a result, we may replace $C_\bullet(S^\phi, O^\phi)$ with the homotopic resolution drawn in \cref{fig:afterMutation}.

\begin{example}
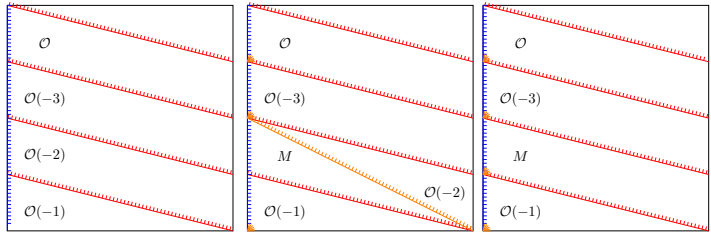
\begin{figure}[h]
        \centering
        \scalebox{.5}{

\begin{tikzpicture}
\begin{scope}[]

\draw  (-3,-3) rectangle (3,3);
\draw[red, fuzz] (-3,-1.5) -- (3,-3) node (v23) {} (-3,0) -- (3,-1.5) node (v25) {} (-3,1.5) -- (3,0) node (v20) {} (-3,3) -- (3,1.5) node (v19) {};
\draw[blue, fuzz] (-3,3) -- (-3,-3);

\end{scope}

\node (v13) at (-2,-2.5) {$\mathcal O(-1)$};
\node (v9) at (-2,-1) {$\mathcal O(-2)$};
\node (v6) at (-2,0.5) {$\mathcal O(-3)$};
\node (v1) at (-2,2) {$\mathcal O$};

\end{tikzpicture}         

\begin{tikzpicture}
\begin{scope}[]

\draw  (-3,-3) rectangle (3,3);
\draw[red, fuzz] (-3,-1.5) -- (3,-3) (-3,0) -- (3,-1.5) (-3,1.5) -- (3,0) (-3,3) -- (3,1.5);
\draw[blue, fuzz] (-3,3) -- (-3,-3);

\begin{scope}[shift={(-3,0)}]
\foreach \i in {0,...,5}
{
        \pgfmathtruncatemacro{\y}{15*\i}
        \draw[orange] (0,0)-- (\y:.2) ;
}

\node at (0,0) {};
\end{scope}

\begin{scope}[shift={(-3,-3)}]
\foreach \i in {0,...,5}
{
        \pgfmathtruncatemacro{\y}{15*\i}
        \draw[orange] (0,0)-- (\y:.2) ;
}

\node at (0,0) {};
\end{scope}

\begin{scope}[shift={(-3,1.5)}]
\foreach \i in {0,...,5}
{
        \pgfmathtruncatemacro{\y}{15*\i}
        \draw[orange] (0,0)-- (\y:.2) ;
}

\node at (0,0) {};
\end{scope}

\end{scope}

\draw[orange, fuzz] (-3,0) .. controls (0,-1.5) and (0,-1.5) .. (3,-3);

\node at (-2,-2.5) {$\mathcal O(-1)$};
\node at (2.25,-2) {$\mathcal O(-2)$};
\node at (-2,0.5) {$\mathcal O(-3)$};
\node at (-2,2) {$\mathcal O$};
\node at (-2,-1) {$M$};
\end{tikzpicture}         
\begin{tikzpicture}
\begin{scope}[]

\draw  (-3,-3) rectangle (3,3);
\draw[red, fuzz] (-3,-1.5) -- (3,-3) (-3,0) -- (3,-1.5) (-3,1.5) -- (3,0) (-3,3) -- (3,1.5);
\draw[blue, fuzz] (-3,3) -- (-3,-3);

\begin{scope}[shift={(-3,0)}]
\foreach \i in {0,...,5}
{
        \pgfmathtruncatemacro{\y}{15*\i}
        \draw[orange] (0,0)-- (\y:.2) ;
}

\node at (0,0) {};
\end{scope}

\begin{scope}[shift={(-3,-3)}]
\foreach \i in {0,...,5}
{
        \pgfmathtruncatemacro{\y}{15*\i}
        \draw[orange] (0,0)-- (\y:.2) ;
}

\node at (0,0) {};
\end{scope}

\begin{scope}[shift={(-3,1.5)}]
\foreach \i in {0,...,5}
{
        \pgfmathtruncatemacro{\y}{15*\i}
        \draw[orange] (0,0)-- (\y:.2) ;
}

\node at (0,0) {};
\end{scope}

\begin{scope}[shift={(-3,-1.5)}]
\foreach \i in {0,...,5}
{
        \pgfmathtruncatemacro{\y}{15*\i}
        \draw[orange] (0,0)-- (\y:.2) ;
}

\node at (0,0) {};
\end{scope}

\end{scope}

\node at (-2,-2.5) {$\mathcal O(-1)$};
\node at (-2,-1) {$M$};
\node at (-2,0.5) {$\mathcal O(-3)$};
\node at (-2,2) {$\mathcal O$};

\end{tikzpicture}         }
        \caption{The effect of a local modification on the Bondal--Thomsen collection.}
        \label{fig:tcollection}
\end{figure}
For the example of $\XB_{\Sigma_\nn}$, we have this local model at each $0$-dimensional stratum of the FLTZ skeleton.
For a fixed $0$-stratum (indexed by $i\in \{0, \ldots, \n\}$), the sheaves labeling the resolution $C_\bullet(S^\phi, O^\phi)$ from \cite{hanlon2024resolutions} are:
    \begin{align*}
    A=\mathcal O(-i)&& B=\mathcal O(-i+1) && C=\mathcal O(-i+1) && D=\mathcal O(-i).
    \end{align*}
    By comparison with \cref{eq:Koszul}, we see that the complex $M$ is quasi-isomorphic to the ideal sheaf $I_0(-i)$.
    The variation of skeleton therefore corresponds to the variation of stratification and labeling drawn in \Cref{fig:tcollection}.
    \label{example:tcollection}
\end{example}

With an understanding of these local models in hand, we are now prepared to identify the linking disks appearing in \cref{thm:sphericalTwists} and prove that they correspond to the expected spherical twists.

\begin{prop}
\label{prop:horizontalLinkingDisks}
For $i\in\{1,\dots,\n\}$, let $K_i$ be the linking disk of the horizontal stop component of $p_{lag}(\stp_\nn)$ indexed by the $0$-stratum $i$.
Under the HMS equivalence of \cref{thm:hms}, one has
\[
K_i \longleftrightarrow \mathcal O_0(-i).
\]
Equivalently, after reversing the cyclic order of the horizontal components, the corresponding linking disk is mirror to $\mathcal O_0(i)$.
\end{prop}
\begin{proof}
Fix $i$. In the local model at the $0$-stratum indexed by $i$, the labels are
\[
A=D=\mathcal O(-i),\qquad B=C=\mathcal O(-i+1),
\]
and these agree with the line bundles assigned to small negative pushoffs of the cotangent fibers at these points under \cref{thm:hms}. 
The mutated object is
\[
M_i:=\cone(B\oplus C \xleftarrow{f_1\oplus f_2} D).
\]

Comparing with the Koszul resolution \cref{eq:Koszul}, we obtain
\[
M_i\simeq I_0(-i),
\]
the ideal sheaf of the origin twisted by $\mathcal O(-i)$.

Reading off the homotopic resolution after the local modification, the corresponding horizontal $1$-stratum is represented by the cone of the canonical map $M_i\to A$. Under the mirror functor, its linking disk is therefore identified with
\[
\cone\bigl(I_0(-i)\to \mathcal O(-i)\bigr)
\]
by \cite[Proposition 1.37]{ganatra2024sectorial}.
The short exact sequence
\[
0\to I_0(-i)\to \mathcal O(-i)\to \mathcal O_0(-i)\to 0
\]
shows that this cone is $\mathcal O_0(-i)$.
\end{proof}

\begin{cor}
\label{cor:SeidelThomasComparison}
Under the HMS equivalence, one has
\[
\mathcal F_{\tau_i}\longleftrightarrow T_{\mathcal O_0(-i)}.
\]
After the cyclic relabeling $i\mapsto -i$, this is precisely the Seidel--Thomas action by twists about
\[
\mathcal O_0(1),\dots,\mathcal O_0(\n).
\]
In particular, the subgroup generated by the $\tau_i$ acts by the Seidel--Thomas spherical twists.
\end{cor}
\begin{proof}
Combine \cref{thm:sphericalTwists,prop:horizontalLinkingDisks} with \cref{prop:Bn1AsConf}. Since $\Pic(\XB_{\Sigma_\nn})\cong \ZZ_\nn$, the collections $\{\mathcal O_0(-i)\}$ and $\{\mathcal O_0(i)\}$ differ only by a cyclic relabeling.
\end{proof}

\subsection{Equivalences between VGIT chambers}
\label{subsec:VGITEquivalence}
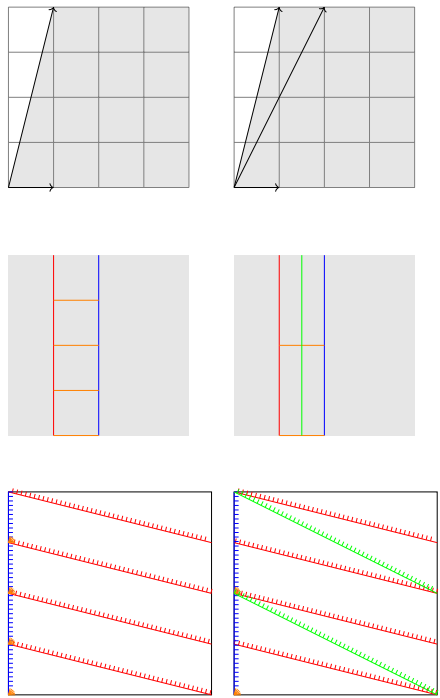
\begin{figure}[tbp]
    \centering
    \scalebox{.6}{
\begin{tikzpicture}

\begin{scope}[shift={(-5,3.5)}]
\fill[gray!20] (-1,2) -- (-2,-2) -- (2,-2) -- (2,2) -- (-1,2);
\draw[help lines, step=1cm] (-2,-2) grid (2,2);
\draw[->] (-2,-2) -- (-1,-2);
\draw[->] (-2,-2) -- (-1,2);
\end{scope}

\begin{scope}[shift={(0,3.5)}]
\fill[gray!20] (-2,-2) rectangle (2,2);
\draw[red] (-1,-2) -- (-1,2);
\draw[blue] (0,-2) -- (0,2);
\draw[orange] (-1,-2) -- (0,-2) (-1,-1) -- (0,-1) (-1,0) -- (0,0) (-1,1) -- (0,1);
\end{scope}

\begin{scope}[shift={(5,3.5)}]
\draw (-2,-2) rectangle (2,2);
\draw[red, fuzz] (-2,-1) -- (2,-2) (-2,0) -- (2,-1) (-2,1) -- (2,0) (-2,2) -- (2,1);
\draw[blue, fuzz] (-2,2) -- (-2,-2);
\begin{scope}[shift={(-2,0)}]
\foreach \i in {0,...,5}
{
        \pgfmathtruncatemacro{\y}{15*\i}
        \draw[orange] (0,0) -- (\y:.13);
}
\node at (0,0) {};
\end{scope}
\begin{scope}[shift={(-2,-2)}]
\foreach \i in {0,...,5}
{
        \pgfmathtruncatemacro{\y}{15*\i}
        \draw[orange] (0,0) -- (\y:.13);
}
\node at (0,0) {};
\end{scope}
\begin{scope}[shift={(-2,1)}]
\foreach \i in {0,...,5}
{
        \pgfmathtruncatemacro{\y}{15*\i}
        \draw[orange] (0,0) -- (\y:.13);
}
\node at (0,0) {};
\end{scope}
\begin{scope}[shift={(-2,-1)}]
\foreach \i in {0,...,5}
{
        \pgfmathtruncatemacro{\y}{15*\i}
        \draw[orange] (0,0) -- (\y:.13);
}
\node at (0,0) {};
\end{scope}
\end{scope}

\begin{scope}[shift={(-5,-3.5)}]
\fill[gray!20] (-1,2) -- (-2,-2) -- (2,-2) -- (2,2) -- (-1,2);
\draw[help lines, step=1cm] (-2,-2) grid (2,2);
\draw[->] (-2,-2) -- (-1,-2);
\draw[->] (-2,-2) -- (-1,2);
\draw[->] (-2,-2) -- (0,2);
\end{scope}

\begin{scope}[shift={(0,-3.5)}]
\fill[gray!20] (-2,-2) rectangle (2,2);
\draw[red] (-1,-2) -- (-1,2);
\draw[green] (-.5,-2) -- (-.5,2);
\draw[blue] (0,-2) -- (0,2);
\draw[orange] (-1,-2) -- (0,-2) (-1,0) -- (0,0);
\end{scope}

\begin{scope}[shift={(5,-3.5)}]
\draw (-2,-2) rectangle (2,2);
\draw[red, fuzz] (-2,-1) -- (2,-2) (-2,0) -- (2,-1) (-2,1) -- (2,0) (-2,2) -- (2,1);
\draw[blue, fuzz] (-2,2) -- (-2,-2);
\begin{scope}[shift={(-2,0)}]
\foreach \i in {0,...,5}
{
        \pgfmathtruncatemacro{\y}{15*\i}
        \draw[orange] (0,0) -- (\y:.13);
}
\node at (0,0) {};
\end{scope}
\begin{scope}[shift={(-2,-2)}]
\foreach \i in {0,...,5}
{
        \pgfmathtruncatemacro{\y}{15*\i}
        \draw[orange] (0,0) -- (\y:.13);
}
\node at (0,0) {};
\end{scope}
\draw[fuzz,green] (-2,0) -- (2,-2) (-2,2) -- (2,0);
\end{scope}

\end{tikzpicture} }
    \caption{Fan, Lagrangian, and front projections for different VGIT chambers.}
    \label{fig:VGITChambers}
\end{figure}
    Given $I=\{i_0<\cdots<i_k\}\subset \{0, 1, 2, 3, \cdots, \nn\}$ with $i_0=0$ and $i_k=\nn$, let $(\Sigma^I_\nn,\beta^I_\nn)$ be the stacky fan in $N=\ZZ^2$ whose underlying fan $\Sigma^I_\nn$ has rays $\rho_i=\RR_{\geq 0}\cdot \langle 1, i\rangle$ for $i\in I$ and 2-dimensional cones $\rho_{i_{a-1}}+\rho_{i_a}$ for $a=1,\ldots,k$, and whose stacky map is
    \[
        \beta^I_\nn \colon \ZZ^I \to N,\qquad e_i\mapsto \langle 1,i\rangle.
    \]
    We abbreviate $\XB_{\Sigma^I_\nn,\beta^I_\nn}$ and $\stp_{\Sigma^I_\nn,\beta^I_\nn}$ by $\XB_{\Sigma^I_\nn}$ and $\stp_{\Sigma^I_\nn}$, respectively.

    Whenever $I\subset I'$, we have that $\XB_{\Sigma^{I'}_\nn}\to \XB_{\Sigma^I_\nn}$ is a crepant partial resolution of singularities. The categories $\Db(\XB_{\Sigma^{I'}_\nn}), \Db(\XB_{\Sigma^I_\nn})$ are (noncanonically) equivalent. In the language of \cref{subsec:Bside} these fans correspond to different chambers of the GIT parameter space.
    
    This derived equivalence can be observed from the perspective of variation of skeleta. In the Lagrangian projection, the associated stop $\stp_{\Sigma_\nn^I}$ consists of horizontal circles at heights $i\in I$, and $i-j$ equally spaced vertical line segments connecting the circles corresponding to adjacent indices $i> j$. The associated graph has compact faces of area 1.
     
   For fixed $\nn$, all such Lagrangian projections of $\stp_{\Sigma_\nn^I}$ are isotopic via area-preserving isotopy. For example, in \Cref{fig:VGITChambers} we draw two such fans and the Lagrangian projections of the FLTZ stops. One isotopy realizing their derived equivalence comes from rotating the 2nd and 4th vertical line segments by 90 degrees (although there are many non-equivalent isotopies).

    By applying the variation of the Bondal--Thomsen collection to the front projection of this isotopy (following \cref{subsec:BTcollection}), we can compute the derived equivalence between these categories at the level of generating objects.

    \begin{rem}
    We could consider the ``Cox Skeleton'' for the $A_1$ singularity, as drawn in \Cref{fig:coxSkeleton}. 
    The partially wrapped Fukaya category associated to this skeleton agrees with the Cox category (as defined in \cite{ballard2025king}) for the resolved $A_1$ singularity; it also matches the skeleton from \cite{donovan2021mirror} at the discriminant point of the schober.
    While this Cox Legendrian is embedded, the isotopy of non-embedded graphs in the plane corresponding to a braid twist would cause a self-intersection to form in the Legendrian lift, and we thus expect its autoequivalence group to be smaller.
    \end{rem}

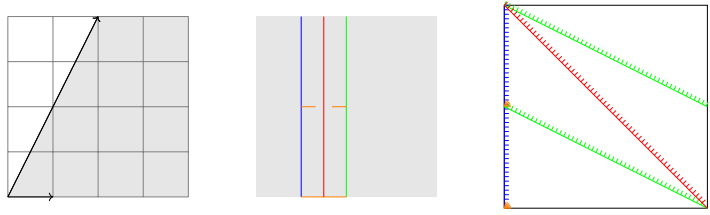
\begin{figure}[tbp]
    \centering
    \scalebox{.6}{
\begin{tikzpicture}

\begin{scope}[shift={(-5,0)}]
\fill[gray!20] (0,2) -- (-2,-2) -- (2,-2) -- (2,2) -- (0,2);
\draw[help lines, step=1cm] (-2,-2) grid (2,2);
\draw[->] (-2,-2) -- (-1,-2);
\draw[->] (-2,-2) -- (0,2);
\end{scope}

\begin{scope}[shift={(0,0)}]
\fill[gray!20] (-2,-2) rectangle (2,2);
\draw[blue] (-.5,-2) -- (-.5,2);
\draw[red] (0,-2) -- (0,2);
\draw[green] (.5,-2) -- (.5,2);
\draw[orange] (-.5,-1) -- (.5,-1) (-.5,0) -- (.5,0);
\node[circle, fill=gray!20, inner sep=1.5pt] at (0,0) {};
\end{scope}

\begin{scope}[shift={(5,0)}]
\draw (-2,-2) rectangle (2,2);
\draw[blue, fuzz] (-2,2) -- (-2,-2);
\begin{scope}[shift={(-2,0)}]
\foreach \i in {0,...,5}
{
        \pgfmathtruncatemacro{\y}{15*\i}
        \draw[orange] (0,0) -- (\y:.13);
}
\node at (0,0) {};
\end{scope}
\begin{scope}[shift={(-2,-2)}]
\foreach \i in {0,...,5}
{
        \pgfmathtruncatemacro{\y}{15*\i}
        \draw[orange] (0,0) -- (\y:.13);
}
\node at (0,0) {};
\end{scope}
\draw[fuzz,green] (-2,2) -- (2,0) (-2,0) -- (2,-2);
\draw[fuzz,red] (-2,2) -- (2,-2);
\end{scope}

\end{tikzpicture} }
    \caption{The Cox Skeleton and its Lagrangian and front projections}
    \label{fig:coxSkeleton}
\end{figure}

\subsection{Partial Resolutions of the \texorpdfstring{$A_{\n}$}{An-1} singularity}
\label{subsec:partialResolutions}
As another consequence, we can recover a result of \cite{donovan2015mixed} on partial resolutions of the $A_{\n}$ singularity.
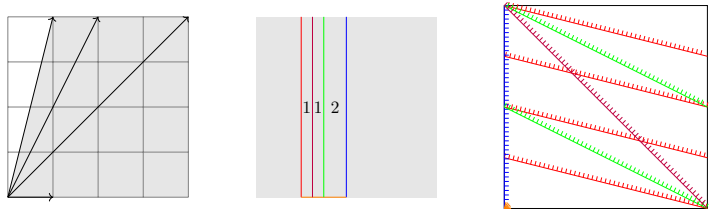
\begin{figure}[tbp]
    \centering
    \scalebox{.6}{

\begin{tikzpicture}

\begin{scope}[shift={(-5,0)}]
\fill[gray!20] (-1,2) -- (-2,-2) -- (2,-2) -- (2,2) -- (-1,2);
\draw[help lines, step=1cm] (-2,-2) grid (2,2);
\draw[->] (-2,-2) -- (-1,-2);
\draw[->] (-2,-2) -- (-1,2);
\draw[->] (-2,-2) -- (0,2);
\draw[->] (-2,-2) -- (2,2);
\end{scope}

\begin{scope}[shift={(0,0)}]
\fill[gray!20] (-2,-2) rectangle (2,2);
\draw[red] (-.75,-2) -- (-.75,2);
\draw[purple] (-.25,-2) -- (-.25,2);
\draw[green] (.25,-2) -- (.25,2);
\draw[blue] (.75,-2) -- (.75,2);
\draw[orange] (-.75,-2) -- (.75,-2);
\node at (-.6,0) {$1$};
\node at (-.25,0) {$1$};
\node at (.25,0) {$2$};
\end{scope}

\begin{scope}[shift={(5,0)}]
\draw (-2,-2) rectangle (2,2);
\draw[red, fuzz] (-2,-1) -- (2,-2) (-2,0) -- (2,-1) (-2,1) -- (2,0) (-2,2) -- (2,1);
\draw[blue, fuzz] (-2,2) -- (-2,-2);
\begin{scope}[shift={(-2,-2)}]
\foreach \i in {0,...,5}
{
        \pgfmathtruncatemacro{\y}{15*\i}
        \draw[orange] (0,0) -- (\y:.13);
}
\node at (0,0) {};
\end{scope}
\draw[fuzz,green] (-2,0) -- (2,-2) (-2,2) -- (2,0);
\draw[fuzz,purple] (-2,2) -- (2,-2);
\end{scope}

\end{tikzpicture} }
    \caption{Partial resolution of $A_3$ singularity. The regions can be braided around each other provided that the permutation sends the regions labeled $1$ to each other.}
\end{figure}
As before, let $\XB_{\Sigma^I_\nn}$ be the toric stack associated to a partial resolution of the $A_{\n}$ singularity. Denote by $\underline{X}_{\Sigma^I_\nn}$ the underlying space of this stack. By \cite{donovan2015mixed}, there is an equivalence of categories
\[D^b\Coh(\XB_{\Sigma_\nn^I})/\langle \mathcal O_{z_i}(j)\rangle_{i\in \Sigma_\nn^I(2),\mathcal O_{z_i}(j)\not\cong \mathcal O_{z_i} }\to D^b\Coh(\underline{X}_{\Sigma^I_\nn}),\]
where $\mathcal O_{z_i}(j)$ are the twists of the skyscraper sheaf at the stacky points of $\XB_{\Sigma_{\nn}^I}$. 
The objects $\mathcal O_{z_i}(j)$ are mirror to the linking disks of $\stp_\nn^I$ corresponding to the $q_2\neq 0$ horizontal line segments in the Lagrangian projection of $\stp_\nn^I$ by the same argument as in \cref{prop:horizontalLinkingDisks}.
By \cref{thm:hms}, we therefore have an equivalence of categories $D^b\Coh(\underline{X}_{\Sigma^I_\nn})\cong \mathcal W(\XA,\underline{\stp}^I_\nn)$ where $\underline{\stp}^I_\nn$ is the Legendrian lift of the graph $\underline{\phi}_\nn^I \colon (G,\bullet)\to M$ consisting of the segment $[0, \nn]\times \{0\}$ and cycles $r=i$ for each $i\in I\setminus\{\nn\}$ (where we deduce HMS at the intermediate stack from quotienting the derived category of the smooth stack / localizing the $A$-side via stop removal).

We now record the braid-group actions that arise for these partial resolutions.
The graph $\underline{\phi}_\nn^I$ has faces of integer areas $i_1-i_0, i_2-i_1, \ldots, i_k-i_{k-1}$ where $I=\{i_0, i_1, \ldots, i_k\}$ with $0=i_0<\cdots<i_k=\nn$.

Let $\pi:\mathcal B_k\to S_k$ denote the homomorphism recording the induced permutation of the strands. Given a coloring function $C:\{1, \ldots, k\}\to \NN$, we define the $C$-partitioned braid group \[\mathcal B_{k, C}:=\{\sigma\in \mathcal B_{k} \st C(\pi(\sigma)(j))=C(j)\text{ for all }j\}\]
to be those braids whose induced permutation fixes the coloring. Given $I$, we produce a coloring function by $C^I(j)=i_j-i_{j-1}$. 
By the same argument as \cref{cor:braidLegendrian}, we have a natural inclusion
\[\mathcal B_{k, C^I}\into \pi_1(\mathcal L^{emb}_{\underline{\stp}^I_\nn}(Y)).\]
This gives an action of the $C$-partitioned braid group\footnote{In \cite{donovan2015mixed}, this is called the mixed braid group.} on $\mathcal W(\XA, \underline{\stp}_\nn^I)$. In a similar fashion, whenever $I, I'\subset \{0, \ldots, \nn\}$ satisfy the property that there exists $\sigma\in S_k$ with $C^I\circ \sigma = C^{I'}$, it follows that $\underline{\stp}^{I'}_\nn\in \mathcal L^{emb}_{\underline{\stp}^I_\nn}(Y)$ and therefore we obtain equivalences
\[D^b\Coh(\underline{X}_{\Sigma^{I'}_\nn})\cong \mathcal W(\XA,\underline{\stp}^{I'}_\nn)\cong \mathcal W(\XA,\underline{\stp}^I_\nn) \cong D^b\Coh(\underline{X}_{\Sigma^I_\nn}).\]
\appendix

\section{Homological mirror symmetry for semiprojective toric Deligne--Mumford stacks}
\label{app:HMS}
As noted in \cref{subsec:Aside}, there are various approaches to toric homological mirror symmetry. 
In this appendix, we outline how to extend the direct geometric approach \cite{abouzaid2009morse} to semiprojective toric Deligne--Mumford stacks over an arbitrary field $\Bbbk$ (or even over $\ZZ$ with respect to the split torus) in the sense of \cite{borisov2005orbifold}. 
We expect that the Floer-theoretic HMS equivalence can be extended to a broader class of toric stacks, namely those satisfying \cite[Condition 1.1]{kuwagaki2020nonequivariant} for which HMS is proved sheaf-theoretically, but doing so would require technical work beyond the scope of this appendix, where we rely on what has been done for toric varieties to the largest extent possible. 

More precisely, recall that a smooth toric Deligne--Mumford stack $\X_{\Sigma, \beta}$ from \cite{borisov2005orbifold} is encoded by the data of a simplicial fan $\Sigma$ in $N_{\RR} = N \otimes \RR$ where $N$ is a lattice and a homomorphism $\beta \colon \ZZ^{\Sigma(1)} \to N$ such that for all $\rho \in \Sigma(1)$ we have $\beta(e_\rho) = b_\rho u_\rho$ for some $b_\rho > 0$ where $e_\rho$ is the standard basis vector of $\ZZ^{\Sigma(1)}$ corresponding to $\rho$ and $u_\rho$ is the primitive generator of $\rho$.
We say that $\mathcal X_{\Sigma, \beta}$ is semiprojective if the underlying coarse moduli space $X_\Sigma$ is a semiprojective toric variety.
Combinatorially, semiprojectivity means that $\Sigma$ has full-dimensional convex support and there is a convex piecewise-linear function on $\Sigma$. 
Equivalently, $\Sigma$ is the normal fan to a full-dimensional lattice polyhedron. 
Throughout this appendix, we will at times abbreviate $\X_{\Sigma, \beta}$ to $\X$.

\begin{df}[{\cite{fang2014coherent}}] \label{def:FLTZ}
The \emph{FLTZ skeleton} mirror to a toric stack is a cylindrical singular Lagrangian in $T^*(M_{\RR}/M) = M_{\RR}/M \times N_{\RR}$ where $M$ is the dual lattice to $N$ given by
\begin{align} \label{eq:fltzskel1}
\LL_{\Sigma, \beta} &= \bigcup_{\sigma \in \Sigma} \{ (m,u) \mid u \in \sigma \text{ and } b_\rho \langle m, u_\rho \rangle \in \ZZ \text{ for all } \rho \in \sigma(1) \} \\ \label{eq:fltzskel2}
& = \bigcup_{\sigma \in \Sigma} \pi(\sigma^{\perp_\beta}) \times \sigma
\end{align}
where $\pi \colon M_\RR \to M_\RR/M$ is the natural projection and
\[ \sigma^{\perp_\beta} = \left\{ m \in M_\RR \mid \text{ there exists } \ell \in \left(\ZZ^{\Sigma(1)}\right)^* \text{ such that } \langle \beta_\RR^* m + \ell , a \rangle = 0 \text{ for all } a \in \langle e_\rho \mid \rho \in \sigma \rangle \right\}.  \]
Further, we set the \emph{FLTZ stop} to be $\stp_{\Sigma,\beta} = \partial_\infty \LL_{\Sigma, \beta}$. 
\end{df}

\begin{rem} In \cite{fang2014coherent} and much of the literature on the coherent-constructible correspondence, \eqref{eq:fltzskel2} differs by a minus sign on the $\sigma$ factor. 
However, the category of sheaves with microsupport in $\LL_{\Sigma, \beta}$ is most naturally identified in \cite{ganatra2024microlocal} with the opposite of the wrapped Fukaya stopped at $\stp_{\Sigma, \beta}$, $\cW(T^*M_\RR/M, \stp_{\Sigma, \beta})$.
Following \cite[Remark 1.2]{ganatra2024microlocal}, one can simply negate the skeleton in order to avoid taking opposite categories, and thus it is natural to work with $\cW(T^*M_\RR/M, \stp_{\Sigma, \beta})$.
\end{rem}

With our notation defined, we can now state the theorem that this appendix is concerned with.

\begin{thm} \label{thm:hms}
    If $\X_{\Sigma, \beta}$ is a smooth semiprojective toric Deligne--Mumford stack, there is a pre-triangulated equivalence
    \[ D^b_{dg}(\X_{\Sigma, \beta}) \simeq \cW( T^*M_{\RR}/M, \stp_{\Sigma, \beta})  \]
    sending a line bundle to the corresponding tropical Lagrangian section.
\end{thm}

Note that we use the term pre-triangulated equivalence to mean that the functor is fully faithful and its image generates (cf. \cite{ganatra2024sectorial}). 

\subsection{The dg-category of line bundles on a toric stack} \label{subsec:dgline}

We now fix a smooth semiprojective toric stack $\X = \X_{\Sigma, \beta}$. In this section, we discuss line bundles on $\X$ and the dg-enhancement of the derived category of $\X$ that we need in \cref{thm:hms}.

Since $\X$ is smooth, the Picard group of $\X$ is isomorphic to the cokernel of the map $\beta^* \colon M \to \ZZ^{\Sigma(1)}$ dual to $\beta$ and given explicitly by $m \mapsto ( \langle m, \beta(e_\rho) \rangle )_{\rho \in \Sigma(1)}$.
That is, any divisor $\sum a_\rho D_\rho$ with $a_\rho \in \ZZ$ gives rise to a line bundle $\Os_\X \left(\sum a_\rho D_\rho\right)$ and
$\Os_\X \left( \sum a_\rho D_\rho\right) \cong \Os_\X \left(\sum c_\rho D_\rho\right)$ if and only if there is an $m \in M$ such that $a_\rho = c_\rho + \langle m, \beta(e_\rho)\rangle$ for all $\rho \in \Sigma(1)$. 
Note that these coincide with equivariant line bundles on $\Bbbk^{\Sigma(1)}$ with respect to the action by the kernel of the map of tori induced by $\beta$. 
For the purposes of homological mirror symmetry, it is convenient to combinatorially encode these line bundles as support functions.
Let
\[ \mathrm{SF}(\Sigma, \beta) = \{ F \colon | \Sigma | \to \RR \mid F \text{ is piecewise linear on } \Sigma \text{ and } b_\rho F(u_\rho) \in \ZZ \text{ for all } \rho \in \Sigma(1) \} .\]
Using that $\Sigma$ is simplicial, we can equivalently (and more canonically) characterize $\mathrm{SF}(\Sigma, \beta)$ as the set of piecewise linear functions on $\Sigma$ that lift to integral piecewise linear functions on the smooth fan $\widehat{\Sigma}$ in $\RR^{\Sigma(1)}$ with cones $\widehat{\sigma} = \langle e_\rho \mid \rho \in \sigma \rangle$ for $\sigma \in \Sigma$.
With that observation, it is natural to define $\iota \colon M \to \mathrm{SF}(\Sigma, \beta)$ by setting $\iota(m)$ to be the support function on $\Sigma$ corresponding to the linear function $\langle \beta^* m, \cdot \rangle$ on $\widehat{\Sigma}$. 
Given this setup, the following is straightforward to verify.

\begin{prop} \label{prop:picardgroup}
    The Picard group of $\X$ is isomorphic to $\mathrm{SF}(\Sigma, \beta)/\iota(M)$.
\end{prop}

Given $F \in \mathrm{SF}(\Sigma, \beta)$, we will denote the associated line bundle by $\Os_{\X} (F)$. 
One can compute the cohomology of $\Os_{\X}(F)$ by a \v{C}ech complex in essentially the same way as for a line bundle on a toric variety.
Namely, the toric stack $\X_{(\Sigma, \beta)}$ with chosen stacky fan $(\Sigma, \beta)$ has a canonical \v{C}ech cover.
For each $\sigma \in \Sigma$ there is an associated toric stack $U_{(\sigma, \beta|_{\widehat{\sigma}})}$.
We note that 
\begin{equation} \label{eq:localglobsec}
    \Gamma\left(U_{(\sigma, \beta|_{\widehat{\sigma}})}, \Os_{\X}(F) \right) = \bigoplus_{m \in M \, \mid \, \langle m , u_\rho \rangle \geq F(u_\rho) \text{ for all } \rho \in \sigma(1)} \Bbbk 
\end{equation}
respecting the natural $M$ grading on the left-hand side and the higher cohomology of $\Os_{\X}(F)$ on $U_{(\sigma, \beta|_{\widehat{\sigma}})}$ vanishes by the identification of $\Os_{\X}(F)$ restricted to $U_{(\sigma, \beta|_{\widehat{\sigma}})}$ with an equivariant line bundle on the affine toric variety associated to $\widehat{\sigma}$.
We choose an ordering $\{ \sigma_i \}_{i = 1, \ldots, |\Sigma(\dim(M))|}$ of the maximal cones of $\Sigma$.
Then, the \v{C}ech complex
\[
\check{C}_m^d(\Os_{\X}(F))=
\bigoplus_{i_0<i_1<\cdots<i_d}
\Gamma\left(
U_{(\sigma_{i_0},\beta|_{\widehat{\sigma}_{i_0}})}\cap\cdots\cap
U_{(\sigma_{i_d},\beta|_{\widehat{\sigma}_{i_d}})}
\right)_m
\]
computes the degree $m \in M$ part of $H^d(\Os_{\X}(F))$ for any $F \in \mathrm{SF}(\Sigma, \beta)$, where the $m$ subscript on the global sections denotes the degree $m$ summand.
We note that
\[
U_{(\sigma_{i_0},\beta|_{\widehat{\sigma}_{i_0}})}\cap\cdots\cap
U_{(\sigma_{i_d},\beta|_{\widehat{\sigma}_{i_d}})}
= U_{(\tau,\beta|_{\widehat{\tau}})}
\]
for the cone $\tau=\sigma_{i_0}\cap\cdots\cap\sigma_{i_d}\in\Sigma$, so \eqref{eq:localglobsec} implies that
\[
\Gamma\left(
U_{(\sigma_{i_0},\beta|_{\widehat{\sigma}_{i_0}})}\cap\cdots\cap
U_{(\sigma_{i_d},\beta|_{\widehat{\sigma}_{i_d}})}
\right)_m
=
\begin{cases}
\Bbbk & \langle m,u_\rho\rangle\geq F(u_\rho) \text{ for all } \rho\in\tau(1),\\
0 & \text{otherwise}.
\end{cases}
\]

Following \cite[Section 6.3]{abouzaid2009morse}, we now define a dg category $\check{C}(\X)$ whose objects are the line bundles $\Os_\X(F)$ and
\[ \hom(\Os_\X(F_1),\Os_\X(F_2)) = \bigoplus_d \bigoplus_{m \in M} \check{C}^d_m(\Os_\X(F_2 - F_1)) \]
equipped with the \v{C}ech differential and cup product. 
We let $D^b_{dg}(\X)$ be the pre-triangulated closure of $\check{C}(\X)$. 

\begin{prop} \label{prop:dbislinebundles}
    $D^b_{dg}(\X)$ is a dg enhancement of $D^b(\X)$.
\end{prop}
\begin{proof}
    By construction and \cref{prop:picardgroup}, the morphism complexes in $\check{C}(\X)$ compute the derived morphisms between all line bundles on $\X$.
    Thus, the claim follows from the fact that line bundles generate $D^b(\X)$.
    Although this generation statement holds more generally for smooth quasi-projective Deligne--Mumford stacks, we use the following concrete reference in the toric setting.
    In the case of interest here, \cite{hanlon2024resolutions} shows that there is a resolution of the diagonal by direct sums of line bundles, which in the proper case implies generation, and in the non-proper case also implies generation after compactifying.
\end{proof}

\subsection{Tropical Lagrangian sections to Morse theory} \label{subsec:troplag}

In this section, we construct the mirror tropical Lagrangian sections to line bundles on $\X$ and analyze their Floer theory analogously to \cite{abouzaid2009morse,hanlon2019monodromy,hanlon2022aspects}. 
Given a piecewise linear function $F$ on $|\Sigma|$, we will extend $F$ continuously to a piecewise linear function on $N_\RR$ (if necessary) and consider smoothings of the form
\[ H^{F, g}(u) = \int_{N_\RR} F(u - x) \eta_{g(r)}(x) \, dx = \int_{N_\RR} F(u - g(r) x) \eta(x) \, dx \]
where $g\colon \RR_{\geq 0} \to \RR_{\geq 0}$ is a smooth function that is constant near $0$, $\eta \colon N_\RR \to \RR$ is a smooth symmetric mollifier function supported on the ball of radius $1$, we set $\eta_t(u) = t^{-n} \eta(u/t)$, and $r = | u |$ is the length of $u$ with respect to a Euclidean metric on $N_\RR$. 

Calculating as in \cite[Lemma 3.1]{hanlon2022aspects}, we have
\begin{equation} \label{eq:convolvederivative}
    dH^{F,g}_u (v) = \int_{N_\RR} \left[ dF_{u - g(r)x} (v) - \frac{g'(r)}{r} dF_{u-g(r)x}(x) u \cdot v \right] \eta(x) \, dx 
\end{equation}
where $\cdot$ is the Euclidean inner product and we have identified $N_\RR$ with its tangent space. 
If we assume that $g(r) = \eps r$ outside of a compact set, then $H^{F,g}$ is $1$-homogeneous outside of a compact set, that is, it is a cylindrical Hamiltonian on $T^*(M_\RR/M) = M_\RR/M \times N_\RR$.
Given $F \in \mathrm{SF}(\Sigma, \beta)$ and $0< a \ll 1$, we set $F_a = F - F_K^a$ where $F_K^a$ is the support function of the (real) scaling of the canonical divisor $-\sum a D_\rho$ and, if necessary, we extend this function so that $F_a(u)$ is a positive piecewise linear function with $\lim_{a \to 0} F_a (u) = + \infty$ monotonically for $u$ outside of $|\Sigma|$. 
We define $H_{cyl}^{F,a,\eps} = H^{F_a, g}$ where $g(r) = \eps r$ outside of a compact set. 
From \cref{eq:convolvederivative} and again calculating as in \cite[Lemma 3.1]{hanlon2022aspects}, we have that
\begin{equation} \label{eq:closetostop} 
\left| dH_{cyl}^{F,a,\eps}(u_\rho) - F(u_\rho) + a\right| \leq \eps C_F 
\end{equation}
for some constant $C_F$ on a complement of a compact subset in the cone 
\[ U^\eps_\rho = \mathrm{cone} \{ u \in N_\RR \mid |u| = 1, u \in \mathrm{star}(\rho), \text{ and } d(u, \sigma) > \eps \text{ for all } \sigma \in \Sigma \text{ such that } \rho \not\in \sigma \} \]
for all $\rho \in \Sigma(1)$. For any $a < 1/(\max_\rho b_\rho)$ we can choose $\eps$ small enough so that $b_\rho dH_{cyl}^{F,a,\eps}(u_\rho) \not \in \ZZ$ on the complement of a compact subset of $U^\eps_\rho$ since $b_\rho F(u_\rho) \in \ZZ$.
Thus, for any $a < 1/(\max_\rho b_\rho)$, it follows that the cylindrical Lagrangian
\[ L_{cyl}^{a, \eps} (F) = \mathrm{graph} \left(dH_{cyl}^{F,a,\eps} \right)\]
satisfies $\partial_\infty L_{cyl}^{a, \eps} (F) \cap \stp_{\Sigma, \beta} = \emptyset$ for all sufficiently small $\eps$ and hence determines an object of $\cW( T^*M_{\RR}/M, \stp_{\Sigma, \beta})$.
Note that $L_{cyl}^{a, \eps}$ has canonical grading and orientation data by \cite{ganatra2024microlocal}, and we always work with that data when viewing $L_{cyl}^{a, \eps}(F)$ as an object of $\cW( T^*M_{\RR}/M, \stp_{\Sigma, \beta})$.
Further, all $L_{cyl}^{a, \eps}(F)$ for small $a, \eps$ are Hamiltonian isotopic through an isotopy avoiding $\stp_{\Sigma, \beta}$ at infinity and hence the object $L_{cyl}^{a, \eps}(F)$ in $\cW( T^*M_{\RR}/M, \stp_{\Sigma, \beta})$ is independent of $a, \eps$ up to quasi-isomorphism.
Therefore, we may at times drop $\eps$ and/or $a$ from the notation. 

Moreover, we can construct explicit cofinal wrappings of these objects by taking $a, \eps$ to zero. 

\begin{prop}
    For $0 < a < 1/ \max_\rho b_\rho$ and $\eps > 0$ sufficiently small as above, taking $a, \eps$ monotonically to zero gives a cofinal wrapping of $L^{a, \eps}_{cyl}(F)$.
\end{prop}
\begin{proof}
    Over $|\Sigma|$, we can check that the isotopy is positive at infinity by the same calculation as in \cite[Lemma 3.3]{hanlon2022aspects} using \eqref{eq:convolvederivative}. 
    Moreover, \eqref{eq:closetostop} shows that every point in $L_{cyl}^{a, \eps}(F)$ over $|\Sigma|$ approaches $\stp_{\Sigma, \beta}$ as $a, \eps$ go to zero. 
    Away from $|\Sigma|$, we have taken $F_a = u/a$ to be a linear function with monotonically increasing slope approaching $\infty$ as $a$ goes to zero. 
    We thus have a cofinal wrapping by \cite[Lemma 2.2]{ganatra2024sectorial}. 
\end{proof}

Note that the upper bound $a < 1/ \max_\rho b_\rho$ is not relevant for the claim that taking $a$ to $0$ gives a cofinal wrapping.
However, this bound prevents the Lagrangian from intersecting the stop as $a$ changes and thus (potentially) changing $L^{a,\eps}_{cyl}(F)$ as an object of $\cW( T^*M_{\RR}/M, \stp_{\Sigma, \beta})$. 
As such, we define $B := 1/ \max_{\rho \in \Sigma(1)} b_\rho$ and use this upper bound in later constructions.

Now, we will compute the subcategory of $\cW( T^*M_{\RR}/M, \stp_{\Sigma, \beta})$ consisting of the $L_{cyl}(F)$ by passing to an equivalent category that is better suited for a direct analysis of Floer theory.
Namely, we will work in the context of monomial admissibility in the sense of \cite{hanlon2019monodromy}. 
We will identify $T^* M_\RR/M$ with $M \otimes \CC^*$ by choosing a strictly convex function $\varphi \colon M_\RR \to \RR$ and taking the Legendre transform.
More explicitly, $M_\RR$ is the base of the $\mathrm{Log}$ map and we use $\varphi$ to make the standard symplectic form on $T^*M_\RR/M$ into the K\"{a}hler form 
\[ \omega = \sum_j dp_j \wedge dq_j = \sum_j d(\partial \varphi / \partial \log|z_j|) \wedge dq_j , \]
where $p_j$ are coordinates on $N_\RR$ and $q_j$ are the corresponding dual coordinates on $M_\RR/M$.
For $\rho \in \Sigma(1)$, we set 
\[ C_\rho^\delta = \{ u \in \mathrm{star}(\rho) \mid d(u, \sigma) \geq \delta \text{ for all } \sigma \in \Sigma \text{ such that } \rho \not \in \sigma \} \]
for any $\delta > 0$. 
The following is a slight generalization of \cite[Corollary 2.40]{hanlon2019monodromy}.

\begin{prop} \label{prop:mondivisionexists}
    For sufficiently small $\delta$, there exist $k_\rho \in \RR_{>0}$ and a strictly convex $\varphi$ such that for all $z = (p,q)$ outside of a compact set, the maximum
    \begin{equation} \label{eq:mondivfunction}
    \max_{\rho \in \Sigma(1)} |z^{b_\rho u_\rho}|^{k_\rho} 
    \end{equation}
    is achieved by a $\rho$ such that $p \in C_\rho^\delta$. 
\end{prop}
\begin{proof}
    The claim follows from the same argument as \cite[Corollary 2.40]{hanlon2019monodromy}, which does not use smoothness of the fan.
    This is perhaps more easily seen by noting that, in the language of \cite{hanlon2019monodromy}, our proposition is asserting the existence of a monomial division adapted to the stacky fan $(\Sigma, \beta)$. 
    In particular, the construction of a K\"{a}hler form adapted to a polytope from \cite{zhou2018lagrangian} holds when the normal fan is simplicial.
\end{proof}

Now, we proceed to construct monomially admissible Lagrangian sections in the sense of \cite{hanlon2019monodromy} and a category of these objects. 
For $F \in \mathrm{SF}(\Sigma, \beta)$, we define $H_{mon}^{F, \eps} = H^{F,g}$ where $g(r) = \eps$ and
\[ L_{mon}^\eps(F) = \mathrm{graph} \left(dH_{mon}^{F, \eps}\right). \]
By \cref{eq:convolvederivative}, we have that
\begin{equation} \label{eq:monsectioncondition}
    dH_{mon}^{F, \eps}(u_\rho) = F(u_\rho)
\end{equation}  
over $C_\rho^\delta$ for all $\eps < \delta$. 
As in the cylindrical case, we will at times assume $\eps$ to be sufficiently small, drop it from the notation, and simply write $L_{mon}(F)$. 

Then, following \cite[Section 3]{hanlon2019monodromy} and using \cref{prop:mondivisionexists} to prove that Floer theory is well-defined, we can construct a category $\Fs_{mon}(\Sigma, \beta)$ by localization.
The objects of the ordered category are pairs $(L_{mon}(F), a)$ with $a \in (0,B)$ and ``right-way'' morphisms are Floer cochain complexes.
That is, we define $\hom\Big( (L_{mon}(F_1) ,a_1), (L_{mon}(F_2), a_2)\Big)$ to be 
\[ \begin{cases} CF( L_{mon}(F_1 - F_K^{a_1}), L_{mon}(F_2 - F_K^{a_2}) ) & a_1 < a_2 \\
\Bbbk \cdot e_{(L_{mon}(F_1),a_1)}  & (L_{mon}(F_1) ,a_1) = (L_{mon}(F_2) ,a_2) \\ 
0 & \text{otherwise} \end{cases} 
\]
where $e_{(L_{mon}(F_1) ,a_1)}$ is a formal degree zero identity element.
The Floer chain complexes and higher operations are computed with respect to the canonical brane structures on these sections and (domain dependent) almost complex structures making the functions $z^{b_\rho u_\rho}$ holomorphic over the complement of a compact subset of $C_\rho^\delta$ for all $\rho \in \Sigma(1)$. 
Then, $\Fs_{mon}(\Sigma, \beta)$ is defined as the localization at natural continuation elements in $CF(L_{mon}(F-F_K^{a_1}), L_{mon}(F-F_K^{a_2}))$ when $a_1 < a_2$. 

To make computations in $\cW(T^*M_\RR/M, \stp_{\Sigma, \beta})$ and prove \cref{thm:hms}, we can work instead with $\Fs_{mon}(\Sigma, \beta)$ thanks to the following. 

\begin{prop} \label{prop:pwismon}
There is a pre-triangulated equivalence $\Fs_{mon}(\Sigma, \beta) \xrightarrow{\sim} \cW(T^*M_\RR/M, \stp_{\Sigma, \beta})$ taking $(L_{mon}(F), a)$ to $L_{cyl}^a(F)$ for every $F \in \mathrm{SF}(\Sigma, \beta)$.
\end{prop} 
\begin{proof}
    First, we note that there is a quasi-equivalence between $\Fs_{mon}(\Sigma,\beta)$ and the subcategory of $\cW(T^*M_\RR/M, \stp_{\Sigma, \beta})$ consisting of the Lagrangian sections $L_{cyl}(F)$.
    This follows from the same argument as in \cite[Theorem 3.5]{hanlon2022aspects} where the main idea is to use Lagrangians which are the graphs of $dH^{F,g}$ where $g$ is a monotonic function that is linear on a large interval after which it becomes constant.
    If the interval is large enough, the integrated maximum principle implies that the Floer theory is that of $L_{cyl}(F)$ in a Liouville subdomain while \eqref{eq:closetostop} precludes the introduction of additional intersections while interpolating to a $L_{mon}(F)$. The details are exactly the same as in \cite[Theorem 3.5]{hanlon2022aspects}.

    It remains to verify that $\cW(T^*M_\RR/M, \stp_{\Sigma, \beta})$ is generated by the $L_{cyl}(F)$. 
    To see this, for any $\theta \in M_\RR/M$, we consider the Lagrangian $L_\theta$ given by the outward conormal to a small ball around $\theta \in M_\RR/M$.\footnote{The Lagrangians $L_\theta$ are introduced and studied in a more general context in \cite[Section 5.2]{ganatra2024microlocal}.}
    By \cite[Proposition 5.22 and Remark 5.23]{ganatra2024microlocal} and stop removal, the Lagrangians $L_\theta$ generate $\cW(T^*M_\RR/M, \stp_{\Sigma, \beta})$.
    We conclude by showing that each $L_\theta$ is Hamiltonian isotopic to some $L_{cyl}(F)$ in the complement of $\stp_{\Sigma, \beta}$, which is a generalization of \cite[Lemma 3.14]{hanlon2022aspects}. 
    By definition, the Lagrangian $L_\theta$ can be written as the graph of the differential of a function $f \colon N_\RR \to \RR$ that satisfies
    \[ df(u_\rho) = \langle \theta, u_\rho \rangle + g(u) \]
    on $U^\eps_{\rho}$ for all $\rho \in \Sigma(1)$ where $g(u)$ is an arbitrarily small positive function. 
    Let $F \in \mathrm{SF}(\Sigma, \beta)$ be such that 
    \begin{equation}
    \label{eq:BTcansupport} F(u_\rho) = \frac{\lfloor \langle \theta, b_\rho u_\rho \rangle \rfloor + 1}{b_\rho}.
    \end{equation}
    For sufficiently small $g(u)$, it then follows from \eqref{eq:closetostop} that $b_\rho df_t(u_\rho) \not \in \ZZ$ where $f_t = (1-t)f + tH_{cyl}^F$. 
    That is, we have produced a Hamiltonian isotopy from $L_\theta$ to $L_{cyl}(F)$ in the complement of $\stp_{\Sigma, \beta}$.
\end{proof}

Note that the proof of generation by tropical Lagrangian sections in \cite{hanlon2022aspects} is more involved and uses an inductive argument by restriction to the toric boundary divisors. 
There should be an analogous approach for stacks as well, but we take a more direct path in \cref{prop:pwismon} using natural generators of the derived category that have received recent interest. 
Namely, the support functions defined by \eqref{eq:BTcansupport} correspond to the line bundles that are summands of the pushforward of the canonical bundle under the toric Frobenius morphism.
To obtain mirror Lagrangians to the summands of the pushforward of the structure sheaf under toric Frobenius, which have been dubbed the Bondal--Thomsen or Thomsen collection \cite{thomsen2000frobenius,Bondal, hanlon2024resolutions,ballard2025king}, one takes the inward conormal to a small ball around $\theta$ rather than the outward conormal $L_\theta$. 
See \cref{subsec:BTcollection}.

\begin{rem}
    It should also be possible to avoid passing to $\Fs_{mon}(\Sigma, \beta)$ altogether and carry out the Morse theory arguments that follow directly in the cylindrical setting. 
    While such an approach would be more direct when working from scratch, we have opted to instead rely on work already in the literature to the extent that is possible.
    The more direct approach would be interesting to pursue in another venue and could also allow one to work with more general fans.
\end{rem}

With \cref{prop:pwismon} in hand, we can now carry out the strategy in \cite{abouzaid2009morse} (also outlined in \cite[Section 3.5]{hanlon2019monodromy}) and pass to Morse theory.

First, we need to carry out the Morse theory analogue of the construction of $\Fs_{mon}(\Sigma, \beta)$. 
We start with an ordered category whose objects are pairs $(H_{mon}^F, a)$ with $a \in (0,B)$ and we define $\hom\Big( (H_{mon}^{F_1} ,a_1), (H_{mon}^{F_2}, a_2)\Big)$ to be
\[
\begin{cases} \bigoplus_{m \in M} CM\left( H_{mon}^{F_2} - H_{mon}^{F_1} - (H_{mon}^{F_K^{a_2}} - H_{mon}^{F_K^{a_1}}) + m \right) & a_1 < a_2 \\
\Bbbk \cdot e_{(H_{mon}^{F_1},a_1)}  & (H_{mon}^{F_1} ,a_1) = (H_{mon}^{F_2} ,a_2) \\ 
0 & \text{otherwise} \end{cases} 
\]
where $e_{(H_{mon}^{F_1} ,a_1)}$ is a formal degree zero identity element and $CM$ denotes the Morse complex with respect to a Riemannian metric on $N_\RR$ inducing a complex structure on $T^*M_\RR/M$ satisfying the property required in the construction of $\Fs_{mon}(\Sigma, \beta)$ above, which is well-defined by a slight generalization of \cite[Proposition 3.28]{hanlon2019monodromy} with an essentially identical proof. 
Higher $A_\infty$ operations are defined by counts of Morse trees.
Then, we define $\Ms(\Sigma, \beta)$ to be the localization at any chain-level representative of a generator of $HM\left(- H_{mon}^{F_K^{a_2}} + H_{mon}^{F_K^{a_1}} \right) \cong \Bbbk$. 
The concluding result of this section reduces the computation of $\Fs_{mon}(\Sigma, \beta)$ to $\Ms(\Sigma,\beta)$ by applying \cite[Theorem 2.3]{fukaya1997zero}.

\begin{prop} \label{prop:monismorse}
    The natural identifications between objects and morphisms induce a quasi-equivalence
    \[ \Fs_{mon}(\Sigma, \beta) \cong \Ms(\Sigma, \beta) \]
\end{prop}
\begin{proof}
    The proof is identical to that of \cite[Corollary 3.31]{hanlon2019monodromy}, which is itself modeled on \cite[Theorem 5.8]{abouzaid2009morse}.
    The only modification is to replace every instance of the monomial $z^{u_\rho}$ for some $\rho \in \Sigma(1)$ with $z^{b_\rho u_\rho}$.
\end{proof}

\subsection{Proof of \cref{thm:hms}} 
We will now identify the structures developed in \cref{subsec:dgline} and \cref{subsec:troplag} to prove \cref{thm:hms}. 
Namely, \cref{thm:hms} follows immediately from \cref{prop:dbislinebundles}, \cref{prop:pwismon}, \cref{prop:monismorse} and the following.

\begin{prop} \label{prop:morseiscech}
There is a quasi-equivalence 
\[ \Ms(\Sigma, \beta) \xrightarrow{\sim} \check{C}(\X) \]
taking $(H^F_{mon}, a)$ to $\Os_{\X}(F)$ for all $F \in \mathrm{SF}(\Sigma, \beta)$ and $a \in (0,B)$.
\end{prop}

Since we have described $\Ms(\Sigma, \beta)$ and $\check{C}(\X)$ entirely in terms of support functions, the proof of \cref{prop:morseiscech} is essentially identical to that of the non-stacky case in \cite{abouzaid2009morse} and its reproduction in \cite[Section 3.5]{hanlon2019monodromy} with some minor adjustments allowing for $\Sigma$ to not be complete.
Nevertheless, the remainder is dedicated to outlining the proof of \cref{prop:morseiscech} in our notation closely following the aforementioned references.

In order to proceed in the most efficient manner and take advantage of the literature, it will be useful to pass to Morse theory on a compact set (see \cite[Remark 3.33]{hanlon2019monodromy} though we note that we expect the argument to be considerably more complicated outside of the semiprojective and simplicial setting).
To that end, if $\Sigma$ is complete, we can assume without loss of generality that all $H^F_{mon}$ that appear in objects of $\Ms(\Sigma, \beta)$ satisfy \eqref{eq:monsectioncondition} on $C_\rho^\delta$ outside the interior of a compact set $Y \subset N_\RR$ whose boundary is a large level set of the function on $N_\RR$ induced by \eqref{eq:mondivfunction}.
In other words, when $\Sigma$ is complete, $Y = \nabla \varphi (Q)$ where $Q \subset M_\RR$ is given by
\[ k_\rho \langle m, b_\rho u_\rho \rangle \leq y \]
for all $\rho$ and some large constant $y$, and \cref{prop:mondivisionexists} implies that $Q$ is a moment polytope of $X_\Sigma$. 
If $\Sigma$ is not complete, we simply take $Y = \nabla \varphi(Q)$ for some large moment polytope $Q$ of a simplicial completion of $\Sigma$ and additionally assume that all $H^F_{mon}$ are linear away from a neighborhood of the images of the facets of $Q$ that are dual to rays of $\Sigma$.
In general, the result is that all critical points and Morse flow trees involved in computing morphisms in $\Ms(\Sigma, \beta)$ lie in the interior of $Y$ as in \cite[Proposition 3.28]{hanlon2019monodromy}.  

Now, we pass from Morse theory on $Y$ to simplicial chains on $Y$. 
We let $Y_b$ be the triangulation obtained as the image under $\nabla \varphi$ of the barycentric subdivision $Q_b$ of $Q$. 
Given a function $h$ on $Y$ such that $h(u_\rho)$ is constant on $Y_\rho := \nabla \varphi(Q_\rho)$ for every facet $Q_\rho$ with primitive normal vector $u_\rho$, we define $\partial_h^+ Y_b$ to be the closure of the maximal simplices $\tau$ of $Y_b$ where $\partial \tau \cap \partial Y \subset Y_\rho$ and $h(u_\rho)$ is positive.
To define the simplicial category, we again start with an ordered category whose objects are pairs $(F, a)$ with $F \in \mathrm{SF}(\Sigma, \beta)$ and $a \in (0,B)$ and define
\[ \hom\Big( (F_1 ,a_1), (F_2, a_2)\Big) = 
\begin{cases} \bigoplus_{m \in M} C\left( Y_b, \partial^+_{F_2 - F_1 -(F_K^{a_2} - F_K^{a_1}) +m} Y_b \right) & a_1 < a_2 \\
\Bbbk \cdot e_{(F_1,a_1)}  & (F_1 ,a_1) = (F_2 ,a_2) \\ 
0 & \text{otherwise} \end{cases} 
\]
where $e_{(F_1 ,a_1)}$ is a formal degree zero identity element and $C$ denotes the relative simplicial cochains on a simplicial pair. 
This is a dg-category with product induced by the cup product and the inclusion 
\[ \partial^+_{F_3 - F_1 -(F_K^{a_3} - F_K^{a_1}) +m} Y_b \subset \partial^+_{F_2 - F_1 -(F_K^{a_2} - F_K^{a_1}) +m} Y_b \cup \partial^+_{F_3 - F_2 -(F_K^{a_3} - F_K^{a_2}) +m} Y_b \]
as defined more generally in \cite[Definition E.2]{abouzaid2009morse}.  
We define $\mathrm{Simp}_{mon}(Y)$ to be the localization at any chain-level representative of a generator of $H\left(Y_b, \partial^+_{-F_K^{a_2} + F_K^{a_1}}\right) \cong \Bbbk$ lying in the degree $0 \in M$ piece of $\hom((F, a_1), (F, a_2))$ with $a_2 > a_1$.
We then have the following categorical version of the usual identification, used by \citeauthor{abouzaid2009morse}, of the Morse cohomology of boundary convex functions with relative cohomology.

\begin{prop} \label{prop:morseissimp}
There is a quasi-equivalence 
\[ \Ms(\Sigma, \beta) \xrightarrow{\sim} \mathrm{Simp}_{mon}(Y) \]
taking $(H^F_{mon}, a)$ to $(F, a)$ for all $F \in \mathrm{SF}(\Sigma, \beta)$ and $a \in (0,B)$.
\end{prop}
\begin{proof}
    The claim follows from \cite[Corollary 4.33]{abouzaid2009morse} as described in the last paragraph of the proof of Proposition 6.7 of \cite{abouzaid2009morse} (see also \cite[Proposition 3.38]{hanlon2019monodromy}).
\end{proof}

Before comparing the simplicial category to the \v{C}ech category of line bundles, it is helpful to simplify the former slightly by removing the push-offs by $F_K$ from the stop that were needed for wrapping in the Fukaya and Morse categories and working more directly on $Q$.
Since we need to keep the large wrapping if $\Sigma$ is not complete, we define a new family of piecewise-linear functions $F_0^a$ such that $F_0^a$ is $0$ on $|\Sigma|$ and $F_0^a$ is positive away from $|\Sigma|$ and tends monotonically to $+\infty$ as $a$ tends to $0$.
For example, we can take $F_0^a$ to be a support function for $- \sum_{\rho \in \Sigma_Q(1) \setminus \Sigma(1)} \frac{1}{a} D_\rho$ where $\Sigma_Q$ is the normal fan to $Q$.
Now, we define $\mathrm{Simp}_{mon}(Q)$ exactly as $\mathrm{Simp}_{mon}(Y)$ except replacing all $F_K^a$ with $F_0^a$, $Y_b$ with $Q_b$, and defining $\partial^+_h Q_b$ to be the closure of the maximal simplices $\tau$ of $Q_b$ where $\partial \tau \cap \partial Q \subset Q_\rho$ and $h\circ \nabla \varphi(u_\rho)$ is positive.
Note that if $\Sigma$ is complete, $F_0$ is identically zero and $\mathrm{Simp}_{mon}(Q)$ can be defined without the parameter $a$ and any localizations by simply setting the hom spaces to be direct sums of $C\left( Q_b, \partial^+_{F_2 - F_1 +m} Q_b \right)$.
The proof of the following is identical to that of \cite[Proposition 3.39]{hanlon2019monodromy}.

\begin{prop} \label{prop:simpyissimpq}
There is a quasi-equivalence
\[\mathrm{Simp}_{mon}(Y) \simeq \mathrm{Simp}_{mon}(Q) \]
which is the identity on objects.
\end{prop}

The final step of identifying $\mathrm{Simp}_{mon}(Q)$ with $\check{C}(\X)$ is almost entirely topological: one identifies simplicial cochains with \v{C}ech cochains and observes that the toric \v{C}ech cover on $\X$ indexes a topological \v{C}ech cover of $Q$ after dualizing the index set from $\Sigma_Q$.

\begin{prop} \label{prop:simpqiscech} 
There is a quasi-equivalence
\[ \mathrm{Simp}_{mon}(Q) \xrightarrow{\sim} \check{C}(\X)  \]
taking $(F,a)$ to $\Os_{\X}(F)$ for all $F \in \mathrm{SF}(\Sigma, \beta)$ and $a \in (0, B)$.
\end{prop}
\begin{proof} In the complete case, the claim follows from Proposition 6.9 and the claim in the proof of Proposition 6.7 of \cite{abouzaid2009morse} (see also \cite[Proposition 3.40]{hanlon2019monodromy}).

If $\Sigma$ is not complete, let $\X_{Q}$ be the smooth toric Deligne--Mumford stack associated to $\Sigma_Q$ with stacky fan structure extending that of $\Sigma$ by setting $b_\rho = 1$ for all $\rho \in \Sigma_Q(1) \setminus \Sigma(1)$.
We choose $F_0^a$ to be the support function of $- \sum_{\rho \in \Sigma_Q(1) \setminus \Sigma(1)} \frac{1}{a} D_\rho$ as suggested above.
In that case, Abouzaid's results (and considering the $M$ grading) identify $\mathrm{Simp}_{mon}(Q)$ with the localization of $\check{C}(\X_Q)$ at the natural transformation given by tensoring by a defining section of the divisor $D := \sum_{\rho \in \Sigma_Q(1) \setminus \Sigma(1)} D_\rho$ (cf. \cite[Theorems 4.5 and 4.9]{hanlon2019monodromy}). 
As observed in \cite{seidel2008subalgebras}\footnote{In Seidel's paper, this observation is made in the context of varieties, but it applies equally to toric Deligne--Mumford stacks by, for instance, thinking of the objects as equivariant sheaves on affine space.}, this localization is $\check{C}(\X_Q \setminus D)$.
Moreover, by construction, we have $\X = \X_Q \setminus D$. 
\end{proof}

We end by noting that chaining together \cref{prop:morseissimp}, \cref{prop:simpyissimpq}, and \cref{prop:simpqiscech} gives \cref{prop:morseiscech}, which was the goal of this section.  Together with \cref{prop:dbislinebundles,prop:pwismon,prop:monismorse}, this proves \cref{thm:hms}.

\printbibliography 
\Addresses

\end{document}